\documentclass[12pt,reqno]{amsart}
\usepackage{amsmath,amssymb,amsthm,bbm,enumerate, array,fullpage}

\newtheorem{theorem}{Theorem}[section]
\newtheorem{corollary}[theorem]{Corollary}
\newtheorem{lemma}[theorem]{Lemma}
\newtheorem{proposition}[theorem]{Proposition}
\newtheorem{definition}[theorem]{Definition}

\theoremstyle{remark}
\newtheorem{remark}[theorem]{Remark}

\newcommand{\s}{{\mathbf{s}}}
\newcommand{\ii}{{\iota}} 
\newcommand{\C}{{\mathbb{C}}}
\newcommand{\A}{{\mathbb{A}}}

\newcommand{\ord}{{\text{ord}}}
\def\Z{\mathbb{Z}}
\def\o{\mathfrak{o}} % used for the ring of integers
\def\a{\alpha}

\newcommand\xqed[1]{%
  \leavevmode\unskip\penalty9999 \hbox{}\nobreak\hfill
  \quad\hbox{#1}}
\newcommand\example{\xqed{$\triangle$}}

\makeatletter
\let\@fnsymbol\@arabic
\makeatother

\begin{document}

\title{Whittaker Coefficients of Metaplectic Eisenstein Series}
\author{Benjamin Brubaker} 
\author{Solomon Friedberg}
\address{School of Mathematics, University of Minnesota, 206 Church Street SE, Minneapolis, MN 55455}
\email{brubaker@math.umn.edu}
\address{Department of Mathematics, Boston College, Chestnut Hill, MA 02467-3806}
\email{solomon.friedberg@bc.edu}
 \thanks{This work was supported by NSF grant DMS-1258675 (Brubaker) and NSA grant H98230-13-1-0246 (Friedberg).}
\begin{abstract} We study Whittaker coefficients for maximal parabolic 
Eisenstein series on metaplectic covers of split reductive groups.  
By the theory of Eisenstein series these coefficients have meromorphic continuation and functional equation.
However they are not Eulerian and the standard methods to compute them in the reductive case do not
apply to covers.  For ``cominuscule'' maximal parabolics, we give an explicit description of the coefficients as
Dirichlet series whose arithmetic content is expressed in an exponential sum. The exponential sum
is then shown to satisfy a twisted multiplicativity, reducing its determination to prime
power contributions. These, in turn, are connected to Lusztig data for canonical bases on the dual group using a result of Kamnitzer. 
The exponential sum at prime powers is shown to simplify for generic Lusztig data. At the remaining degenerate
cases, the exponential sum seems best expressed in terms of Gauss sums depending on string data for canonical
bases, as shown in a detailed example in $GL_4$. Thus we demonstrate that the arithmetic part of metaplectic 
Whittaker coefficients is intimately connected to the relations between these two expressions for canonical
bases. 
\end{abstract}
\maketitle

\section{Introduction}

The computation of the Whittaker coefficients of Eisenstein series on reductive groups has had
far-reaching consequences for
the study of automorphic forms. As two prominent examples, Langlands' computations of
the constant terms of Eisenstein series highlighted the role of the dual group in automorphic forms
and inspired his functoriality conjectures, while the Casselman-Shalika formula for the
coefficients of Borel Eisenstein series at unramified places has been a staple of the Langlands-Shahidi and
Rankin-Selberg methods used to obtain
analytic properties of $L$-functions. In this paper, we provide a general method for explicitly 
computing the Whittaker coefficients of parabolic
Eisenstein series on metaplectic covers of reductive groups. 

In order to explain the novel features of this more general context, we briefly review
the theory for reductive algebraic groups.
Let $G$ be a split connected reductive algebraic group defined over a number field $F$
and let $P$ be a standard maximal parabolic subgroup of $G$ with Levi factorization $P=MN$.   If $\pi$ is an automorphic
representation of $M(\A_F)$ then one may form the corresponding Eisenstein series on $G(\A_F)$.  The constant term
of this Eisenstein series is an Euler product and may be expressed in terms of the Langlands
$L$-functions $L(s,\pi,r_i)$ where $\oplus_{i=1}^m r_i$ is a decomposition of the adjoint representation of $^L\!G$ on
the complexified Lie algebra $^L{\mathfrak n}$.  The Eisenstein series has continuation in $s$ and
satisfies a functional equation and hence so
does this constant term.
When $\pi$ is generic the further study of the Whittaker coefficients
of the Eisenstein series may be used to obtain the continuation of each of these $L$-functions; this is the
Langlands-Shahidi method.  In doing so one uses in a critical way the uniqueness of the Whittaker
functional to see that the Whittaker coefficients are Eulerian.  Then the Casselman-Shalika formula may
be applied at almost all places to evaluate the local contributions to the Whittaker coefficients in
terms of local Langlands $L$-functions.  By an induction in stages argument, 
one can similarly study Eisenstein series for non-maximal parabolics.

When $G$ is a split semisimple simply connected algebraic group over $F$
and $F$ contains the $n$-th roots of unity,
Matsumoto~{\cite{matsumoto}} defined an $n$-fold ``metaplectic" covering group $\widetilde{G}_{\A_F}$ of $G(\A_F)$
that splits over $G(F)$.  The group law makes use of the local Hilbert symbol in the completions
of $F$.  Following Brylinski and Deligne \cite{brylinski-deligne}, one may also define metaplectic covers of 
reductive and non-simply connected groups, sometimes at the expense of
requiring that $F$ contain more roots of unity.

In view of the importance of Whittaker coefficients for Eisenstein series on reductive groups,  it is 
natural to ask whether similar explicit formulas exist for metaplectic covers.  
Already for the $n$-fold cover of $GL_2$ one sees that the situation in the metaplectic case  is 
fundamentally different.
The constant term of the Eisenstein series in that case is a quotient of Hecke $L$-functions, but in fact 
Kubota showed in \cite{kubota} that each 
non-degenerate Whittaker coefficient is a 
Dirichlet series that is an infinite sum of $n$-th order Gauss sums.  
When $n > 2$, these Whittaker coefficients are not
Eulerian and the space of Whittaker functionals is of dimension
greater than 1. However, the metaplectic Eisenstein series still has analytic continuation and functional equation
and the Whittaker coefficients do as well.  Though these 
Dirichlet series built out of Gauss sums are not Langlands $L$-functions, 
they share all of the same analytic properties while possessing a more complicated arithmetic structure.

This structure was further explored for covers of higher rank groups in the case of Eisenstein series
induced from the Borel subgroup by the authors and Daniel Bump \cite{bbf-wmd2, eisenxtal, bbf-book}. 
Working in the $S$-integer formalism (for a finite set of places $S$, to be described in the next section), we gave a complete answer for the 
Whittaker coefficients of Borel Eisenstein series on an $n$-fold cover of $SL_{r+1}$ in \cite{eisenxtal}. The results are expressed
using $n$-th order Gauss sums whose defining data comes from a surprising source -- Kashiwara's crystal bases. 

More precisely, the Whittaker coefficients are indexed by an $r$-tuple of nonzero $S$-integers~$\boldsymbol{k}$
which specifies the character of the maximal unipotent at each simple root. To each such $\boldsymbol{k}$
the associated Whittaker coefficient contains a multiple Dirichlet series roughly of the form
$$\sum_{C_1,\dots,C_r\neq0} H(C_1,\dots,C_r; \boldsymbol{k}) |C_1|^{-2s_1}\dots|C_r|^{-2s_r}.$$
Here the $s_i$ are complex variables and the series initially converges when their real parts are sufficiently large.
The sum is over $S$-integers $C_i$ modulo units, and $|C_i|$ denotes the norm of 
$C_i$. As noted above, these series are not typically Eulerian but rather
are {\it twisted Euler products} -- the coefficients $H$ are not multiplicative,
but combine by means of $n$-th power residue symbols -- so that ultimately the entire
series is determined by coefficients of the form $H(p^{\ell_1}, \ldots, p^{\ell_r}; p^{m_1}, \ldots, p^{m_r}) =: H(p^{\boldsymbol{\ell}}; p^{\boldsymbol{m}})$ 
for a prime $p$ in the $S$-integers.  It is these latter coefficients which may be described using crystal bases. For work
on Borel Eisenstein series on covers of other classical groups, see 
\cite{bbcg, friedberg-zhang, friedberg-zhang2}.

In this work, we study the Whittaker coefficients
of metaplectic Eisenstein series induced from a cover of a maximal parabolic
subgroup $P$.  Such series are constructed from metaplectic automorphic representations of lower rank groups.
While a number of our methods work in complete generality, our sharpest results are obtained in the case when 
the maximal parabolic $P$ is ``cominuscule'' -- the omitted simple root appears with multiplicity one in the highest root. 
Equivalently, $P$ is cominuscule if its unipotent radical is abelian. In
particular, all maximal parabolics for groups of type $A$ satisfy this condition, and the only groups without such
parabolics are those of types $E_8, F_4,$ or $G_2$. See Section~5 for further discussion. 
In this case, we establish a general formula for these coefficients in terms
of exponential sums and the Whittaker coefficients of the inducing data (Theorem~\ref{whittakeratlast}).  These coefficients
inherit analytic continuation and functional equations from the Eisenstein series, thus giving new examples of Dirichlet series with
these analytic properties. The exponential sums appearing in the Whittaker coefficient are then connected to
the representation theory of the dual group of~$G$. 

The metaplectic cover presents obstacles to the use of standard methods for 
studying parabolic Eisenstein series on reductive groups.  Indeed,
local Whittaker models on covers are not unique (and hence the resulting global Whittaker coefficients are
not Eulerian).  Moreover there is a complicated determination of the metaplectic cocycle which
describes the group law on the cover, one that is not amenable to calculations except in rank one.

To explain our resolution of these issues, recall that the Eisenstein series is defined (working over the ring $\o_S$ of $S$-integers)
as an average over elements $\gamma \in P(\o_S) \backslash G(\o_S)$. We show that representatives $\gamma$ may be
parametrized by embeddings of rank one matrices at the roots in the unipotent radical of $P$. For such 
matrices, it is possible to systematically handle the cocycle that arises and so this parametrization
is well-suited for covers.  In it, the
order of roots appearing in the product is dictated by a reduced decomposition for $w^P$, the Weyl
group element such that $w_0 = w_M w^P$ where $w_M$ is the long element in the Weyl group of
the Levi factor $M$ of $P$. If $P$ is assumed cominuscule, then representatives $\gamma$ may be
determined using only the bottom rows of the embedded $SL_2(\o_S)$ matrices, called $(c_i, d_i)$.
Furthermore, we may parametrize double cosets in the flag variety needed in the computation by all non-zero
$d_i$ and $c_i$'s modulo prescribed products of the $d_i$.

Though the decomposition depends on a reduced decomposition for $w^P$, if $P$ is cominuscule
such decompositions are unique up to commuting simple reflections (see Remark~\ref{canremark}).
Moreover in Section~\ref{canbases}, to any fixed prime $p$ in $\o_S$, the valuations of the $d_i$
are shown to match the lengths of edges in the one-skeleton of Kamnitzer's MV polytopes \cite{kamnitzer}
corresponding to $w^P$. Thus, by results in \cite{kamnitzer}, the valuations of the $d_i$ are the $\boldsymbol{i}^P$-Lusztig
data for canonical basis elements in the dual group $G^\vee$, where $\boldsymbol{i}^P$ is the reduced word for the decomposition of $w^P$.
Thus we parametrize contributions to the Whittaker coefficient in terms of Lusztig data.
 Two qualifications are in order. First, Kamnitzer works
in the affine Grassmannian, but as the combinatorics of valuations is independent of the base field, we
can draw a formal analogy. Second, the results of \cite{kamnitzer} are presented for the Borel case,
but one can do a relative version of \cite{kamnitzer}.
A similar connection was made in
the local setting by McNamara \cite{mcnamara-duke}, who was able to connect a decomposition of the
unipotent radical of the standard Borel $B$ directly to MV cycles, initially defined by Mirkovi\'c and Vilonen \cite{mirkovic-vilonen} 
in the context of the geometric Satake correspondence.

There is another approach to computing maximal parabolic
Eisenstein series on the metaplectic group, using Pl\"ucker
coordinates on the quotient space $P\backslash G$.  Indeed, this was the approach used in
prior work of Bump, Hoffstein and the authors
\cite{bfh-nonvanishing, bump-hoffstein, bbfh}.  To use
such a parametrization on the metaplectic group, 
one must compute the Kubota symbol, the root of unity arising from the metaplectic two-cocycle. 
But computing the Kubota symbol from Pl\"ucker coordinates is difficult, and 
the formulas even in low rank cases are extremely complicated (compare Proskurin \cite{proskurin}
and \cite{bbfh}, Eqn.\ (18)). 
Our approach here avoids this by systematic use of
the factorization into rank one subgroups.

There is also another description of Whittaker coefficients of metaplectic Borel
Eisenstein series due to Chinta, Gunnells, and Offen~\cite{chinta-gunnells-method, chinta-offen}. Its local components are constructed as alternators built from a metaplectic version of the usual Weyl group action,
generalizing the Casselman-Shalika formula combined with the Weyl character formula in the case of the trivial cover $n=1$. 
This description over a local field holds for all covers of unramified reductive groups \cite{mcnamara-metcs}. However, this 
approach does not provide a way to calculate the individual coefficients
$H(p^{\boldsymbol{\ell}}; p^{\boldsymbol{m}})$ (for $n=1$ this is equivalent to passing from the Weyl
character formula to the Freudenthal multiplicity formula) and seems restricted to the Borel case.

We expect several applications from the results established here. First, the Dirichlet series we obtain have analytic continuation and functional
equation, and these properties may then be used to study the distribution of their coefficients. To give an example, one may induce a pair of automorphic forms on the double cover of $GL_2$ 
to $GL_4$ using the detailed example presented here. A Whittaker coefficient of the resulting Eisenstein series yields a Dirichlet series involving an exponential sum, analyzed below, and
a product of Fourier coefficients of the inducing data.  These Fourier coefficients in turn
are related to central values of quadratically-twisted $GL_2$ $L$-functions by work of Waldspurger.  Thus our work may be used to study the distribution of pairs of such central values
as one varies the quadratic twist. 

We emphasize that our results are not limited to maximal parabolic subgroups. They apply equally well to any metaplectic, parabolic Eisenstein series where the parabolic arises from a sequence of nested maximal cominuscule parabolics simply by inducing in stages. This was carried out for one particular sequence in type $A$ in \cite{eisenxtal}. Using the results of this paper to induce in stages from the Borel subgroup, we obtain generalizations of many of the results in \cite{eisenxtal} for covers of groups not of type $E_8, F_4,$ or $G_2$. (As noted above, these three exceptional types are the only cases with no cominuscule parabolics.)  A further analysis of a combinatorial nature is carried out for certain types
in \cite{friedberg-zhang, friedberg-zhang2}.

Lastly, even in the linear algebraic group (i.e.\ the 1-fold cover)
case, this computation is new. As noted above, previous applications of parabolic Eisenstein series 
for reductive groups avoided evaluating an exponential sum by invoking the fact that automorphic forms are unramified principal
series at almost all places and hence evaluated by the Casselman-Shalika formula. When computing Whittaker coefficients
in the case of unique Whittaker models, it may still be preferable to use the more familiar method exploiting intertwining operators 
on unramified principle series. But in contrast, the methods developed here apply to other interesting unipotent periods arising in the character expansions of automorphic forms,
expressed as sums over unipotent orbits (see for example \cite{miller-sahi}) as well as certain classes of unipotent period integrals motivated by mathematical physics \cite{green-miller-russo-vanhove, kleinschmidt-et-al}.

This paper is organized as follows.  In Section 2 we set the notation and review information
about the metaplectic groups (both local and global).  Section 3 explains how to construct 
a metaplectic Eisenstein series on a split reductive group $G$ by inducing parabolically.  This is not as trivial as it sounds, 
since if the Levi subgroup $M$ of a
maximal parabolic subgroup $P$ is of the form $M_1\times M_2$ where the $M_i$ are linear algebraic groups,
then it does not follow that the inverse image of $P$ in the covering group is a direct product.
(See for example Takeda \cite{takeda}.) 
The computation of the Whittaker coefficients of the Eisenstein series is given in Section 4, using
a decomposition theorem from Section 5. This order of presentation allows us to complete the discussion
of the Whittaker integral before turning to combinatorial questions that occupy the remainder of the paper.
Section 5 begins with results for arbitrary reductive groups and their quotients $P\backslash G$. Gradually we place
additional assumptions on the parabolic $P$, arriving at cominuscule parabolics in order to achieve a particularly nice parametrization of double cosets that arise in computing the Eisenstein series.
For such parabolics, the Whittaker coefficient may be expressed as a Dirichlet series involving an exponential piece
and the Whittaker coefficients of the inducing data.  

In Section 6 the exponential sum appearing in the Whittaker function is examined and a 
twisted multiplicativity theorem is established, reducing these contributions for general argument to the prime
power case.  This result thus subsumes and generalizes twisted multiplicativity
for specific Lie types as in \cite{eisenxtal, friedberg-zhang}.  
In Section 7 we briefly illustrate these ideas by working
out one case in detail, that of the Eisenstein series on a cover of $GL_4$ attached to $P$
with Levi factor $GL_2\times GL_2$.  
Section 8 concerns canonical bases
and the exponential sum.  Here we explain the link between the computations above, the Lusztig data,
and MV cycles and polytopes.  The key point is that Lusztig data appears naturally in the double coset
parametrization presented in Section 5, while the value of the exponential sum that appears
in the Whittaker coefficient computation is best understood via the polytope built from Kashiwara's string data, 
as studied by Berenstein-Zelevinsky and Littelmann. It is an important feature that both objects appear 
naturally at different points in the computation. The problem of giving uniform expressions
for the exponential sum at powers of a prime $p$, valid for all groups and all reduced decompositions
for long elements of the Weyl group, is thus closely connected to the appearance of difficult piecewise linear maps
in the bijection between Lusztig and string data. 

Finally, Sections~\ref{genericevalsection} and \ref{vanishing} concern the evaluation of the exponential sum $H(p^{\boldsymbol{\ell}}; p^{\boldsymbol{m}})$ arising in the Whittaker coefficient at powers of a fixed prime $p$.  In Section 9 we prove
two results applicable for all cominuscule parabolics. The first (Proposition~\ref{phifunprop}) states that the exponential sum has a very regular
evaluation for prime powers $p^{\boldsymbol{\ell}}$ for $\boldsymbol{\ell}$ away from certain bounding hyperplanes depending on $\boldsymbol{m}$. These hyperplanes conjecturally cut out the associated $\boldsymbol{i}$-Lusztig data for the representation of highest weight $\boldsymbol{m}$ of the dual group (see Remark~\ref{canbasesconnection}). The second result (Proposition~\ref{genericvanprop}) states that the exponential sum vanishes for $\boldsymbol{\ell}$ outside the set of Lusztig data bounded in terms of the highest weight $\boldsymbol{m} + \rho$ of the dual group, where $\rho$ is the Weyl vector of the dual
group, with the possible exception of a certain set of hyperplanes. 

We would like to be able to show that for fixed $\boldsymbol{m}$, the exponential sum $H(p^{\boldsymbol{\ell}}; p^{\boldsymbol{m}})$ vanishes at {\it all} integer lattice points $\boldsymbol{\ell}$ outside the set of $\boldsymbol{i}$-Lusztig data for the corresponding highest weight representation. This would allow us to conclude that any prime-powered Whittaker coefficient is supported at only finitely many $\boldsymbol{\ell}$. In \cite{eisenxtal}, the authors and Bump were able to show this in the very special case of the nicest (i.e.~``Gelfand-Tsetlin'') word $\boldsymbol{i}$ in type $A$.  However, in Section~\ref{vanishing} we revisit the $GL_4$ example initially presented in Section 7 and demonstrate that even in the cominuscule case, such a simple vanishing statement is not possible in general.  First we explain that the exponential sum at prime powers 
indeed has support at infinitely many lattice points outside the set of Lusztig data for the corresponding highest weight representation, under the equivalence 
of the prime power coordinates with Lusztig data. Nevertheless, the total contribution to the Dirichlet series from all basis vectors 
in a given weight space does vanish for Lusztig data outside the string data polytope (Proposition~\ref{cancel}).
It seems reasonable to expect a similar phenomenon to hold for all cominuscule parabolic, metaplectic Eisenstein series.
In Section~\ref{evaluation}, we demonstrate the complexity of the evaluation for Lusztig data lying outside the hyperplanes
determined by $\boldsymbol{m}$, but inside that of $\boldsymbol{m}+\rho$, i.e.~the lattice points not covered by Proposition~\ref{phifunprop}. For many of these points, we give the evaluation
in terms of the corresponding string data (Theorem~\ref{stringmatching}). Thus the evaluation of the Whittaker coefficients
depends, in an essential way, on the relations between the two most important parametrizations of canonical bases.

We thank Arkady Berenstein, Daniel Bump, Jeffrey Hoffstein, Peter McNamara, and Lei Zhang for helpful conversations.

\section{Metaplectic covers\label{covers}}

Throughout this paper, let $n\geq1$ be a fixed integer and let
$F$ be a totally complex number field containing a full set of $2n$-th roots of unity.
(The arguments below would work for function fields as well with minor modifications.)
For each place $v$ of $F$, let $F_v$ denote the corresponding completion at $v$.
When $v$ is non-archimedean, let $\mathfrak{o}_v$ denote the ring of integers of $F_v$ and $\varpi_v$ be a 
local uniformizer. Assume that $G$ is a split reductive group over $F$. This ensures that, at each place, $G(F_v)$
arises by base extension from a smooth reductive group scheme $\mathbf{G}$ over $\mathfrak{o}_v$.
We first discuss the $n$-fold metaplectic covers attached to $G$, both local and global.

\subsection{Local fields\label{localfieldssect}}

The local metaplectic group 
$\widetilde{G}_v$ is a central extension of the group $G(F_v)$ by the group of $n$-th roots of unity $\mu_n$:
$$ 1 \longrightarrow \mu_n \longrightarrow \widetilde{G}_v \longrightarrow
  G(F_v)\longrightarrow 1. $$
Central extensions of semisimple, simply connected, algebraic groups $G(k)$ over an infinite field $k$ by $K_2(k)$ were classified by Brylinski and Deligne \cite{brylinski-deligne} as follows. Let $T$
be a maximal split torus of $G$, and let 
$Y = \text{Hom}(\mathbb{G}_m, T)$ be the group of cocharacters of $T$. This comes with a natural action of the Weyl group $W$ of $G$. Then the central extensions are in bijection with $W$-invariant symmetric bilinear forms
$$ B: Y \times Y \longrightarrow \mathbb{Z} $$
such that $Q(\alpha^\vee) := B(\alpha^\vee, \alpha^\vee) / 2 \in \mathbb{Z}$ for all coroots $\alpha^\vee$.
In the case that $k$ is the local field $F_v$, by composing with the Hilbert symbol map
$ K_2(F_v) \rightarrow \mu_n$, one realizes the group $\widetilde{G}_v$ above. 

If $v$ is a complex place, then the Hilbert symbol is identically $1$ and the $n$-fold metaplectic cover $\widetilde{G}_v$ is
isomorphic to $G(F_v)\times\mu_n$.  For the rest of this subsection, suppose instead that $v$ is non-archimedean.

Basic properties of the $n$-th order local Hilbert symbol may be found in~\cite{serre}.
In particular, we recall that when
$\ord_v(n)=0$ the symbol is a bilinear map $( \cdot, \cdot ) : F_v^\times \times F_v^\times \longrightarrow \mu_n$
satisfying
$$ (s, t) (t,s) = (t, -t) = (t, 1-t) = 1 \quad \text{for all $s, t \in F_v^\times$}. $$
The assumptions on $F$ ensure that $n | (q_v-1)$, where $q_v$ is the cardinality of the residue field
of $F_v$. This implies that the Hilbert symbol $(s,t)$, $s, t \in F_v^\times$, is determined from the congruence
$$ (s,t) \equiv \left( (-1)^{\ord_v(s) \ord_v(t)} s^{\ord_v(t)}t^{-\ord_v(s)} \right)^{(q_v-1)/n}\bmod \varpi_v\mathfrak{o}_v, $$
where $\ord_v$ is the valuation at $v$.
Moreover, since $2n | (q_v-1)$, we have the evaluations
$$ (-1, t) = 1, \quad (\varpi_v^a, \varpi_v^b) = 1 \quad \text{for any} \quad t \in F_v^\times, a, b \in \mathbb{Z} .$$

In the special case $G = SL_2$, the restriction map $K_2(F_v) \rightarrow \mu_n$ induces a map
of abelian groups $\xi: \mathbb{Z} \longrightarrow H^2(SL_2(F_v), \mu_n)$.
To work over a split reductive group we use a slightly more general result, due to McNamara \cite{mcnamara-found}.

\begin{theorem} To any split reductive group $G$ over a local field $F_v$ with assumptions
as above, and a choice of bilinear form $B$, there exists a central extension $\widetilde{G}_v$ of $G(F_v)$ by
$\mu_n$ such that, for each root $\alpha$, the pullback of the central extension under the morphism of
group schemes
$$ \iota_\alpha : SL_2 \rightarrow G $$
to a central extension of $SL_2$ is realized by the cohomology class $\xi(Q(\alpha^\vee))$.
\label{localmet} \end{theorem}

We will work with a two-cocycle $\sigma_v\colon G(F_v)\times G(F_v)\to \mu_n$ whose cohomology
class in $H^2(G(F_v), \mu_n)$ corresponds to $\widetilde{G}_v$.
Then we may realize
$$\widetilde{G}_v=\{(g,\zeta)\mid g\in G(F_v), \zeta\in\mu_n\}$$ with multiplication given by
$$(g_1,\zeta_1)(g_2,\zeta_2)=(g_1g_2,\sigma_v(g_1,g_2)\,\zeta_1\zeta_2).$$
If $G=SL_2$, following Kubota \cite{kubota}, we have the following explicit formula
for a cocycle $\sigma^{\alpha}_v$ whose class in $H^2(SL_2(F_v),\mu_n)$
is $\xi(Q(\alpha^\vee))$:
$$ \sigma^{\alpha}_v(g, h) = \left( \frac{x(gh)}{x(g)}, \frac{x(gh)}{x(h)} \right)^{-Q(\alpha^\vee)}, \quad \text{where} \quad x(g) = x \left( \begin{matrix} a & b \\ c & d \end{matrix} \right) = \begin{cases} c & \text{if $c \ne 0$}, \\ d & \text{if $c = 0$.} \end{cases} $$
For general $G$, we use the two-cocycle $\sigma_v$ that matches $\sigma^{\alpha}_v$ upon
composition with $\iota_\alpha$ for each coroot $\alpha$.  Since $\sigma_v$ is given
in terms of local Hilbert symbols, by Hilbert reciprocity we have that
if $g,h\in G(F)$ then $\prod_v \sigma_v(g,h)=1$.

Moreover, given a choice of ordered basis  $e_1, \ldots, e_r$ for the cocharacter group $Y$, there is an induced isomorphism $(F_v^\times)^r \simeq T(F_v)$. Then 
the 2-cocycle $\sigma$ on a pair of torus elements $s = (s_1, \ldots, s_r)$ and $t = (t_1, \ldots, t_r)$ may be given
explicitly (see p.~304 of~\cite{mcnamara-found}) by:
\begin{equation} \sigma_v(s,t) = \prod_{i \leq j} (s_i, t_j)_{2n}^{q_{i,j}} \quad \text{where} \quad Q(\sum_i y_i e_i) = \sum_{i \leq j} q_{i,j} y_i y_j. 
\label{torusformula} \end{equation}
Here we have written $(s,t)_{2n}$ to denote the local Hilbert symbol of order $2n$. 
Note that the $q_{i,j}$ depend on the ordered basis.
In Section~\ref{eisensteincon}, we place additional constraints on this ordered basis of $Y$ in order to
construct parabolic Eisenstein series. 

Finally, we record two splitting properties of the central extension $\widetilde{G}_v$.

\begin{proposition} The extension
$\widetilde{G}_v$ splits canonically over any unipotent subgroup of $G(F_v)$.
\end{proposition}

The splitting is via the trivial section for the usual unipotent subgroup corresponding to positive roots.  
See, for example, M\oe glin-Waldspurger
 \cite[Appendix 1] {moeglin-waldspurger}.  Building on work of Brylinski-Deligne (see Section 10.7 of
\cite{brylinski-deligne}) and Moore (\cite{moore}, Lemma 11.3), one also has the following splitting (see~
\cite[Theorem 2]{mcnamara-found} and the references cited there).

\begin{proposition}\label{max-cpt} Suppose that $n$ is relatively prime to the residue characteristic of $F_v$. Then
$\widetilde{G}_v$ splits over the maximal compact subgroup $K_v = G(\mathfrak{o}_v)$.
\end{proposition}

In particular, if
$n$ is relatively prime to the residue characteristic, following Kubota \cite{kubota} we may choose a lifting 
$(k, \kappa_v(k))$ of $G(\mathfrak{o}_v)$ to $\widetilde{G}_v$ where, on matrices in $SL_2(\mathfrak{o}_v)$,
\begin{equation} \kappa_v \left(\iota_\alpha \left( \begin{matrix} a & b \\ c & d \end{matrix} \right) \right) = \begin{cases} (c,d)^{-Q(\alpha^\vee)} & \text{if $0 < |c|_v < 1$} \\ 1 & \text{otherwise.} \end{cases} \label{ksection} \end{equation}
For these $F_v$ the local ``Kubota symbol'' $\kappa_v$ then satisfies
\begin{equation} \sigma_v(g_v, g_v') = \frac{\kappa_v(g_v) \kappa_v(g_v')}{\kappa_v(g_v g_v')} \quad \text{for all $g_v, g_v' \in G(\mathfrak{o}_v)$.} \label{localkubota} \end{equation}  

For the remaining non-archimedean places, the cover splits over a subgroup of $K_v$ of finite index.

\subsection{S-integers\label{sints}}

The metaplectic cover may also be defined over a ring of $S$-integers $\mathfrak{o}_S$, the elements of $F$ integral outside a finite set of places $S$.  We require that $S$ 
contains the archimedean places $S_\infty$ and all places ramified over $\mathbbm{Q}$ (in particular
including those dividing $n$), and is large enough such that the ring $\mathfrak{o}_S$ is a principal ideal domain.  
Let $F_S = \prod_{v \in S} F_v$. The metaplectic extension of
$G(F_S)$ by $\mu_n$, denoted $\widetilde{G}_{F_S}$, is the fiber product over $\mu_n$ of the local extensions $\widetilde{G}_v$ for
each $v \in S$ (that is, the quotient of the direct product $\prod_{v\in S}\widetilde{G}_v$ by the equivalence relation
identifying the central $\mu_n$ in each factor).
We will use $\sigma$ to denote the two-cocycle $\sigma=\prod_{v\in S}\sigma_v$
and $\s$ for the section $\s(g)=(g,1)$. For brevity, we will sometimes write $\widetilde{G}$ for $\widetilde{G}_{F_S}$.

We recall the Kubota map $\kappa$ in the context of the $S$-integers.  
If $c,d\in\mathfrak{o}_S$ are coprime, let $ \left( \frac{d}{c} \right)\in\mu_n$
denote the $n$-th power residue symbol.  (Thus if $c=p$ is prime in $\mathfrak{o}_S$ then this quantity
is congruent to $d^{(|p|-1)/n}$ modulo $p\mathfrak{o}_S$, where $|\cdot|$ denotes the absolute norm; see~\cite{neukirch}.)
Embed $\mathfrak{o}_S$ in $F_S$ by the diagonal embedding;
this gives rise to an embedding of  $G(\mathfrak{o}_S)$ 
in $ G(F_S)$.

\begin{lemma} 
  There exists a map $\kappa : G(\mathfrak{o}_S)
  \longrightarrow \mu_n$ such that
  \begin{equation}
    \label{kubotabasicproperty} \kappa (\gamma \gamma') = \sigma (\gamma,
    \gamma') \kappa (\gamma) \kappa (\gamma') .
  \end{equation}
  If $\alpha$ is any positive root then
  \begin{equation}
    \label{kubotaformula} \kappa \left( \iota_{\alpha} \left(\begin{array}{cc}
      a & b\\
      c & d
    \end{array}\right) \right) = \begin{cases}
      \left( \frac{d}{c} \right)^{Q(\alpha^\vee)} & \text{if $c \neq 0$}\\
      1 & \text{if $c = 0$,}
    \end{cases}
  \end{equation}
 where  $\iota_\alpha :
{SL}_2 \longrightarrow G$ is the canonical embedding.
\end{lemma}

\begin{proof} For $\gamma \in G(\mathfrak{o}_S)$, define
$$ \kappa(\gamma) = \prod_{v \not\in S} \kappa_v(\gamma). $$
The right-hand side is well-defined owing to (\ref{localkubota}), which ensures that
$\kappa_v(\gamma) = 1$ for almost all $v \not\in S$. 
Using  (\ref{ksection}) and the relation between the power residue and Hilbert
symbols (see for example Neukirch \cite{neukirch}; in 
Proposition V.3.4 note that the Hilbert symbol there is the inverse of ours), one obtains
$$ \kappa \left( \iota_\alpha \left( \begin{matrix} a & b \\ c & d \end{matrix} \right) \right) = \prod_{\substack{v \not\in S \\ \varpi_v | c}} (\varpi_v,d)^{-Q(\alpha^\vee) \text{ord}_v(c)} = \prod_{\substack{v \not\in S \\ \varpi_v | c}} \left( \frac{d}{\varpi_v} \right)^{Q(\alpha^\vee) \text{ord}_v(c)} = \left( \frac{d}{c} \right)^{Q(\alpha^\vee)}, $$
since $c, d$ are coprime, thus $d$ is a unit in $\mathfrak{o}_v$ whenever $\varpi_v | c$.
Then if $\gamma,\gamma'\in G(\mathfrak{o}_S)$, by Hilbert reciprocity
$$ \sigma(\gamma,\gamma') = \prod_{v \in S} \sigma_v(\gamma, \gamma') = \prod_{v \not\in S} \sigma_v(\gamma,\gamma')^{-1} = \frac{\kappa(\gamma\gamma')}{\kappa(\gamma) \kappa(\gamma')}, $$
as desired. 
 \end{proof}

\begin{corollary} \label{isahom} The map $\ii:G (\mathfrak{o}_S)\to 
\widetilde{G}$ given by $\ii(\gamma)=(\gamma,\kappa(\gamma))$
is a homomorphism.
\end{corollary}

\section{Eisenstein series on the metaplectic group\label{eisensteincon}}

We now define the maximal parabolic Eisenstein series on the $n$-fold cover $\widetilde{G}$ of $G(F_S)$ that will be the
focus of the remainder of the paper. Let $P$ be a standard maximal parabolic subgroup of $G$ and 
let $M$ be the Levi subgroup of $P$. We shall suppose
that this Levi subgroup factors as $M = M_1 \times M_2$ where $M_i$ are linear algebraic groups.
Given automorphic representations on suitable covers of $M_1(F_S)$ and $M_2(F_S)$ 
we shall define a maximal parabolic Eisenstein series on $\widetilde{G}$.

If $H$ is any subgroup of $G(F_S)$, let $\widetilde{H}=\{(h,\zeta)\mid h\in H,\zeta\in\mu_n\}$ denote the full 
inverse image of $H$ in $\widetilde{G}$.  In particular, we
have the subgroups $\widetilde{M}_i$, $i=1,2$ and $\widetilde{M}$ of $\widetilde{G}$.  
For any two subgroups $H_1,H_2$ of $G(F_S)$, let $\widetilde{H}_1\times_{\mu_n} \widetilde{H}_2$ denote the fiber product of their inverse images over $\mu_n$.
In general $\widetilde{M}_1 \times_{\mu_n} \widetilde{M}_2$ does {\sl not} embed in $\widetilde{G}$,
owing to the nature of the metaplectic cocycle. (For example in the case $G = GL_r$, this may be seen using 
block compatibility as in Banks, Levy and Sepanski~\cite{bls}.) Instead, in what follows, we construct a finite index subgroup 
$\widetilde{M}_1^0 \times_{\mu_n} \widetilde{M}_2^0$ of
$\widetilde{M}_1 \times_{\mu_n} \widetilde{M}_2$ which does embed in $\widetilde{G}$.

This requires additional assumptions on the construction of the metaplectic cover from Section~\ref{covers}. There we chose an ordered basis $e_1, \ldots, e_r$
for the cocharacter group $Y$ of $G$ giving rise to an induced
isomorphism $(F_S^\times)^r \simeq T(F_S)$. Let $T_i$ be the maximal split torus of $M_i$ for $i=1,2$.  We further assume that the ordered basis for $Y$ may
be partitioned into ordered bases for the group of cocharacters $Y_i$ of the $T_i$ giving rise to respective isomorphisms with copies of $F_S^\times$. Let $\chi^{(i)}_1, 
\ldots, \chi^{(i)}_{m_i}$ be the dual basis of characters $X_i$ of $T_i$.
Then we at last define 
$$T_i^0 =  \left\{ t \in T_i(F_S) \; \Big\vert \; \prod_{j=1}^{m_i} \chi_j^{(i)}(t) \in \Omega \right\}  \quad \text{where} \quad \Omega = \mathfrak{o}_S^\times (F_S^\times)^n. $$
The subgroup $\Omega$ is maximal isotropic with respect to the $n$-th order Hilbert symbol. The maximality may be deduced
from Proposition~8 of Section~XIII.5 of \cite{weil-bnt}. The condition that the product of characters is in $\Omega$ mimics the condition on the determinant naturally appearing in block compatibility for the general linear group. It is clear that the subgroup $\widetilde{T}_1^0 \times_{\mu_n} \widetilde{T}_2^0$ 
embeds in $\widetilde{G}$ according to~(\ref{torusformula}), provided that we order the basis elements of $Y$ so that $e_1, \ldots, e_{m_1}$ are a basis for $Y_1$ of $T_1$.

Now define $\widetilde{M}_i^0$ to be the subgroup of $\widetilde{M}_i$ generated by $\widetilde{T}_i^0$, $U_i$ (the maximal unipotent subgroup of $M_i$, which embeds
in $\widetilde{M_i}$ by the trivial section) and $W_{M_i}$, the Weyl group of $M_i$. (Here we realize $W_{M_i}$ as a subgroup of $M_i(\mathfrak{o}_S)$.)
Then $\widetilde{M}_i^0$ is a subgroup of finite index in $\widetilde{M}_i$ according to the Bruhat decomposition for $M_i$. 
To see that the resulting subgroup $\widetilde{M}_1^0 \times_{\mu_n} \widetilde{M}_2^0$ embeds in $\widetilde{G}$,
 it suffices to compute the cocycle coming from a Weyl group element in
 $W_{M_i}$ and a torus element of $\widetilde{T}_j^0$, $\{i,j\}=\{1,2\}$. 
 This is trivial according to the definition of the action in \cite{matsumoto, mcnamara-thesis}.
 
 Recall that $S_\infty$ denotes the set of archimedean places.  Let $S_{\text{fin}}$ denote the set of finite
 places in $S$, and let $F_\infty=\prod_{v\in S_\infty} F_v$, $F_{S_\text{fin}}=\prod_{v\in S_\text{fin}} F_v$.
Thus setting $H = T_i$ or $M_i$ with $i=1,2$, we may write
$H^0=H^0_{\text{fin}}\times H(F_{\infty})$ where $H(F_{\infty})=\prod_{v\in S_{\infty}} H(F_v)$.  Since the
cocycle is trivial at any complex place, we also have a natural factorization
$\widetilde{H}^0=\widetilde{H}^0_{\text{fin}}\times H(F_{\infty})$.  For the same reason, we also
have a natural factorization $\widetilde{G}=\widetilde{G}_{\text{fin}}\times G(F_\infty)$.

\subsection{Inducing in Stages}  We fix an embedding of $\mu_n$ into $\mathbb{C}^\times$ throughout this paper.
Let $(\pi_i, V_i)$ for $i=1,2$ be genuine automorphic representations of $\widetilde{M}_i$
which are unramified outside $S$.  
This requires a bit of care.  Indeed,
if $\sigma$ is a two-cocycle representing a given cohomology class in $H^2(G(F_S),\mu_n)$,  let $\sigma_i$ for $i=1,2$ be the restriction of $\sigma$ to ${M}_i(F_S)$.
Then the class of $\sigma_i$ is in $H^2(M_i(F_S),\mu_n)$.   However, it is possible
that the class of 
this cocycle lies in a smaller group $H^2(M_i(F_S),\mu_{n'})$ where $n'$ is a proper divisor of $n$.  (For example,
if $P$ is a parabolic subgroup of $G=GSp_{2a+b}$ such that $M\cong GL_a\times GSp_b$, the cocycle $\sigma$ restricted to
${GL}_a$ lies in $H^2(GL_a(F_S),\mu_{n/(n,2)})$.) In this case, a genuine automorphic 
representation on $\widetilde{M}_i$ corresponds to one on the $n'$-fold cover, extended to the $n$-fold cover
obtained by the inclusion $H^2(M_i(F_S),\mu_{n'})\subseteq H^2(M_i(F_S),\mu_n)$.

Let $(\pi_i^0, V_i)$ denote the restriction of $\pi_i$ to $\widetilde{M}_i^0$.  
Then $\pi_1^0\otimes \pi_2^0$ gives a representation of $\widetilde{M}_0\simeq \widetilde{M}_1^0 \times_{\mu_n} \widetilde{M}_2^0$.
(Indeed, since both $\pi_i^0$ are genuine, the product $\pi_1^0\otimes \pi_2^0$ is well-defined modulo $\sim_{\mu_n}$.)
We will define a representation of $\widetilde{G}$ by inducing in stages, first from $\widetilde{M}_0$ to $\widetilde{M}$ and then parabolically inducing from $\widetilde{M}$ to $\widetilde{G}$.

For each $v\in S_\text{fin}$, let $K_v$ be a compact open subgroup of $G(F_v)$ such that $\widetilde{G}_v$ splits over $K_v$. For each $v\in S_\infty$ let
$K_v$ be a maximal compact subgroup. For all $v \in S$, we choose $K_v$ in ``good'' relative position with respect to $M(F_v)$. If $v$ is Archimedean this means that $K$ and the maximal torus of $M$ are orthogonal relative to the Killing form, and if $v$ is finite, then the vertex of $K$ in the Bruhat-Tits building lies in the apartment of the maximal torus of $M$. (In particular, as noted in Proposition~\ref{max-cpt}, we may take $K_v=G(\mathfrak{o}_v)$ when $n$ is relatively prime to the residual characteristic of $F_v$.) This ensures an Iwasawa decomposition to be used later.
Finally let  $K=\prod_{v\in S} K_v$ and $K_\infty = \prod_{v\in S_\infty} K_v$. We regard $K$ as contained in $\widetilde{G}$ via the product of the local splittings.

For the first step, the induced representation $V_{\tilde{M}}$ of $\widetilde{M}$ is constructed as follows:
\begin{multline*} V_{\tilde{M}} := \text{Ind}_{\tilde{M}_0}^{\tilde{M}}(\pi_1^0 \otimes \pi_2^0) =  \{ \phi : \widetilde{M} \longrightarrow V_1 \otimes V_2 \; | \text{$\phi$ genuine and $K\cap \widetilde{M}$-finite, and} \\ \phi(\tilde{h}  \tilde{m}) = (\pi_1^0\otimes \pi_2^0)(\tilde{h})\cdot \phi(\tilde{m}) \; \forall \; \tilde{h} \in \widetilde{M}_0, \tilde{m} \in \widetilde{M} \}. \end{multline*}
As usual $\widetilde{M}$ acts on $V_{\tilde{M}}$ by the right regular representation; we denote this $\pi_{\tilde{M}}$ below.
Note that any $\phi \in V_{\tilde{M}}$ is determined by its restriction to $\widetilde{M}_{\text{fin}}$ since $M(F_\infty)$ embeds in $\widetilde{M}_0$.

Let us explain how to construct functions in $V_{\tilde{M}}$ concretely.  
Let $R$ be a set of coset representatives for $M_0\backslash M(F_S)$. For later use, we require this set of representatives to be in $T(F_S)$ and the equal to the identity element of $T(F_v)$ for $v \in S_\infty$. 
Let 
\begin{equation}\label{momega}
 \mathcal{M}_{P,R}(\Omega) := \Big\{ \Psi : M \rightarrow \mathbb{C} \, | \Psi(hr)=\sigma(h,r)^{-1}\Psi(r)~\text{for all $h\in M_0$,
$r\in R$}\Big\}.
\end{equation}
Note that $$\sigma(h,m)=\sigma_1(h_1,m_1)\sigma_2(h_2,m_2)$$ if $h=(h_1,h_2)$ and $m=(m_1,m_2)$ are in $M_0$
by the block compatibility property of $\sigma$ and the isotropy property of $\Omega$.
 In the special case $M=GL_1\times GL_1$ this factor is identically $1$, so the space is independent of the
 choice of $R$ (compare \cite{brubaker-bump}).  
Let $f_i\in V_i$ for $i=1,2$, and $\Psi\in \mathcal{M}_{P,R}(\Omega)$.  Define the $V_1 \otimes V_2$-valued function
\begin{equation}\label{general-phi}
\phi_{f_1,f_2,\Psi}(\tilde{m}):=\zeta\,\Psi(m)\,\,
 \pi_1^0(\s(m_1)) \otimes \pi_2^0(\s(m_2)) \cdot (f_1\otimes f_2)
 \end{equation}
 for $\tilde{m}=(m, \zeta) = ((m_1,m_2)r,\zeta)$ with $m_i\in M_i^0$, $r\in R$,  $\zeta\in\mu_n.$
(We frequently write $\phi(\tilde{m})$ and omit the subscripts.)
 
 \begin{proposition}\label{spanning}  The space $V_{\tilde{M}}$ is spanned by the functions 
 $\phi_{f_1,f_2,\Psi}$ with $f_i\in V_i$, $i=1,2$ and $\Psi\in \mathcal{M}_{P,R}(\Omega)$.
 Moreover, this spanning set does not depend on the choice
of coset representatives $R$ (up to multiplication by nonzero constants).
 \end{proposition}
 
 \begin{proof}
If $\tilde{h}=(h,\eta)=((h_1,h_2),\eta)\in\widetilde{M}_1^0\times_{\mu_n}\widetilde{M}_2^0$ and
$\tilde{m}=(m,\zeta)=((m_1,m_2)r,\zeta)$ with $m_i\in M_i^0$, $r\in R$,  $\zeta\in\mu_n$, then we have
$\tilde{h}\tilde{m}=((h_1m_1,h_2m_2)r,\zeta\eta\sigma(h,m))$.  Also 
using the cocycle property it is not hard to check that $\Psi\in \mathcal{M}_{P,R}$ if and only if
$$\Psi(h m) = \sigma(h, m)^{-1} 
\sigma(h,mr^{-1})
\Psi(m) 
\text{~for all $h\in M_0$, $m\in M$ with $m\in M_0r$, $r\in R$}. $$
Thus we see that if $\phi=\phi_{f_1,f_2,\Psi}$, then
\begin{multline*}
\phi(\tilde{h}\tilde{m})=\zeta\eta\sigma(h,m) \Psi(hm)  \pi_1^0(\s(h_1m_1)) \otimes \pi_2^0(\s(h_2m_2)) \cdot (f_1\otimes f_2)
=\\\eta\,\pi_1^0(\s(h_1))\otimes\pi_2^0(\s(h_2))\cdot \phi(\tilde{m}).
\end{multline*}
 Hence $\phi_{f_1,f_2,\Omega}\in V_{\tilde{M}}$.

Conversely, the space $V_{\tilde M}$ is spanned by the functions $\phi\in V_{\tilde{M}}$ such that $\phi(\s(r))$
is a pure tensor for each $r\in R$; write $\phi(\s(r))=f_{1,r}\otimes f_{2,r}$ with $f_{i,r}\in V_i$.  
Then $\phi(\tilde{h}\s(r))=(\pi_1^0\otimes \pi_2^0)(\tilde{h})\cdot f_{1,r}\otimes f_{2,r}$
for each $\tilde{h}  \in \widetilde{M}_0$.   
If $\tilde{m}=((m_1,m_2)r,\zeta)$ with $m_i\in M_i^0$, $r\in R$,  $\zeta\in\mu_n,$
then with $\tilde{h}=((m_1,m_2),\sigma((m_1,m_2),r)^{-1}\zeta)$ we see that
$\tilde{m}=\tilde{h}\s(r)$. Thus
$$\phi(\tilde{m})=\sigma((m_1,m_2),r)^{-1}\zeta\,\pi_1^0(\s(m_1))\otimes \pi_2^0(\s(m_2))\cdot f_{1,r}\otimes f_{2,r}.$$
For each $r\in R$, define $\Psi_r(m)=\sigma(m_0,r)^{-1}$ when $m=m_0r$ with $m_0\in M_0$, and 
$\Psi_r(m)=0$ if $m\not\in M_0r$.  
Then it is immediate that
 $\Psi_r\in\mathcal{M}_{P,R}(\Omega)$.  We see that
 $$\phi =\sum_{r\in R} \phi_{f_{1,r},f_{2,r},\Psi_r}.$$
Thus we conclude that the space $V_{\tilde{M}}$ is indeed spanned by functions
of the form (\ref{general-phi}).

Last, let us show that the set of functions $\phi_{f_{1},f_{2},\Psi}$ with $f_1\in V_1$, $f_2\in V_2$ and
$\Psi\in \mathcal{M}_{P,R}(\Omega)$ does not depend on the choice
of coset representatives $R$.  To see this, let $r\in R$, and 
let $R'$ be a set of coset representatives in which the representative $r$ is replaced by 
$hr$ with $h\in M_0$.  We add superscripts to indicate the choice of coset
representatives.  If $h=(h_1,h_2)$ and if $m=m_0r\in M_0 r$ then
$\Psi^{R'}_{hr}(m)=\sigma(m_0h^{-1},hr)^{-1}$.  Using the cocycle property, it is then easy to check that
$$\phi^{R'}_{f_1,f_2,\Psi^{R'}_{hr}}=\sigma(h^{-1},hr)^{-1} \phi^R_{f_1',f_2',\Psi_r^R}$$
with $f_1'=\pi_1^0(\s(h_1))^{-1}f_1$, $f_2'=\pi_2^0(\s(h_2))^{-1}f_2$.
\end{proof}

Next we parabolically induce up to $\widetilde{G}$. Suppose that $P$ has Levi decomposition $P=MU^P$.
Then  $\widetilde{P} = \widetilde{M} U^P$, where we have identified $U^P$ with its canonical embedding $\s(U^P)$.
Let  $\delta_P$ be the modular character of $P$ and for 
 $s \in \mathbb{C}$, let $\chi_s$ denote the character $\chi_s = \delta_P^s$.
 Extend these characters trivially to the cover by composing with the projection $\widetilde{M}\to M(F_S)$ (thus they are non-genuine characters of $\widetilde{M}$).
Let $\pi_s=\chi_s\delta_P^{1/2}\pi_{\tilde{M}}$, and extend this to a representation of $\widetilde{P}$ that is trivial on $U^P$. The space of the induced representation is
\begin{multline*}  \text{Ind}_{\tilde{P}}^{\tilde{G}}(\pi_s) =  \{ \varphi : \widetilde{G} \longrightarrow V_{\tilde{M}} \; |
~\varphi~\text{smooth~and} \\ \varphi(\tilde{m} u \tilde{g}) = \pi_s(\tilde{m}) \cdot \varphi(\tilde{g}) \; \forall \; \tilde{m} \in \widetilde{M}, u \in U^P, \tilde{g} \in \widetilde{G} \}. \end{multline*}

For later use, note that since $\widetilde{P} = \widetilde{P}_\text{fin} \times P(F_\infty)$, the representation $\pi_s$ factors as $\pi_s = \pi_{s, \text{fin}} \otimes \pi_{s, \infty}$.

\subsection{Explicit Test Vectors}

Let $B$ be a standard Borel subgroup of $G$ such that $B\subseteq P$ and let $U:=U^B$
denote the unipotent radical of $B$. Let $U^P$ denote the unipotent radical of $P$ and $U_M := U \cap M$, where $M$ is the Levi subgroup of $P$. Let $w_0\in G$ denote the long element of the Weyl
group $W$. It may be factored as $w_0 = w_M w^P$ where $w_M$ is the long element of $W_M = W \cap M$, the Weyl group of $M$, and $w^P$ is
a minimal length coset representative for the long element in $W_M \backslash W$.  We realize $w_0,w_M,w^P$ in $G(\mathfrak{o}_S)$
(and keep the same notation).
  
Recall that the finite set of places $S = S_\text{fin} \cup S_{\infty}$. Let $\mathfrak{o}_\text{fin} := \prod_{v \in S_\text{fin}} \mathfrak{o}_v$. For each $v\in S_{\text{fin}}$, let $\mathfrak{a}_v$ be an ideal of $\mathfrak{o}_v$
which is sufficiently large that if $a\in F_v$, $a-1\in \mathfrak{a}_v,$ then $a$ is an $n$-th power
in $F_v^\times$. 
Let $\mathfrak{a} := \prod_{v \in S_\text{fin}} \mathfrak{a}_v$, an ideal of $\mathfrak{o}_{\text{fin}}$.
We further require that $\mathfrak{a}$ is divisible
by the conductors of $\pi_1$ and $\pi_2$ in $\mathfrak{o}_\text{fin}$.
Let $U^P(\mathfrak{a})$ be the subgroup of $U^P(F_{S_\text{fin}})$ generated by the one-parameter subgroups
at roots in $U^P$ with elements in $\mathfrak{a}_v$ for all $v\in S_{\text{fin}}$. These definitions of groups and
elements depend on a realization of the algebraic group $G$; this is reviewed briefly in Section~\ref{sectionMDT}.

Fix coset representatives $R$ for $M_0 \backslash M(F_S)$ as above.  Given automorphic forms $f_i\in V_i$, $i=1,2$ and 
given $\Psi$ in $\mathcal{M}_{P,R}(\Omega)$ (see (\ref{momega})), recall that $\phi_{f_1,f_2,\Psi}\in V_{\tilde{M}}$ is 
defined by (\ref{general-phi}).
Using this function,
we may construct a function $\varphi_{s,\mathfrak{a}}$ in $\text{Ind}_{\tilde{P}}^{\tilde{G}}(\pi_s)$ associated to 
sufficiently large ideals $\mathfrak{a}$ of $\mathfrak{o}_\text{fin}$ as follows.  For simplicity, we assume that the vectors $f_i$ in $V_i$ are fixed under the standard maximal compact subgroups of $M_i(F_\infty)$. (More generally, one may proceed along the lines of Section 2 of~\cite{shahidi}, by acting on the right by a suitably chosen finite dimensional representation of the compact subgroup of $M(F_\infty)$.)
To define $\varphi_{s, \mathfrak{a}}$, factor $\tilde{g} = (\tilde{g}_\text{fin}, g_\infty)$, and let
\begin{equation} \varphi_{s,\mathfrak{a}}(\tilde{g}) =  \begin{cases} \pi_s((\tilde{m}_\text{fin},m_\infty)) \cdot \phi_{f_1,f_2,\Psi} & \begin{array}{l} \text{if $\tilde{g}_\text{fin} = \tilde{m}_\text{fin} u_\text{fin} w^P n_\text{fin} \in \widetilde{P}_{\text{fin}} w^P U^P(\mathfrak{a})$} \\ \text{and $g_\infty =  m_\infty u_\infty k_\infty \in G(F_\infty)$}, \end{array} \\ 0 & \; \text{otherwise.} \end{cases}
\label{goodmtildevector} \end{equation}
Here $\tilde{m}_\text{fin} \in\widetilde{M}_\text{fin}$, $u_\text{fin} \in U^P(F_{S_\text{fin}})$, and $n_\text{fin} \in U^P(\mathfrak{a})$. Moreover, $m_\infty \in M(F_\infty)$, $u_\infty \in U(F_\infty)$ and $k_\infty \in K_\infty$.

\begin{proposition} \label{funinducedspace}  The function $\varphi_{s,\mathfrak{a}}$
is an element of $\text{\textnormal{Ind}}_{\tilde{P}}^{\tilde{G}}(V_{\tilde{M}})$.
\end{proposition}

\begin{proof} It follows immediately from the definition that $\varphi_{s,\mathfrak{a}}$ satisfies the correct transformation property under left multiplication.
To check smoothness, if $v\in S_\text{fin}$ then make use of the bijective ``generalized Pl\"ucker coordinate" map
$\lambda: P(F_v)\backslash G(F_v)\to \mathbb{P}(V)$ where $\mathbb{P}(V)$ is the projective space attached to a certain 
finite dimensional $F_v$-vector space $V$.  Moreover, we may identify $V$ with $F_v^t$ for some $t$ such that
$\lambda(w_P)=[1:0:\ldots:0]$ and such that the image of $P(F_v) w^P U^P(\mathfrak{a}_v)$ is precisely the points
of the form $[1:a_1:\ldots:a_{t-1}]$ with $a_{i,v}\in \mathfrak{a}_v$, $1\leq i\leq t-1$.
Let $K(\mathfrak{a}_v)$ be the principal congruence subgroup of $K_v$ 
modulo $\mathfrak{a}_v$.  If $k\in K(\mathfrak{a}_v)$, then since $k$ is congruent to the identity modulo $\mathfrak{a}_v$, 
$\lambda(P(F_v)g)$ is of
the form $[1:a_1:\ldots:a_{t-1}]$ with $a_{i,v}\in \mathfrak{a}_v$ for $1\leq i\leq t-1$ if and only if $\lambda(P(F_v)gk)$ is of
this form.   
To see that $f_{\mathfrak{a}}$ is locally constant at $v$,  
we look at the decompositions of $g$ and of $gk$ for such $k$.
Suppose that $g=muw_0n$ and $gk=m'u'w_0n'$ with $m, m'\in M(F_v)$, $u,u'\in U^P(F_v)$, $n,n'\in U^P(\mathfrak{a}_v)$.
Then $k=g^{-1}gk=n^{-1}w_0^{-1}u^{-1}m^{-1}m'u'w_0n'$. 
Since $n, n'\in K(\mathfrak{a}_v)$ and $w_0$ normalizes $K(\mathfrak{a}_v)$, it follows that $u^{-1}m^{-1}m'u'\in K(\mathfrak{a}_v)$, whence $m^{-1}m'\in K(\mathfrak{a}_v)\cap M(F_v)$.  Moreover, the cocycles in this last computation do not change when we push this from $G(F_v)$ to $\widetilde{G}_v$.  
Indeed, if $\tilde g=\tilde m u w_0 n$ and $\tilde g k=\tilde m' u' w_0 n'$
then $\tilde m'=\tilde m k'$ for some $k'\in K(\mathfrak{a}_v)\cap \widetilde{M}(F_v)$,
and also $K(\mathfrak{a}_v)\cap \widetilde{M}_v$ is naturally
isomorphic to $(K(\mathfrak{a}_v) \cap \widetilde{M}_{1,v})\times (K(\mathfrak{a}_v)\cap \widetilde{M}_{2,v})$.
It follows that $\varphi_{s,\mathfrak{a}}(\tilde{g}) = \varphi_{s,\mathfrak{a}}(\tilde{g}k)$ for all $k\in K(\mathfrak{a}_v)$;
that is, $\varphi_{s,\mathfrak{a}}$ is locally constant at $v$.  If instead $v\in S_\infty$, then since the local cocycle
at an archimedean place is trivial, it follows immediately from the definition that 
$\varphi_{s,\mathfrak{a}}(\tilde{g}) = \varphi_{s,\mathfrak{a}}(\tilde{g}k)$ for all $k\in K_v$.
\end{proof}

\subsection{Construction of the Eisenstein Series\label{construction}}
Suppose that $\varphi_s\in  \text{Ind}_{\tilde{P}}^{\tilde{G}}(\pi_s)$ for $s$ in an open set of $\mathbb{C}$,
that $\varphi_s$ is continuous as a function of $s$ and that $\varphi_s$ restricted to $K$ is independent of $s$.
(These assumptions are  true for the test vectors described above.)  

Consider $\varphi_s(\iota(\gamma) \tilde{g})$ where $\gamma\in P(\mathfrak{o}_S)$ and $\tilde{g}\in\widetilde{G}$.
Write $\gamma=hu$ with $h\in M(\mathfrak{o}_S)$ and $u\in U^P(\mathfrak{o}_S)$.  By the defining property of 
$\text{Ind}_{\tilde{P}}^{\tilde{G}}(\pi_s)$ and the definition of $\delta_P$, we have
$$\varphi_s(\iota(\gamma) \tilde{g})=\pi_s(\iota(h)) \varphi_s(\tilde{g})=\pi_{\tilde{M}}(\iota(h)) \varphi_s(\tilde{g}).$$
Moreover, $\iota(h)\in \widetilde{M}_0$. Suppose 
 $\tilde{m}\in \widetilde{M}$.  Then, using the definition of $V_{\tilde{M}}$, we obtain
$$\varphi_s(\iota(\gamma)\tilde{g})(\tilde{m})=\varphi_s(\tilde{g})(\iota(h) \tilde{m}).$$

Let $\Lambda:V_1\otimes V_2\to\C$ be the functional that evaluates automorphic
forms at the identity.  If $\phi\in V_{\tilde{M}}$ and $\tilde{m}\in \widetilde{M}$ let us write
$\phi_{\tilde{m}}$ in place of $\phi(\tilde{m})$.  Then $\phi_{\tilde{m}}\in V_1\otimes V_2$
and 
\begin{equation*}
\Lambda(\phi_{\tilde{h}\tilde{m}})=\phi_{\tilde{m}}(\tilde{h})\qquad \text{ for all $\tilde{h}\in \widetilde{M}_0$, $\tilde{m}
\in \widetilde{M}$.}
\end{equation*}
Denote the application of $\Lambda$ to $\varphi_s$ by
\begin{equation}\label{invariance}
\varphi_s(\iota(\gamma)\tilde{g})(\tilde{m})(e)=\varphi_s(\tilde{g})(\tilde{m})(\iota(h)).
\end{equation}
From now on we assume that $\varphi_s(\tilde{g})(\tilde{m})$ is $\iota(M(\mathfrak{o}_S))$-fixed.  For given $f_i \in V_i$, 
this can always be arranged for the test vectors defined above by choosing $S$ sufficiently large.
Writing $\xi_s$ for the composition of $\varphi_s$ with $\Lambda$, we obtain
$$\xi_s(\iota(\gamma)\tilde{g})=\xi_s(\tilde{g})\qquad \text{~for~all~}\gamma\in P(\mathfrak{o}_S), \tilde{g}\in\widetilde{G}, $$
(an equality of complex-valued functions on $\widetilde{M}$).  Let us observe that $\xi_s$ restricts to an automorphic form of $\widetilde{M}$.  Indeed, (\ref{invariance}) shows that $\xi_s$ is a function on $\widetilde{M}$ that is left-invariant under $M(\mathfrak{o}_S)$.
Applying a second functional $\Lambda'$ given by evaluation at the identity of $\widetilde{M}$, then $f_s := \Lambda' \circ \xi_s$ is a complex-valued function on $\widetilde{G}$. Hence we may construct an Eisenstein series from $f_s$.

Given a function $f_s$ as above, we define the Eisenstein series
 \begin{equation}\label{Eisenstein}
 E_{f,s}(g)=\sum_{\gamma\in P(\mathfrak{o}_S)\backslash G(\mathfrak{o}_S)}
  f_s(\ii(\gamma) g).
  \end{equation}
  It converges for $\Re(s)$ sufficiently large.
  We sometimes suppress the dependence on $f$ and write the series simply as $E_s(g)$.

\begin{proposition}[M\oe glin-Waldspurger]\label{ACFE} $E_{f,s}$ possesses analytic continuation in $s$ to a meromorphic function on $\mathbb{C}$ with functional equation as $s \mapsto 1-s$ and $f\mapsto {\mathcal M}(w_0)f_s$,
where ${\mathcal M}(w_0)$ is the intertwining operator corresponding to the long element $w_0$.
\end{proposition}

\begin{proof} We show how to use the strong approximation theorem to prove that
a test vector given as a function on $\widetilde{G}$ can be written as an adelic
test vector satisfying the conditions imposed by M\oe glin-Waldspurger \cite{moeglin-waldspurger}.

Using the strong approximation theorem, it follows that the adelic metaplectic cover $\widetilde{G}_{\mathbb{A}}$ of 
$G(\mathbb{A})$ has the decomposition
$$ \widetilde{G}_{\mathbb{A}} = G(F) \widetilde{G}_{F_S} K^S $$
where $K^S$ denotes the product of $G(\mathfrak{o}_v)$ at $v \not\in S$. Here
$G(F)$ and $K^S$ have been identified with their embeddings as 
subgroups of the metaplectic group $\widetilde{G}_{\mathbb{A}}$, as the 
cocycle is trivial on $G(F)$ by Hilbert reciprocity and splits over $K^S$
according to the Kubota map.

Suppose first that $f_s$ is $K^S$-invariant.
Then we may lift the functions $f_s$ on $\widetilde{G}_{F_S}$ appearing in $\pi(s)$ 
to functions $f_s^{\mathbb{A}}$ on $\widetilde{G}_{\mathbb{A}}$, as follows. For any 
$\tilde{g}_\mathbb{A} \in \widetilde{G}_{\mathbb{A}}$,
we write
$$ \tilde{g}_\mathbb{A} = \gamma \tilde{g} k, \qquad \gamma \in G(F), \tilde{g} \in \widetilde{G}_{F_S}, k \in K^S $$
and then define $f_s^{\mathbb{A}}$ via the formula
$$ f_s^{\mathbb{A}} (\tilde{g}_{\mathbb{A}}) =  f_s (\tilde{g}). $$  
This function can be used to write the Eisenstein series $E_{f,s}$ adelically, and the results
of \cite{moeglin-waldspurger} then apply.  (A similar formula appears for Borel Eisenstein series on covers
of the general linear group in \cite{kazhdan-patterson}, p.\ 124.)

More generally, since functions in $V_{\tilde{M}}$ are $K\cap \widetilde{M}$-finite, we may find a finite dimensional
representation $\tau$ of $K^S$ and integrate against $f_s$ to extend it to an operator-valued function $\hat{f}_s$ that transforms
under the right by $\tau$, as in \cite{Shahidi-ParkCity}, Lecture 2.   This function may be extended to an adelic
function by $\hat{f}_s^{\mathbb{A}} (\tilde{g}_{\mathbb{A}}) = \tau(k)  \hat{f}_s (\tilde{g}).$  The matrix coefficients
may then be used to make a test vector with the desired properties, and the properties of the original series follow.
See Langlands \cite{scripture},  Chapter 4 and Shahidi \cite{Shahidi-ParkCity}, Section 2.1, especially Equation (2.7).
\end{proof}

We note that the section $f_s^{\mathbb{A}}$ is not factorizable.  Indeed, even inducing from the Borel subgroup,
it is necessary to introduce an infinite sum of factorizable sections.  For type $A$, this is the passage from $f_*$ to
$f_0$ found in \cite{kazhdan-patterson}, p.\ 109.

\section{Whittaker coefficients of the Eisenstein series}
We turn to the computation of the Whittaker coefficients of the Eisenstein series.
Recall that $B$ denotes a Borel subgroup of $G$ such that $B\subseteq P$ and $U$
denotes the unipotent radical of $B$.  For convenience,
since $U(F_S)$ embeds canonically in $\widetilde{G}$, if $u\in U(F_S)$ we write
$u\in\widetilde{G}$ instead of $(u,1)\in \widetilde{G}$.
Let $\psi$ be a non-degenerate character of $U(F_S)$ that is trivial on $U(\o_S)$.
Then  the $\psi$-Whittaker coefficient of the Eisenstein series is given by the integral
$$ \int_{U(\o_S) \backslash U(F_S)} E_{f,s}(u\tilde{g})\, \overline{\psi(u)} \,du, \qquad \tilde{g}\in\widetilde{G}.$$

In this section we evaluate these coefficients, modulo some decomposition theorems that are treated in
detail in later sections.   The sharpest of these theorems require that $G$ is cominuscule -- see
Definition~\ref{comindef} -- so we make these assumptions henceforth.  
Since $\widetilde{G}$ acts on the right, without loss we take $\tilde{g}=1$ from now on.
(To be clear, we compute the coefficients for a general  vector $f_s$, but the cleanest expression
is obtained for a specific set of special vectors which are not stable under this right action.  This is similar to the situation for Borel Eisenstein series;
see \cite{eisenxtal}, Section 6.)
To evaluate, we suppose that $\Re(s)$ is sufficiently large,
and we replace the Eisenstein series by its definition as a sum.  Since $\psi$ is non-degenerate,
it is not hard to check that the
Whittaker coefficient is supported on summands
arising from the relative ``big Bruhat cell'' $P(F_S) w_0 U^P(F_S) \, \cap \, G(\o_S)$, where $w_0\in G(\mathfrak{o}_S)$ is a representative for the long Weyl group element. 
Using the usual unfolding arguments, together with the fact that $U^P(\o_S)$ acts properly on the 
quotient $P(\o_S) \backslash P(F_S) w_0 U^P(F_S) \, \cap \, G(\o_S)$ (see Section~\ref{parametrization}, 
Corollary~\ref{corgencase}, below), we obtain
$$ \int_{U^P(F_S)} \int_{U_M(\o_S) \backslash U_M(F_S)} 
\sum_{\gamma \in P(\o_S) \backslash P(F_S) w_0 U^P(F_S) \, \cap \, G(\o_S) / U^P(\o_S)} \!\!\!\!\!\!\!\!\!\!\!
f_s(\ii (\gamma) u_M u^P)\, \overline{\psi(u_Mu^P)} \, du_M \, du^P, $$
where we have factored $U = U_M U^P$, writing $u\in U$ as $u = u_Mu^P$, $u_M\in U_M$, $u^P\in U^P$.  

Let ${U}_-^{P}$ be the opposite unipotent to $U^P$.
Next we replace double coset representatives $\gamma$ by $\gamma w_0$ such that
$$ \gamma \in P(\o_S) \backslash (P(F_S) U_{-}^P(F_S) \, \cap \, G(\o_S) ) / U^P(\o_S), $$
and apply a decomposition theorem for double cosets from Section~\ref{sectionMDT}
(see Corollary~\ref{corgencase} and Proposition~\ref{generalunipres}) together with Proposition~\ref{isahom} above. 
These results
express each such $\gamma$ as a product of $N$ embedded $SL_2(\mathfrak{o}_S)$ matrices according to a reduced decomposition for the long element in the relative Weyl group, where $N$ is the dimension of $U^P$. 
The double cosets are parametrized by the bottom rows $(c_j,d_j)$ of the $SL_2(\mathfrak{o}_S)$ matrices $g_j$
with $d_j\ne 0$ and with $(c_j,d_j)$ modulo units for $1\leq j\leq N$ and $c_j$ modulo 
$D_j := d_j \prod_{\ell = j+1}^N d_{\ell}^{\langle \gamma_j, \gamma_{\ell}^\vee \rangle}.$   The $\gamma_j$ here are the
roots in $U^P$ and the ordering of these roots is described in Section~\ref{sectionMDT}.

By repeated use of a rank one Bruhat decomposition, and keeping track of the contributions from the cocyles, we 
arrive at a decomposition $\gamma=v^P \tilde{\mathfrak{D}} w_0 u_\gamma$ for each
coset representative $\gamma$ where $v^P\in U^P(F_S)$,
$\tilde{\mathfrak{D}}\in\widetilde{G}$ projects to an element $\mathfrak{D}$ of $T(F_S)$, and $u_\gamma \in U(F_S)$.  (In the
notation of Proposition~\ref{generalunipres}, $u_\gamma=w_0vw_0^{-1}$. As indicated in this result, $u_\gamma$
depends on the choice of $g_j$.) 
We thus obtain:
\begin{multline*}
\int_{U^P(F_S)} \int_{U_M(\o_S) \backslash U_M(F_S)} \sum_{\substack{g_j \\ j=1, \ldots, N}} 
\left[ \prod_{k=1}^N (d_k, c_k)_S^{q_k} \left( \frac{d_k}{c_k} \right)^{\!\! q_k} \right] \times \\ 
f_s(v^P \tilde{\mathfrak{D}} w_0 u_\gamma u_Mu^P) \overline{\psi(u_Mu^P)} \, du_M \, du^P, 
\end{multline*}
where the $g_j$ have bottom row $(c_j, d_j)$, varying over coprime pairs with additional conditions as above, and
$\tilde{\mathfrak{D}}\in\widetilde{G}$ is given explicitly in Proposition~\ref{generalunipres}
below. For $c,d \in \mathfrak{o}_S$, the symbol $(d,c)_S$ is the product of the local Hilbert symbols $(d,c)_v$ over $v \in S$. The $q_k$ are the
integers $Q(\gamma_k^\vee)$ coming from the choice of bilinear form in the definition of $\widetilde{G}$.
Moreover, the group elements $v^P, u_\gamma$ depend only on this choice of $g_j$ 
while the $\tilde{\mathfrak{D}}$ depends only on the $d_j$.

The reciprocity law (see for example Neukirch \cite{neukirch}, Chapter VI, Theorem 8.3) implies that each of the terms in the product may be rewritten:
$$  (d_k, c_k)_S \left( \frac{d_k}{c_k} \right) = \left( \frac{c_k}{d_k} \right). $$
For $\text{Re}(s)$ sufficiently large, after a variable change we then obtain
\begin{equation*} \sum_{\substack{g_j \\ j=1, \ldots, N}} \psi(u_\gamma) \prod_{k=1}^N \left( \frac{c_k}{d_k} \right)^{\!\! q_k} \int_{U^P(F_S)} \int_{U_M(\o_S) \backslash U_M(F_S)} f_s(\tilde{\mathfrak{D}} w_M w^P u_Mu^P) \overline{\psi(u_Mu^P)} \, du_M \, du^P 
\end{equation*}
where we have factored $w_0 = w_M w^P$.
Since $\psi((w^P)^{-1} u_M w^P) = \psi(u_M)$, this becomes
\begin{equation}
\sum_{\substack{g_j \\ j=1, \ldots, N}} \psi(u_\gamma) \prod_{k=1}^N \left( \frac{c_k}{d_k} \right)^{\!\! q_k} \int_{U^P(F_S)} \int_{U_M(\o_S) \backslash U_M(F_S)} f_s(\tilde{\mathfrak{D}} w_M u_M w^P u^P) \overline{\psi(u_Mu^P)} \, du_M \, du^P. \label{mdsform-1} \end{equation}
Note further that the invariance of $f_s$ under left multiplication by elements $P(\mathfrak{o}_S) \cap U^-(F_S)$ ensures that the character $\psi$ must be trivial on all elements of the form $(\mathfrak{D} w_0)^{-1} p (\mathfrak{D} w_0)$ with $p \in P(\mathfrak{o}_S) \cap U^-(F_S)$. Indeed, we just move these elements $p$ rightward in $f_s$ and change variables. Hence, for any fixed choice of $\psi$, this places further divisibility conditions on the $d_i$ in order to yield a non-zero sum (see Lemma~\ref{whitdivis}). It follows that the exponential sum appearing in (\ref{mdsform-1}) depends only on the bottom row elements $(c_j, d_j)$ of $g_j$. Summing over $c_j$ mod $D_j$ for fixed choices of $d_j$ with $j = 1, \ldots, N$, we denote the result
\begin{equation}
H(d_1, \ldots, d_N) := \sum_{c_j (\text{mod } D_j)} \psi(u_\gamma) \prod_{k=1}^N \left( \frac{c_k}{d_k} \right)^{\!\! q_k}. 
\label{hdefined} \end{equation}
(The $u_\gamma$ here do depend on the $g_j$ so this is a modest abuse of notation; as noted the sum is independent of the choices of top rows.)

The integral appearing in the series (\ref{mdsform-1}) is given by
\begin{multline}
\int_{U^P(F_S)} \int_{U_M(\o_S) \backslash U_M(F_S)} f_{s}(\tilde{\mathfrak{D}} w_M u_M w^P u^P) \overline{\psi(u_Mu^P)} \, du_M \, du^P = \\
\delta_P^{s+1/2}(\mathfrak{D}) \int_{U^P(F_S)} \int_{U_M(\o_S) \backslash U_M(F_S)} \mkern-30mu \Lambda' \circ \Lambda \circ [ \pi_{\tilde{M}}(\tilde{\mathfrak{D}} w_M u_M ) \cdot \varphi_s(w^P u^P)] \overline{\psi(u_Mu^P)} \, du_M \, du^P. \label{ifsdefined}
\end{multline}
Thus~(\ref{mdsform-1}) may be expressed as a Dirichlet series in $s$:
$$ \sum_{\substack{d_j \in \mathfrak{o}_S  / \mathfrak{o}_S^\times, \, d_j \ne 0 \\ j = 1, \ldots, N}} H(d_1, \ldots, d_N) \delta_P^{s+1/2}(\mathfrak{D}) I_{f_s}(d_1, \ldots, d_N), $$
where $I_{f_s}(d_1, \ldots, d_N)$ is the double integral appearing on the right-hand side of~(\ref{ifsdefined}).

To further simplify $I_{f_s}(d_1, \ldots, d_N)$, write 
$$\tilde{\mathfrak{D}} w_M u_M = w_M (w_M^{-1} \tilde{\mathfrak{D}} w_M) u_M = w_M u_M^{\mathfrak{D}} (w_M^{-1} \tilde{\mathfrak{D}} w_M), $$ where $u_M^{\mathfrak{D}}$ denotes $u_M$ conjugated by $w_M^{-1} \tilde{\mathfrak{D}} w_M$. Note that $w_M^{-1} \tilde{\mathfrak{D}} w_M = (\mathfrak{D}^{w_M}, \zeta_{\mathfrak{D}})$ with $\mathfrak{D}^{w_M} \in M$ and $\zeta_\mathfrak{D} \in \mu_n$. Then, further factoring $\mathfrak{D}^{w_M} = (\mathfrak{D}_1, \mathfrak{D}_2) r$ with $(\mathfrak{D}_1, \mathfrak{D}_2) \in M_0 \simeq M_1^0 \times M_2^0$ and $r \in R$, we obtain
$$ \tilde{\mathfrak{D}} w_M u_M = w_M u_M^{\mathfrak{D}} (\mathfrak{D}, \zeta_\mathfrak{D}) = w_M u_M^{\mathfrak{D}} \mathbf{s}(\mathfrak{D}_1,\mathfrak{D}_2) \mathbf{s}(r) \zeta_\mathfrak{D} \sigma( (\mathfrak{D}_1, \mathfrak{D}_2), r)^{-1}. $$
Then $I_{f_s}(d_1,\ldots,d_N)$ is equal to $\zeta_\mathfrak{D} \sigma( (\mathfrak{D}_1, \mathfrak{D}_2), r)^{-1}$ times
\begin{equation}
\int_{U^P(F_S)} \int_{U_M(\o_S) \backslash U_M(F_S)} [\pi_{\tilde{M}}(\mathbf{s}(r)) \cdot \varphi_s(w^P u^P)](e)(w_M 
u_M^{\mathfrak{D}} \mathbf{s}(\mathfrak{D}_1,\mathfrak{D}_2)) \overline{\psi(u_Mu^P)} \, du_M \, du^P.
\label{readytospecialize} \end{equation}
Here $[\pi_{\tilde{M}}(\mathbf{s}(r)) \cdot \varphi_s(w^P u^P)](e)$ denotes evaluation of the function $[\pi_{\tilde{M}}(\mathbf{s}(r)) \cdot \varphi_s(w^P u^P)]$ on $\widetilde{M}$ at the identity $e$ in $\widetilde{M}$. Thus it defines an element of $V_1\otimes V_2$ and the inner integral is a Whittaker coefficient of an automorphic form in $V_1\otimes V_2$.

To say more, we restrict to the
test vector $f_{s,\mathfrak{a}}$ obtained from $\varphi_{s,\mathfrak{a}}$ as explained above. Then separating $w^Pu^P$ into its finite and infinite components and extracting the components in $\tilde{m}_\text{fin}(w^P u^P) = 1$ and $m_\infty = m_\infty(w^P u^P)$ according to~(\ref{goodmtildevector}), then~(\ref{readytospecialize}) becomes
\begin{multline*} \int_{U^P(\mathfrak{a})\times U^P(F_\infty)} \int_{U_M(\o_S) \backslash U_M(F_S)} \mkern-55mu [\pi_{\tilde{M}}(\mathbf{s}(r)) \pi_s(1,m_\infty) \cdot \phi_{f_1, f_2, \Psi}](e)(w_M 
u_M^{\mathfrak{D}} \mathbf{s}(\mathfrak{D}_1,\mathfrak{D}_2)) \overline{\psi(u_Mu^P)} \, du_M \, du^P \\
 =  \mu(\mathfrak{a}) \int_{U^P(F_\infty)} \int_{U_M(\o_S) \backslash U_M(F_S)} \mkern-20mu \Psi(r)\pi_s((1, m_\infty)) (f_1 \otimes f_2)(w_M 
u_M^{\mathfrak{D}} \mathbf{s}(\mathfrak{D}_1,\mathfrak{D}_2)) \overline{\psi(u_Mu^P)} \, du_M \, du^P,
\end{multline*} 
where $\mu(\mathfrak{a})$ denotes the measure of $U^P(\mathfrak{a}) \cap U^P(\mathfrak{o}_\text{fin})$. Thus including the 
above factor of $\zeta_\mathfrak{D} \sigma( (\mathfrak{D}_1, \mathfrak{D}_2), r)^{-1}$ into $I_{f_s}(d_1,\ldots,d_N)$ and making use of (\ref{momega}), $I_{f_s}(d_1,\ldots,d_N)$ equals
$$\mu(\mathfrak{a}) \zeta_\mathfrak{D} \Psi(\mathfrak{D}) \int_{U^P(F_\infty)} \int_{U_M(\o_S) \backslash U_M(F_S)} \pi_s((1,m_\infty)) (f_1 \otimes f_2)(w_M 
u_M^{\mathfrak{D}} \mathbf{s}(\mathfrak{D}_1,\mathfrak{D}_2)) \overline{\psi(u_Mu^P)} \, du_M \, du^P. $$
Now the cover splits at the infinite places, and the Whittaker models for $f_1$ and $f_2$ at those places are thus unique.
Hence, after a change of variables in $U_M$ to undo the conjugation by $\mathfrak{D}^{w_M}$, the inner integral gives the Whittaker coefficients of $f_1$ and $f_2$ associated with the character $u_M\mapsto \psi(u_M^{(\mathfrak{D}^{w_M})^{-1}})$, denoted $c_{f_1, f_2}^\psi(\mathfrak{D})$, times the archimedean Whittaker functions  of $f_1$ and $f_2$.  These are
evaluated at the  Borel (or $P(F_\infty)$)-part of $w^P u^P$ written in Iwasawa coordinates.  We also get a contribution of $\delta_P^{s+1/2}$
on this part.  The result is a Jacquet-type integral representing the Whittaker function of $G(F_\infty)$ attached to the representation with
Langlands parameters determined by the inducing data and $s$. We refer to this function as $\mathcal{W}_{f_1, f_2, s}(1)$, noting in particular that it is independent of $\mathfrak{D}$. Collecting these results in a theorem, we have shown:
\begin{theorem} Given a cominuscule parabolic $P$, the Whittaker coefficient of the maximal parabolic Eisenstein series $E_{f,s}$ is given by
\begin{equation}\label{general-case} \int_{U(\o_S) \backslash U(F_S)} E_{f,s}(u)\, \overline{\psi(u)} \,du = \sum_{\substack{d_j \in \mathfrak{o}_S / \mathfrak{o}_S^\times,  \, d_j \ne 0 \\ j=1, \ldots, N}} H(d_1, \ldots, d_N)\delta_P^{s+1/2}(\mathfrak{D}) I_{f_s}(d_1, \ldots, d_N).
\end{equation}
Here $H(d_1, \ldots, d_N)$ is the exponential sum defined in (\ref{hdefined})
(and described in detail in Section~\ref{sectionMDT} below), $\mathfrak{D}$ is computed from the $d_j$ as in Proposition~\ref{generalunipres}, and $I_{f_s}(d_1, \ldots, d_N)$ is defined in~(\ref{ifsdefined}).

Further, if one sets $\varphi_s = \varphi_{s,\mathfrak{a}}$ as defined in~(\ref{goodmtildevector}) with associated test vector $f := f_{s,\mathfrak{a}}$, then the Whittaker coefficient equals
\begin{equation}\label{series-with-ACFE} \mathcal{W}_{f_1,f_2,s}(1) \sum_{\substack{d_j \in \mathfrak{o}_S / \mathfrak{o}_S^\times, \, d_j \ne 0 \\ j=1, \ldots, N}} H(d_1, \ldots, d_N)\delta_P^{s+1/2}(\mathfrak{D}) \Psi(\mathfrak{D}) \zeta_\mathfrak{D} \, c_{f_1, f_2}^\psi(\mathfrak{D}). 
\end{equation}
\label{whittakeratlast} \end{theorem}

From Proposition~\ref{ACFE}, we deduce that the Whittaker coefficient \eqref{general-case}, when multiplied by a suitable
normalizing factor, has analytic continuation to all complex $s$ and satisfies a functional equation in which $s$
is replaced by $1-s$ and $f$ by ${\mathcal M}(w_0)f$.  For the trivial cover case $n=1$, this observation is the starting point for
Langlands-Shahidi theory, but for $n>1$ these series had not been exhibited before except in a few
instances.  Moreover, one may then apply Tauberian methods to
obtain information about the growth of the coefficients of these series.  Below, our focus will be on the combinatorial
and representation-theoretic description of the exponential sum $H$ that appears in \eqref{general-case} and
\eqref{series-with-ACFE}, so we 
do not pursue this further here.
However, we plan to study these consequences in subsequent works.

We remark that one could carry out this computation adelically, writing an adelic Eisenstein series
as discussed following Proposition~\ref{ACFE} and integrating over $U(F)\backslash U(\mathbb{A})$.  Using the factorization
$U=U_M\,U^P$ and unfolding in the usual way, one gets an iterated integral over $U^P(\mathbb{A})$ and 
over $U_M(F)\backslash U_M(\mathbb{A})$, the latter giving rise to the Whittaker coefficients of the inducing data.
The difficulty here is that the section in the Eisenstein series is not factorizable, and the resulting expression is
thus not simply a product of local terms.  (In addition, the Whittaker integral does not arise from a unique
model and is not factorizable.)  
In fact, this difficulty arises already for the Borel Eisenstein series on an $n$-fold cover of $GL_2$, 
whose Whittaker coefficients
are infinite sums of $n$-th order Gauss sums and are not Eulerian; compare the treatments in
\cite{brubaker-bump}, Section~5 (working over $F_S$) and \cite{kazhdan-patterson}, Ch.~II.3 (working over $\mathbb{A}$).
Working over $F_S$ for arbitrarily large but finite $S$ carries the same information as the 
corresponding adelic calculation.

\section{Double coset parametrizations\label{sectionMDT}}

In this section, we present an explicit parametrization for double cosets appearing in the definition of Eisenstein series. 
The idea is to use embedded rank one subgroups on which the metaplectic cocycle has a simple form. 
First, we perform the parametrization over the points of the algebraic group, from which the generalization to 
metaplectic covers readily follows. The use of Chevalley-Steinberg relations to explore double cosets of reductive 
groups is a common technique in algebraic combinatorics, especially that relating to canonical bases as in 
\cite{berenstein-zelevinsky, morier-genoud}. Indeed we connect the results of this section to canonical bases in 
Section~\ref{canbases}. But because our decomposition differs from those previously appearing in the literature, 
including working over a ring, we give a complete and self-contained treatment. Initially in this section, we 
let $\mathfrak{o}$ be any principal ideal domain with field of fractions $F$.

\subsection{Coset parametrizations for reductive groups}

Before discussing the general coset parametrization theorems of this section, we briefly recall the necessary structure theory of reductive groups via root data, for which a basic reference is Springer \cite{springer}.

Let $(X, \Phi, Y, \Phi^\vee)$ be the root datum corresponding to the pair $(G, T)$. It comes equipped with a bilinear pairing $\langle \cdot , \cdot \rangle \; : \; X \times Y \longrightarrow \mathbb{Z}$ such that $\langle \alpha, \alpha^\vee \rangle = 2$ for all roots $\alpha \in \Phi$ and simple reflections $s_\alpha(x) := x - \langle x, \alpha^\vee \rangle \alpha$ which generate the Weyl group $W$. Let $e_\alpha$ denote the isomorphism from the additive group $G_a$ to the unique closed subgroup $U_\alpha$ of $G$ such that
\begin{equation} t e_\alpha(x) t^{-1} = e_\alpha( \alpha(t) x) \label{tconj} \end{equation}
for all $x \in F$ and $t \in T(F)$. The isomorphism $e_\alpha$ is chosen so that, if we set
$$ w_\alpha (x) := e_\alpha(x) e_{-\alpha}(-x^{-1}) e_\alpha(x) \quad \text{for any $x \in F^\times$,} $$
then $w_\alpha(1)$ has image $s_\alpha$ in $W$. Then for any $x \in F^\times$ and any root $\beta \in \Phi$, defining the torus element $h_\beta(x) := w_\beta(x) w_\beta(1)^{-1}$, the conjugation in (\ref{tconj}) may be rewritten
\begin{equation} h_\beta(y) e_\alpha(x) h_\beta(y)^{-1} = e_\alpha( y^{\langle \alpha, \beta^\vee \rangle} x). \label{hecommutation} \end{equation}

Let $\Phi^+$ denote a set of positive roots with simple roots $\alpha_i$. Let $s_i$ be shorthand for the reflection corresponding to $\alpha_i$. To any $w \in W$, and a reduced decomposition as product of simple reflections $w = s_{i_1} \ldots s_{i_j}$, one has 
$$ \Phi_w := \{ \alpha \in \Phi^+ \; | \; w^{-1}(\alpha) \in \Phi^- \} = \{ \alpha_{i_1}, s_{i_1}(\alpha_{i_{2}}), \ldots, s_{i_1} \cdots s_{i_{j-1}}(\alpha_{i_{j}})\}. $$ 
Using this description, a reduced decomposition of $w$ gives rise to a total ordering on $\Phi_w$. For our application to parabolic induction, let $P$ be any parabolic subgroup of $G$. Write the long element $w_0 = w_M w^P$ with $w_M$ the long element of the Levi subgroup $M$ of $P$. Starting from a reduced decomposition for $w_M = s_{i_1} \cdots s_{i_t}$, extend this to a reduced decomposition for $w_0 = s_{i_1} \cdots s_{i_t} s_{i_{t+1}} \cdots s_{i_{t+N}}$. Let $\boldsymbol{i} = (i_1, \ldots, i_{t+N})$ be the corresponding ``reduced word'' and $<_{\boldsymbol{i}}$ be the corresponding ordering on $\Phi^+$ so that
\begin{equation} \alpha_{i_1} >_{\boldsymbol{i}} s_{i_1}(\alpha_{i_2}) >_{\boldsymbol{i}} \cdots >_{\boldsymbol{i}} s_{i_1} \cdots s_{i_{t+N-1}} (\alpha_{i_{t+N}}). \label{iordering} \end{equation}
The word may then be used to give an ordered parametrization of the positive roots $\Phi_P$ in $U^P$, the unipotent radical of the parabolic $P$.  We index the roots of $\Phi_P$ by $\gamma_i$ as follows:
\begin{equation} \gamma_1 = w_M (\alpha_{i_{t+1}}), \quad \gamma_{2} = w_M s_{i_{t+1}}(\alpha_{i_{t+2}}), \quad \ldots, \quad \gamma_N = w_M s_{i_{t+1}} \cdots s_{i_{t+N-1}} (\alpha_{i_{t+N}}). \label{phip} \end{equation}
Here $\gamma_{j} >_{\boldsymbol{i}} \gamma_k$ if $j < k$. Note that $\gamma_N = \alpha_{i_{t+N}}$, a simple root. The group $U_{-}^P$, the opposite of the unipotent radical of $P$, is generated by all $e_{-\a}(x)$ with $\a\in \{ \gamma_1, \ldots, \gamma_N \}$.

For any $\alpha \in \Phi^+$, we continue to denote by $\iota_\alpha$ the canonical morphism of group schemes from $SL_2$ to $G$ corresponding to $\alpha$:
$$ \iota_\alpha \; : \; SL_2 \longrightarrow \langle e_\alpha, e_{-\alpha} \rangle, \quad \text{such that} \quad \iota_\alpha \begin{pmatrix} 1 & x \\ 0 & 1 \end{pmatrix} = e_\alpha(x), \; \iota_\alpha \begin{pmatrix} 1 & 0 \\ x & 1 \end{pmatrix} = e_{-\alpha}(x). $$

Steinberg \cite{steinberg} gave a presentation for $G(F)$ valid over any field and depending only on $F$, the root datum, and structure constants $\eta_{\alpha, \beta; i,j}$ appearing in the commutator relation
\begin{equation}
 e_\alpha(s)e_\beta(t)e_\alpha(s)^{-1} = e_\beta(t) \big[\!\!\!\!\!\! \prod_{\substack{ i,j\in\Z^+ \\ i\alpha+j\beta=\gamma\in\Phi }} \!\!\!e_\gamma(\eta_{\alpha,\beta; i,j}s^it^j)\big] \quad (\alpha + \beta \ne 0) \label{eecommutation}. 
\end{equation}
The constants $\eta_{\alpha, \beta; i,j}$ are integers and the ordering on the product corresponds to a total ordering of the roots $\Phi$. We take the total ordering as in (\ref{iordering}) extended to $\Phi = \Phi^- \cup \Phi^+$ in the obvious way. Strictly speaking, Steinberg's presentation in \cite{steinberg} is for simply connected groups, but see \cite{springer}, Chapter 9, for a presentation valid for any reductive group.

The following result will be used repeatedly in the decomposition theorems below.
\begin{lemma} \label{convexorder} Fix a reduced word $\boldsymbol{i}$ and ordering $<_{\boldsymbol{i}}$ on $\Phi^+$. Let $\alpha, \beta \in \Phi^+$. \hfill \begin{enumerate}[a)]
\item Suppose $\alpha <_{\boldsymbol{i}} \beta$ and there exists $\gamma \in \Phi$ of the form $\gamma = i \alpha - j \beta$ for some positive integers $i,j$. Then either $\gamma \in \Phi^+$ with $\gamma <_{\boldsymbol{i}} \alpha$ or $\gamma \in \Phi^{-}$ with $-\gamma >_{\boldsymbol{i}} \beta$.
\item Suppose $\alpha >_{\boldsymbol{i}} \beta$ and there exists $\gamma \in \Phi$ of the form $\gamma = i \alpha + j \beta$ for some positive integers $i,j$. Then 
$\alpha >_{\boldsymbol{i}} \gamma >_{\boldsymbol{i}} \beta$.
\end{enumerate}
\end{lemma}

\begin{proof}[Proof of $(a)$] Write
$$ \alpha = s_{i_{1}} \cdots s_{i_{m-1}} \alpha_{i_m}, \quad \beta = s_{i_{1}} \cdots s_{i_{\ell-1}} \alpha_{i_\ell} \quad \text{where $m > \ell$.} $$
Then setting $w = s_{i_{1}} \cdots s_{i_{m-1}}$, we have $w^{-1}(\alpha) \in \Phi^+$ while $w^{-1}(\beta) \in \Phi^-$ (since $\beta \in \Phi_w$). Thus $w^{-1}(\gamma) = iw^{-1}(\alpha) - jw^{-1}(\beta) \in \Phi^+$. But $w$ is characterized by $\Phi_w = \{ \delta \in \Phi^+ \; | \; w^{-1}(\delta) \in \Phi^- \} = \{ \delta \in \Phi^+ \; | \; \delta >_{\boldsymbol{i}} \alpha \}$ so if $\gamma \in \Phi^+$, then $\gamma <_{\boldsymbol{i}} \alpha$. 

Using similar arguments for $-w_0 \alpha$ and $-w_0 \beta$ we may show that if $\gamma \in \Phi^{-}$, then $-\gamma >_{\boldsymbol{i}} \beta$. Pick $v = s_{i_{t+N}} \cdots s_{i_{\ell+1}}$ so that $v^{-1}(-w_0 \alpha) \in \Phi^{-}$ while $v^{-1}(-w_0 \beta) \in \Phi^+$. Since $\gamma \in \Phi^{-}$ if and only if $-w_0 \gamma \in \Phi^{-}$, then $-w_0(-\gamma) \in \Phi^+$. Further $v^{-1}(-w_0(-\gamma)) \in \Phi^+$. Since $-w_0$ is order reversing, the claim follows.
\end{proof}

\begin{proof}[Proof of $(b)$] Write
$$ \alpha = s_{i_1} \cdots s_{i_{k-1}} (\alpha_{i_k}), \quad \beta = s_{i_1} \cdots s_{i_{k'-1}} (\alpha_{i_{k'}}), \quad k' > k. $$
Choose $w = s_{i_{1}} \cdots s_{i_{k-1}}.$ Then $w^{-1}(\alpha)$ and $w^{-1}(\beta)$ are in $\Phi^+$ and thus, so is $w^{-1}(\gamma)$. According to the explicit description of $\Phi_w$, this implies $\gamma <_{\boldsymbol{i}} \alpha$. Now consider:
$$ - w_0 \alpha = s_{i_{t+N}} \cdots s_{i_{k+1}} (\alpha_{i_k}), \quad -w_0 \beta = s_{i_{t+N}} \cdots s_{i_{k'+1}} (\alpha_{i_{k'}}). $$
Choose $v = s_{i_{t+N}} \cdots s_{i_{k'+1}}$ so that $v^{-1}(-w_0 \beta)$ and $v^{-1}(-w_0 \alpha)$ are in $\Phi^+$. Thus $v^{-1}(- w_0 \gamma) \in \Phi^+$ and so $-w_0 \gamma <_{\boldsymbol{i}} -w_0 \alpha$. As $-w_0$ is order reversing, the claim follows.
\end{proof}

Let $B_1$ be the standard rank one Borel subgroup of $SL_2$ and $B_1^{-}$ the opposite Borel.

\begin{theorem} \label{mainthm}
To any $u \in U_{-}^P(F)$, there exist $g_1, \ldots, g_N \in SL_2(\o) \cap B_1(F) B_1^-(F)$ such that
\begin{equation} \iota_{\gamma_1}(g_1) \cdots \iota_{\gamma_N}(g_N) u^{-1} \in B(F), \label{gup} \end{equation}
the $F$-points of the standard Borel subgroup of $G$.
\end{theorem}

\begin{proof}
Given any $u \in U_{-}^P(F)$, we may write
$$ u^{-1} = e_{-\gamma_N}(x_N) \cdots e_{-\gamma_1}(x_1), \quad \text{with} \; x_i \in F. $$
Write $x_N = -c_N / d_N$ with $c_N, d_N \in \o$ and $\gcd(c_N, d_N)=1$. Then we may choose $a_N, b_N \in \o$ such that $a_N d_N - b_N c_N =1$, as $\mathfrak{o}$ is a principal ideal domain.
Letting
\begin{equation} \label{rankonehee} g_N = \begin{pmatrix} a_N & b_N \\ c_N & d_N \end{pmatrix}
\quad \text{we have} \quad \iota_{\gamma_N}(g_N) e_{-\gamma_N}(x_N) = h_{\gamma_N}(d_N^{-1}) e_{\gamma_N}(b_N d_N). \end{equation}
(Note that the choice of $d_N$ is necessarily non-zero and hence $g_N \in B_1(F) B_1^-(F).$) Thus
$$ \iota_{\gamma_N}(g_N) u^{-1} = h_{\gamma_N}(d_N^{-1}) e_{\gamma_N}(b_N d_N) e_{-\gamma_{N-1}}(x_{N-1}) \cdots e_{-\gamma_{1}}(x_{1}). $$
Now we move the terms $h_{\gamma_N}(d_N^{-1}) e_{\gamma_N}(b_N d_N)$ rightward past all of the $e_{-\gamma_j}(x_j)$. We must show the result is of the form
$$ e_{-\gamma_{N-1}}(x_{N-1}') \cdots e_{-\gamma_{1}}(x_{1}') b \quad \text{for some} \; b \in B(F), x_j' \in F. $$
More generally, if $2 \leq k \leq N$, we will show that for any $d, y, x_1, \ldots, x_{k-1} \in F$, 
\begin{equation} h_{\gamma_k} (d) e_{\gamma_k}(y) e_{-\gamma_{k-1}}(x_{k-1}) \cdots e_{-\gamma_1}(x_1) = e_{-\gamma_{k-1}}(x_{k-1}') \cdots e_{-\gamma_1}(x_1') b, \label{cranky} \end{equation}
for some $x_1', \ldots x_{k-1}' \in F$ and $b \in B(F)$. Then choosing $g_{k-1}$ according to $x_{k-1}'$ in a similar manner to $g_N$ above, the first statement of the theorem will follow by repeated application of (\ref{cranky}).

The proof of (\ref{cranky}) will follow from the relations (\ref{hecommutation}) and (\ref{eecommutation}), together with Lemma~\ref{convexorder} which characterizes the terms appearing in relation (\ref{eecommutation}) when we attempt to move the $e_{\gamma_k}$ rightward past the terms $e_{-\gamma_j}$ with $\gamma_k <_{\boldsymbol{i}} \gamma_j$.

Note the condition that $\gamma = i \alpha - j \beta$ in part (a) of Lemma~\ref{convexorder} is precisely that appearing in Steinberg's commutation relation (\ref{eecommutation}) for $e_\alpha$ and $e_{-\beta}$ with $\alpha, \beta \in \Phi^+$. Take $\alpha = \gamma_k$ for some $k$ and $\beta = \gamma_j$ for some $j < k$. Thus $\alpha <_{\boldsymbol{i}} \beta$. Applying the commutation relation (\ref{eecommutation}) may produce elements of the form $e_\delta$ for additional positive roots 
$\delta$ which also must be moved rightward into $B(F)$. But part (a) of Lemma~\ref{convexorder} ensures that these are smaller than $\alpha$ (and hence smaller than $\beta$) so we may repeatedly use the lemma for any such $\delta$ and $\gamma_j$ with $j < k$. The total number of times a positive root may appear in a commutation is finite since the roots strictly decrease in the ordering with each application. All such one-dimensional unipotent subgroups associated to positive roots are, of course, contained in $B(F)$ so once they are moved rightward past the $e_{-\gamma_j}$ we need not consider them further.

If we commute an $e_\alpha$ (with $\alpha \in \Phi^+$) past an $e_{-\gamma_j}$ to produce a negative 
root $\gamma$, part (a) of the lemma also ensures that $-\gamma$ is larger than $\gamma_j$, so 
$\gamma = -\gamma_i$ for some $i < j$. That is, another element of the form $e_\gamma$ already appears 
to the right of $e_{-\gamma_j}$ in the expression for $u^{-1}$. We would like to combine these unipotents at the 
same negative root together. Indeed, part (b) of the lemma ensures that we may pass $e_\gamma(x)$ to the right 
past $e_{-\gamma_j}(x_j), \ldots, e_{-\gamma_{i+1}}(x_{i+1})$ without generating one-dimensional unipotent terms 
at negative roots not in this list. Then $e_\gamma(x) \cdot e_{-\gamma_i}(x_i) = e_{-\gamma_i}(x+x_i)$.

Thus to prove (\ref{cranky}), it remains only to move the term $h_{\gamma_k}(d)$ rightward past all of the 
unipotent elements $e_{-\gamma_j}$ with $1 \leq j < k$ and show that the result is of the form given on the 
right-hand side of (\ref{cranky}). This is immediate from (\ref{hecommutation}).
\end{proof}

Let $B_1(\o)$ denote the standard Borel subgroup of upper triangular matrices in $SL_2(\o)$. 
Recall that the cosets $B_1(\mathfrak{o}) \backslash SL_2(\o)$ are parametrized by their bottom rows -- pairs 
of integers $(c, d)$ with $\gcd(c,d)=1$ -- modulo the diagonal multiplication by units. Moreover, the pair 
$(c,d)$ corresponds to a coset in the ``big cell'' $B_1(\o) \backslash SL_2(\o) \cap B_1(F) B_1^-(F)$ if and only if 
$d \ne 0$. Motivated by this, we introduce the notation
$$ \mathcal{R} := \{ (c, d) \in \o^2 / \o^\times \; | \; \gcd(c, d)=1, d \ne 0 \}. $$

\begin{corollary} \label{leftcosets}
The assignment of $g_1, \ldots, g_N \in SL_2(\o) \cap B_1(F) B_1^-(F)$ to any 
$u \in U_{-}^P(F)$ as in Theorem~\ref{mainthm}
induces a bijection
$$ \mathcal{R}^N \stackrel{\eta}{\longrightarrow} P(F) \backslash P(F) U^-_P(F) $$
given by first completing $(c_i, d_i) \in \mathcal{R}^{(i)}$ to a matrix
$$ g_i = \begin{pmatrix} a_i & b_i \\ c_i & d_i \end{pmatrix} \in SL_2(\o) \cap B_1(F) B_1^-(F) $$
and then composing with the map
$$ (SL_2(\o) \cap B_1(F) B_1^{-}(F))^{N}\, \stackrel{\varphi}{\longrightarrow} \, P(F) \backslash P(F) U^-_P(F), $$
given in the previous theorem by $\varphi(g_1, \ldots, g_N) = P(F) \iota_{\gamma_1}(g_1) \cdots \iota_{\gamma_N}(g_N).$ The map $\varphi$ (and hence the bijection $\eta$) depends on the choices $(a_i, b_i)$ for each pair $(c_i, d_i)$, but is a bijection of sets for any such assignment.
\end{corollary}

\begin{proof}
In the proof of the previous theorem, the set of $SL_2(\o)$ matrices $(g_1, \ldots, g_N)$ mapping to a given coset $P(F) u$ is determined according to the ordering as follows. Each $g_i$ is determined by first choosing $(c_i, d_i)$ depending only on the $g_j$ with $j > i$. Then the choice of $a_i, b_i$ such that $a_i d_i - b_i c_i = 1$ is free. Thus the map $\varphi$ depends on the assignment of $(a_i, b_i)$ to $(c_i, d_i)$ but any choice is allowable.

There is an obvious bijection between $\mathcal{R}^N$ and $P(F) \backslash P(F) U^-_P(F)$ given by mapping
$$ \{ (c_i, d_i) \} \in \mathcal{R}^N \mapsto P(F) e_{-\gamma_N}(x_N) \cdots e_{-\gamma_1}(x_1), \quad \text{where } x_i = \frac{c_i}{d_i} \text{ for } i = 1, \ldots, n. $$
Writing
$$ \eta( \{ (c_i, d_i) \}_{i=1}^N) = P(F) e_{-\gamma_N}(x_N') \cdots e_{-\gamma_1}(x_1') $$
for some elements $x_i' \in F$, then the proof of the previous theorem shows that $x_N' = x_N$ (and so $\eta$ agrees with the obvious bijection in rank one). Moreover the previous proof demonstrates that for each $i < N$, $x_i' = x_i + k$ for some constant $k \in F$ depending on $g_j$ with $j > i$ (and the structure constants of the realization of $G$). Thus composing with these additive changes of variables preserves the bijection. 
\end{proof}

\subsection{Restriction to ``braidless'' and ``cominuscule'' parabolic subgroups}\label{parametrization}

Next, we give a parametrization for double cosets $P(F) \backslash P(F) U_{-}^P(F) / U_{-}^P(\o)$. The strategy for determining a set of representatives again uses the Steinberg commutation relations. These are defined over a field, but the following lemma will ensure that they may be used in certain cases to produce elements in the ring of integers $\o$. Thus far, our statements have held for all reductive groups, all parabolics $P$, and all reduced expressions $w^P$ for the ``long element'' of $W_M \backslash W$, i.e., the longest among minimal length left coset representatives. In this section, we restrict our attention to maximal parabolic subgroups $P$. First we require a definition, made in terms of fundamental weights $\omega_{\alpha_i} =: \omega_i$ for the weight lattice $X$ of $G$.

\begin{definition}[Littelmann, Section 3 of \cite{littelmann}]\label{braidless-123} A fundamental weight $\omega$ is said to be ``braidless'' if, for any $w \in W$, the following holds for pairs of simple roots $\alpha, \beta$:
$$ \text{If } \langle w(\omega), \alpha^\vee \rangle > 0 \; \text{and} \; \langle w(\omega), \beta^\vee \rangle > 0, \; \text{then} \; \langle \beta, \alpha^\vee \rangle = 0. $$
If $P_\omega$ is the associated maximal parabolic subgroup of $G$, we say the parabolic is ``braidless.''
\label{braidlessdef} \end{definition}

In Lemma~3.1 of \cite{littelmann}, the braidless fundamental weights are characterized as follows, using Bourbaki's labeling of the Dynkin diagram for the root system $\Phi$ of rank $r$ (so forks and double bonds in classical groups are closer to the node labeled `$r$' rather than the one labeled `$1$'). A fundamental weight $\omega$ is braidless if and only if it is in the following collection:
\begin{itemize}
\item $\Phi$ is of type $A$,
\item $\Phi$ is of type $B$ or $C$ and $\omega = \omega_1$ or $\omega_r$,
\item $\Phi$ is of type $D$ and $\omega = \omega_1, \omega_{r-1}$ or $\omega_r$,
\item $\Phi = E_6$ and $\omega = \omega_1$ or $\omega_6$ or $\Phi = E_7$ and $\omega = \omega_7$,
\item $\Phi = G_2$.
\end{itemize}
The name ``braidless'' refers to the fact that any element $w \in W_{M_\omega} \backslash W$, where $M_\omega$ is the Levi subgroup of $P_\omega$, has a reduced decomposition which is unique up to exchange of orthogonal reflections (i.e., no braid relations are required).

\begin{lemma} \label{nonneg}
Let $G$ be a connected, reductive group with $\Phi(G) \ne G_2$. Then the maximal parabolic $P = P_\omega$ is braidless if and only if 
$$ \langle \gamma_j, \gamma_k^\vee \rangle \geq 0 \quad \text{for all roots $\gamma_j, \gamma_k \in U^P.$} $$
\end{lemma}

\begin{proof} We use the characterization of braidless weights given above, and break the proof into finitely many cases according to Cartan type. By linearity of the inner product, the question reduces to a study of inner products for adjacent simple roots. For example, in type $A$ if $\alpha_j$ denotes the omitted root, then all roots in $U^P$ are of the form $\gamma = \alpha_i + \cdots + \alpha_k$ with $i \leq j \leq k$. But $\langle \alpha_i, \alpha_i^\vee \rangle = 2$ for all simple roots $\alpha_i$ and $\langle \alpha_i, \alpha_j^\vee \rangle = -1$ if $i = j \pm 1$ and $\langle \alpha_i, \alpha_j^\vee \rangle = 0$ otherwise, so the inner product of roots in $U^P$ is non-negative.
As the remaining cases are a similar elementary verification, we omit the details.

To prove the converse, that the inner product condition on $U^{P_\omega}$ implies the weight is $\omega$ is braidless, we simply produce counterexamples. The roots in $U^P$ are characterized as those whose expression, as a linear combination of simple roots, involves the omitted simple root $\alpha$. If $P$ is not braidless, it is straightforward to construct inner products $\langle \gamma, \alpha^\vee \rangle = -1$ with $\gamma \in U^P$. For example, in type $B$ with omitted simple long root $\alpha_j$ with $j > 1$ (and necessarily with $j < n$)
$$ \langle \alpha_i + \cdots + \alpha_{j} + 2\alpha_{j+1} + \cdots + 2\alpha_n, \alpha_j^\vee \rangle = -1. $$
We leave the remaining counterexamples, easily produced from the tables of positive roots in the appendices of Bourbaki \cite{bourbaki}, to the reader.
\end{proof}

\begin{definition} A fundamental weight $\omega_i$ is said to be ``cominuscule'' if its associated simple root $\alpha_i$ appears with coefficient 1 in the expansion of the highest root. In keeping with earlier conventions, the associated maximal parabolic $P := P_{\omega_i}$ is said to be cominuscule if $\omega_i$ is cominuscule.
\label{comindef} \end{definition}

A simple case-by-case check according to Cartan type shows that cominuscule weights are a subset of braidless weights. Moreover,
cominuscule parabolics are precisely those whose unipotent radical is abelian, as may be readily checked by the Steinberg commutation
relations~(\ref{eecommutation}).

\begin{proposition} \label{doublecosetprop}
If $\Phi(G) \ne G_2$ and $P$ is braidless, then there exists a bijection between 
the double cosets in $P(F) \backslash P(F) U_{-}^P(F) / U_{-}^P(\o)$ and pairs of
integers $(c_j, d_j)$ with $\gcd(c_j, d_j)=1$, $d_j \ne 0$ and modulo units, and $c_j$ modulo $D_j$, 
where $D_j$ is an integer depending on the entries of $\iota(\gamma_k)$ with $k \geq j$. 
In particular if $P$ is cominuscule then
$$ D_j = d_j \prod_{\ell = j+1}^N d_{\ell}^{\langle \gamma_{j}, \gamma_{\ell}^\vee \rangle}. $$ 
\end{proposition}

\begin{proof} In light of Corollary~\ref{leftcosets}, it suffices to determine the effect of right multiplication by any element $u \in U_{-}^P(\o)$ on $\iota_{\gamma_1}(g_1) \cdots \iota_{\gamma_N}(g_N)$. Given any such $u$, factor it as
$$ u = e_{-\gamma_1}(t_1) \cdots e_{-\gamma_N}(t_N) \quad \text{with $t_j \in \o$}. $$ 
In order to perform the multiplication $\iota_{\gamma_1}(g_1) \cdots \iota_{\gamma_N}(g_N) u$, we break each $\iota_{\gamma_k}(g_k)$ into $h_{\gamma_k}, e_{\gamma_k},$ and $e_{-\gamma_k}$ as before and use the commutation relations (\ref{hecommutation}) and (\ref{eecommutation}) to move each $e_{-\gamma_j}(t_j)$ in $u$ leftward until it is adjacent to $\iota_{\gamma_j}(g_j)$. 

Caution is required here as the decomposition of $\iota_{\gamma_k}(g_k)$ into $h_{\gamma_k}, e_{\gamma_k},$ and $e_{-\gamma_k}$ is only defined over the field $F$. In particular, the commutation of $e_{-\gamma_j}(t_j)$ with $t_j \in \o$ and $h_{\gamma_k}(d_k^{-1})$ for certain pairs of positive roots $\gamma_j, \gamma_k$ in the unipotent radical of $P$ dilates the argument $t_j$ by an element of $\o$ only if $\langle \gamma_j, \gamma_k^\vee \rangle \geq 0$. This is guaranteed by Lemma~\ref{nonneg}.

The process of moving $e_{-\gamma_j}(t_j)$ leftward will produce linear changes of coordinates in $t_j$ and create many new elements $e_\alpha$ for $\alpha \in \Phi$ whose arguments depend on $t_j$ and the matrices $g_i$ with $i \in [j+1,N]$. It suffices to compute the dilation factor appearing in the linear change $t_j \mapsto t_j'$ once $e_{-\gamma_j}(t_j')$ is adjacent to $\iota_{\gamma_j}(g_j)$; as $t_j$ is arbitrary in $\mathfrak{o}$ then $c_j \in g_j$ may be restricted to a finite set determined by the size of the dilation. Upon fixing a choice of $t_j$, all of the $e_\alpha$ created by moving $e_{-\gamma_j}(t_j)$ leftward may be considered to have constant arguments in $t_j$ (now fixed) and the entries of the $g_i$. As we move subsequent $e_{-\gamma_{j'}}(t_{j'})$ leftward for $j' > j$, these $e_\alpha$'s involving $t_j$ produce harmless additive changes of variables in the $t_{j'}$ which are isomorphisms of $\mathfrak{o}$.

To determine the required dilation factor, we need only compute the effect of moving $e_{-\gamma_j}(t_j)$ past any $\iota_{\gamma_k}(g_k)$ as in (\ref{rankonehee}) with $k > j$. Then 
\begin{align} \iota_{\gamma_k}(g_k) e_{-\gamma_j}(t_j) & = h_{\gamma_k}(d_k^{-1}) e_{\gamma_k}(b_k d_k) e_{-\gamma_k}(c_k / d_k) e_{-\gamma_j}(t_j) \nonumber \\
& = h_{\gamma_k}(d_k^{-1}) e_{\gamma_k}(b_k d_k) e_{-\gamma_j}(t_j) \prod_{\gamma = a \gamma_k + b \gamma_j \in \Phi} e_{-\gamma}(\eta_{-\gamma_k, -\gamma_j; a, b} t_j^{b} (c_k / d_k)^a) e_{-\gamma_k}(c_k / d_k) \nonumber \\
& = e_{-\gamma_j}(d_k^{\langle \gamma_j, \gamma_k^\vee \rangle} t_j) h_{\gamma_k}(d_k^{-1}) \prod_{\gamma = a' \gamma_k - b' \gamma_j \in \Phi} e_{-\gamma}(\eta_{\gamma_k, -\gamma_j; a', b'} t_j^{b'} (b_k d_k)^{a'}) e_{\gamma_k}(b_k d_k) \times \nonumber \\
& \qquad \prod_{\gamma = a \gamma_k + b \gamma_j \in \Phi} e_{-\gamma}(\eta_{-\gamma_k, -\gamma_j; a, b} t_j^{b} (c_k / d_k)^a) e_{-\gamma_k}(c_k / d_k). \label{onecommutation}
\end{align}
In the leftmost product of (\ref{onecommutation}), which ranges over pairs $(a',b')$, $\gamma \ne -\gamma_j$ so it may be ignored. However, we still need to compute the commutator of $e_{\gamma_k}(b_k d_k)$ with terms in the rightmost product of (\ref{onecommutation}), as these may produce $e_{-\gamma_j}$'s. This can only happen for linear combinations $a \gamma_k + b \gamma_j$ with $b = 1$ since $\gamma_k$ and $\gamma_j$ are linearly independent.

Thus we obtain two cases: if $a \gamma_k + \gamma_j \in \Phi$ for some positive integer $a$ and $k > j$, then moving $e_{-\gamma_j}(t_j)$ past $\iota_{\gamma_k}(g_k)$ results in
$$ e_{-\gamma_j}(d_k^{\langle \gamma_j, \gamma_k^\vee \rangle} t_j (1 +\eta_{\gamma_k, -\gamma_k-\gamma_j; a, 1} \eta_{-\gamma_k, -\gamma_j; a, 1} b_k^a c_k^a) ). $$

If $a \gamma_k + \gamma_j \not\in \Phi$ for any $a$, then the dilation is simply
$$ e_{-\gamma_j}(d_k^{\langle \gamma_j, \gamma_k^\vee \rangle} t_j). $$ 
As $\gamma_k, \gamma_j$ are positive roots in the maximal parabolic $P$, they both contain the simple root $\alpha$ omitted from $P$. Thus the condition $a \gamma_k + \gamma_j \in \Phi$ is impossible if $P$ is cominuscule. Hence the simpler formula for $D_j$, $j=1, \ldots, N$ in these cases follows. 
 \end{proof}

\begin{corollary} \label{corgencase} If $\Phi(G) \ne G_2$ and $P$ is braidless, then the map 
\begin{eqnarray*}  \tilde{\varphi} : (B_1(\o) \backslash SL_2(\o) \cap B_1(F) B_1^{-}(F))^N & \longrightarrow & P(\o) \backslash G(\o) \cap P(F) U_{-}^P(F) \\
(g_1, \ldots, g_N) & \longmapsto & P(\o) \iota_{\gamma_1}(g_1) \cdots \iota_{\gamma_N}(g_N) \end{eqnarray*}
is a bijection of sets. Moreover, $U_{-}^P(\o)$ acts properly on the right of this quotient and the double cosets $P(\o) \backslash G(\o) \cap P(F) U_{-}^P(F) / U_{-}^P(\o)$ may be parametrized by
bottom rows $(c_j, d_j)$ of the $g_j$ with $\gcd(c_j,d_j)=1$, $d_j \ne 0$ and modulo units, and $c_j$ modulo $D_j$, with $D_j \in \o$ defined in the previous proposition.
\end{corollary}

\begin{proof} The map $\varphi$ in the previous theorem factors as $\varphi = \iota \circ \tilde{\varphi}$ where
$$ \iota: P(\o) \backslash G(\o) \cap P(F) U_{-}^P(F) \hookrightarrow P(F) \backslash P(F) U_{-}^P(F) $$
induced by the inclusion of $G(\o)$ in $G(F)$. The map is injective since, if $\iota(g_1) = \iota(g_2)$, then
$g_1 g_2^{-1} \in P(F) \cap G(\o) = P(\o)$. Thus $\tilde{\varphi}$ is onto since $\varphi$ is onto and $\iota$ is
injective. Moreover, $\tilde{\varphi}$ is injective since $\varphi$ is injective. The parametrization of the double cosets
follows immediately from Theorem~\ref{mainthm}. 
\end{proof}

\subsection{A remark about parabolics and coset representatives\label{cosetrepremark}}

The proof of Proposition~\ref{doublecosetprop} using Chevalley-Steinberg relations for determining double coset representatives for $P(\o) \backslash G(\o) \cap P(F) U_{-}^P(F) / U_{-}^P(\o)$ forced us to assume that the parabolic $P$ was cominuscule. In Corollary~\ref{leftcosets}, left coset representatives for $P(F) \backslash P(F) U_{-}^P(F)$ with integral entries were constructed using pairs $(c_i, d_i)$ of bottom row elements of $SL_2( \mathfrak{o})$ for $i=1, \ldots N$. It may be possible to refine this set of representatives to provide double coset representatives indexed by integers $c_i$ mod $D_i$ where $D_i$ is a product of $d_j$ with $j \in [1, N]$ using alternate methods.  

For example, the Klingen parabolic of $Sp_4$ (which omits the short root $\alpha_1$) is braidless, but not cominuscule. However, if one uses the theory of generalized Pl\"{u}cker coordinates (cf. \cite{fomin-zelevinsky}) in conjunction with the above recipe for embedded $SL_2$'s, then a set of representatives for the required double cosets is given in terms of the bottom row of $\varphi(g_1, g_2, g_3)$ in $Sp_4$ with $\varphi$ as in Corollary~\ref{leftcosets}. For the unique convex ordering, this bottom row is (in terms of the entries $(a_i, b_i, c_i, d_i)$ of $g_i$):
$$ (A_1, A_2, A_3, A_4) := (a_3c_2 d_1 - c_1 c_3, c_1 d_3 - b_3 c_2 d_1, c_3 d_1 d_2, d_1 d_2 d_3) $$
and by Pl\"{u}cker coordinates, the double cosets required are given by $A_i$ mod $A_4$ with $A_4$ non-zero. Clearly the condition that $A_3$ is mod $A_4$ is equivalent to taking $c_3$ mod $d_3$. Making a change of coordinates in $(A_1, A_2)$ by left-multiplying by $\begin{pmatrix} d_3 & c_3 \\ b_3 & a_3 \end{pmatrix}$ in $SL_2$, then $(A_1, A_2) \mapsto (c_2 d_1, c_1)$ and thus the double cosets may be parametrized by $c_i$'s modulo products of $d_j$'s.
The example here is simplified by the fact that there are no non-trivial Pl\"{u}cker relations.

\subsection{Group decompositions for metaplectic covers}

Corollary~\ref{corgencase} ensures that the method for evaluating metaplectic Whittaker functionals works for any braidless maximal parabolic $P$.
Indeed, there is a factorization of representatives for $P(\o) \backslash G(\o)$ in terms of embedded rank one elements which allows for simple computation of the Kubota homomorphism. Further, because this factorization yields explicit representatives for double cosets, we may perform the characteristic unfolding arguments to evaluate the Whittaker coefficient of the parabolic Eisenstein series. We do this in the context of the metaplectic group and with $\o = \o_S$ for a suitably chosen finite set of places $S$ as before. Recall that $(d,c)_S$ denotes the product of the $n$-th order local Hilbert symbols over $v \in S$ and $\mathbf{s}(g)$ is the section $g \mapsto (g,1)$. Finally, write $q_k$ for value of the quadratic form $Q(\gamma_k^\vee)$ associated to the bilinear form in the definition of $\widetilde{G}$ as in Theorem~\ref{localmet}.

\begin{proposition} \label{generalunipres} Let $g_1, \ldots, g_N \in SL_2( \o_S)$ be of the form
$$ g_i = \begin{pmatrix} a_i & b_i \\ c_i & d_i \end{pmatrix}, \quad d_i \ne 0, \quad \text{for $1 \leq i \leq N$.} $$
Then there exists a decomposition of the form:
$$ \mathbf{s}(\iota_{\gamma_1}(g_1)) \mathbf{s}(\iota_{\gamma_{2}}(g_{2})) \cdots \mathbf{s}(\iota_{\gamma_N}(g_N)) = \prod_k (d_k, c_k)_S^{q_k} \, v^P
\tilde{\mathfrak{D}} \, v $$
with $v^P \in U^P(F_S)$,  $v \in U_{-}(F_S)$, and $\tilde{\mathfrak{D}} \in \widetilde{T}_{F_S}$ given by
$$ \tilde{\mathfrak{D}} = \mathbf{s}(h_{\gamma_1}(d_1^{-1})) \cdots  \mathbf{s}(h_{\gamma_N}(d_N^{-1})). $$
Moreover, when $-\alpha_j$ is in the parabolic $P$, the canonical projection of $v \in U_{-}(F_S)$ to the unipotent subgroup $U_{-\alpha_j}$ is $e_{-\alpha_j}(v_j)$ with
\begin{equation} v_{j} = \sum_{(k,k') \in \mathcal{S}_j} \left[ (-1)^{i+i'} \eta_{i,i',k,-k'} \left( b_k d_k^{-1} \right)^i \left( c_{k'} d_{k'}^{-1} \right)^{i'} \prod_{\ell \geqslant k} (d_{\ell}^{-1})^{\langle \alpha_j, \gamma_\ell^\vee \rangle} \prod_{k' < \ell < k} (d_{\ell}^{-1})^{i' \langle \gamma_{k'}, \gamma_{\ell}^\vee \rangle} \right] \label{unipelt} \end{equation} 
where $\mathcal{S}_j$ denotes the set of all pairs $(k,k')$ with $k > k'$ such that $i \gamma_k - i' \gamma_{k'} = -\alpha_j$ for positive integers $i, i'$. Here, $\eta_{i,i',k,-k'} := \eta_{i,i',\gamma_k,-\gamma_k'}$ are the integer coefficients that appear in the Steinberg commutation relation (\ref{eecommutation}). The projection of $v$ to the unipotent subgroup $U_{-\alpha_j}$ with $-\alpha_j$ not in $P$ is equal to $e_{-\alpha_j}(c_N / d_N)$.
\end{proposition}

\begin{remark} \label{canremark} This result applies to all maximal parabolics $P$. If $P$ is assumed braidless, any reduced decomposition for $w^P$ is unique up to exchange of orthogonal reflections (cf. Lemma 3.2 of \cite{littelmann}), whose corresponding embedded rank one subgroups commute. Thus in all of the above results, the right-hand side of the factorization identity depends only on the choice of parabolic $P$, and not on the reduced decomposition of $w^P$.
\end{remark}

\begin{proof} The existence of the decomposition over reductive groups was proved in the previous section. In order to perform it over the metaplectic group, we need only keep track of the cocycles, which enter as follows. In decomposing each $\mathbf{s} (\iota_{\gamma_k}(g_k))$ according to a rank one Bruhat decomposition, we obtain
\begin{equation} \mathbf{s}(\iota_{\gamma_k}(g_k)) = (d_k, c_k)_S^{q_k} \, \mathbf{s}(e_{\gamma_k}(b_k / d_k)) \mathbf{s}(h_{\gamma_k}(d_k^{-1})) \mathbf{s}(e_{-\gamma_k}(c_k / d_k)). \label{rankonefact} \end{equation}
We move $e_{-\gamma_{k'}}$ rightward past the pieces of $\iota_{\gamma_{k}}$ where $k > k'$. If we wish to keep track of the terms $v_j$ appearing in $v \in U_{-}(F_S)$, then there are three sources. First,
moving the unipotent $e_{-\gamma_{k'}}(x)$ rightward past $h_{\gamma_\ell}(d_\ell^{-1})$ with $k' < \ell < k$ results in 
\begin{equation} e_{-\gamma_{k'}}(x d_\ell^{- \langle \gamma_{k'}, \gamma_\ell^\vee \rangle}) \label{firstcomms} \end{equation}
according to (\ref{hecommutation}). Next, we move the unipotent in (\ref{firstcomms}) rightward, past $\iota_{\gamma_k}(g_k)$ for a pair $(k,k')$ in $\mathcal{S}_j$ and
record when $i \gamma_k - i' \gamma_{k'} = -\alpha_j$ for a simple root $\alpha_j$ (as the associated commutation relations (\ref{eecommutation})
produce a term $e_{-\alpha_j}$ whose argument includes that of (\ref{firstcomms}) to the power $i'$). Then we move each of the unipotent elements 
$e_{-\alpha_j}(x)$ created by a commutation of $\gamma_k$ and $\gamma_{k'}$ rightward past $h_{\gamma_\ell}(d_\ell^{-1})$ with index $\ell \geqslant k$,
which multiplies $x$ by $(d_\ell^{-1})^{\langle \alpha_j, \gamma_\ell^\vee \rangle}$ according to (\ref{hecommutation}).

As for the resulting metaplectic cocycle, the individual Bruhat decompositions produce $S$-Hilbert symbols $(d_k, c_k)_S^{q_k}$
for each $g_k$, $k = 1, \ldots, N$. According to \cite{matsumoto} (or Section~6 of \cite{steinberg}) the only other nontrivial cocycles arise from multiplications of diagonal elements $h_{\gamma_k}$, so the result follows. \end{proof}

\section{Twisted Multiplicativity}\label{Twisted-Multiplicativity}

Let $\psi: F_S \longrightarrow \mathbb{C}$ be a character of $F_S$ that is trivial on $\mathfrak{o}_S$ but no larger fractional ideal. Given an element $\boldsymbol{t} = (t_1, \ldots, t_r) \in (\mathfrak{o}_S - \{0\})^r$, let $\psi_{\boldsymbol{t}}$ be the character of $U(F_S)$ such that $\psi_{\boldsymbol{t}} (w_0e_{-\alpha_j}(x)w_0^{-1})= \psi(t_j x)$ for $x \in F_S$ and $j=1, \ldots, r$.

In Theorem~\ref{whittakeratlast}, we showed that for special choice of test vector, the $\psi_{\boldsymbol{t}}$-Whittaker coefficient is expressible in the form
\begin{equation} \mathcal{W}_{f_1,f_2,s}(1) \sum_{\boldsymbol{d} \in (\mathfrak{o}_S - \{0\} / \mathfrak{o}_S^\times)^N} H(\boldsymbol{d}; \boldsymbol{t})\delta_P^{s+1/2}(\mathfrak{D}) \Psi(\mathfrak{D}) \zeta_\mathfrak{D} \, c_{f_1, f_2}^{\psi_{\boldsymbol{t}}}(\mathfrak{D}). \label{mdsform} \end{equation}
Here we have used $\boldsymbol{d}$ to denote the $N$-tuple $(d_1, \ldots, d_N)$ with $d_i \in \mathfrak{o}_S / \mathfrak{o}_S^\times$ and non-zero. We have also written $H(\boldsymbol{d}; \boldsymbol{t})$ in place of $H(\boldsymbol{d})$ to emphasize the dependence on the character
$\psi_{\boldsymbol{t}}$. Thus explicitly, we may write $H(\boldsymbol{d}; \boldsymbol{t})$ using the definition in~(\ref{hdefined}) as
\begin{equation} H(\boldsymbol{d}; \boldsymbol{t}) := \sum_{c_j (\text{mod } D_j)} \psi(\sum_j t_j v_j) \prod_{k=1}^N \left( \frac{c_k}{d_k} \right)^{\!\! q_k} \label{hdtdefined} \end{equation}
with $v_j$ as defined in~(\ref{unipelt}) and $D_j$ as in Proposition~\ref{doublecosetprop}. Implicit here is that summands are 0 unless the following divisibility condition holds.

\begin{lemma}To any $\boldsymbol{t} = (t_1, \ldots, t_r) \in (\mathfrak{o}_S - \{0\})^r$, the coefficient $H(\boldsymbol{d}; \boldsymbol{t})$ vanishes unless, for each simple root $\alpha_j \in P$,
$$ t_j \prod_{i=1}^N d_i^{- \langle \alpha_j, \gamma_i^\vee \rangle} \in \mathfrak{o}_S. $$ 
\label{whitdivis} \end{lemma}

\begin{proof} As noted in the discussion below (\ref{mdsform-1}), the Whittaker integral vanishes unless the character $\psi_{\boldsymbol{t}}$ is trivial on all elements of the form 
$$ (\mathfrak{D} w_0)^{-1} p (\mathfrak{D} w_0) \quad \text{with $p \in P(\mathfrak{o}_S) \cap U_{-}(F_S)$.} $$
Here, $\mathfrak{D} = h_{\gamma_1}(d_1^{-1}) \cdots h_{\gamma_N}(d_N^{-1})$ according to Proposition~\ref{generalunipres}.
Thus it suffices to check the condition for $p = e_{-\alpha_j}(x)$ with $\alpha_j$ a simple root in the the parabolic $P$ and any $x \in \mathfrak{o}_S$. The result now follows from repeated application of (\ref{hecommutation}).
\end{proof}

The results of Section~\ref{sectionMDT} imply that the exponential sum $H(\boldsymbol{d}; \boldsymbol{t})$ has a particularly nice form when the maximal parabolic $P$ is cominuscule. In these cases, we now demonstrate that $H$ is multiplicative in both $\boldsymbol{d}$ and $\boldsymbol{t}$ up to an explicitly prescribed $n$-th root of unity. Collectively, these two properties are commonly referred to as ``twisted multiplicativity.'' 

Let $\omega_i^\vee$ denote the $i$-th fundamental coweight, so that
$$ \langle \alpha_j, \omega_i^\vee \rangle = \delta_{i,j}. $$

\begin{proposition} Let $P$ be a cominuscule parabolic. Given any $\boldsymbol{d} = (d_1, \ldots, d_N)$ with $d_i \in \mathfrak{o}_S / \mathfrak{o}_S^\times$ and non-zero, write $\boldsymbol{t} = (t_1 t_1', \ldots, t_r t_r')$ such that $\gcd(\prod_{i=1}^r t_i, \prod_{j=1}^{N} d_j ) = 1$. Then
$$ H(\boldsymbol{d}; \boldsymbol{t}) = \prod_{k=1}^N \prod_{i=1}^r \left( \frac{t_i^{-\langle \gamma_k, \omega_i^\vee \rangle}}{d_k} \right)^{\!\! q_k} H(\boldsymbol{d}; t_1', \ldots, t_r'). $$
\end{proposition}

\begin{proof}
According to~(\ref{hdtdefined}), if $\boldsymbol{t} = (t_1 t_1', \ldots, t_r t_r')$, then
\begin{equation} H(\boldsymbol{d}; \boldsymbol{t}) = \sum_{c_i} \left[ \prod_{i=1}^N \left( \frac{c_i}{d_i} \right)^{\!\! q_i} \right] \psi \left( \sum_{j=1}^r t_j t_j' v_j \right) \label{expreminder} \end{equation}
where, because $P$ is assumed cominuscule, the sum is over $c_i$ modulo 
\begin{equation} D_i := d_i \prod_{\ell = i+1}^N d_{\ell}^{\langle \gamma_i, \gamma_{\ell}^\vee \rangle} \label{Dkdef} \end{equation} 
for $i=1, \ldots, N$. Moreover, $v_j \in F_S$ with $j = 1, \ldots, r$ is the element appearing in~(\ref{unipelt}).

The result will follow by applying a change of variables
$$ c_i \longmapsto c_i \prod_{j=1}^r t_j^{-\langle \gamma_i, \omega_j^\vee \rangle}. $$
Indeed, the change produces the residue symbols in the statement of the proposition and is an automorphism of residue classes mod $D_i$ according to the assumed relative primality of $t_j$'s and $d_i$'s. So it suffices to show this change of variables eliminates the dependence on $t_i$ appearing in the argument of the character $\psi$ in (\ref{expreminder}).

As $b_k$ is a multiplicative inverse of $c_k$ (mod $d_k$), then to any index $j_0$ and pairs $(k, k')$ in $\mathcal{S}_{j_0}$ (whose definition is given in Proposition~\ref{generalunipres}):
$$ b_k^{i} c_{k'}^{i'} \longmapsto b_k^{i} c_{k'}^{i'} \prod_{j=1}^r t_j^{-\langle -\gamma_k \cdot i + \gamma_k' \cdot i', \omega_j^\vee \rangle} = b_k^{i} c_{k'}^{i'} t_{j_0}^{-1} $$
where the last equality simply follows from the definition of $\mathcal{S}_{j_0}$. This cancels the $t_{j_0}$ appearing in $\psi$ for each $j_0 \in [1,\ell]$ and the result follows. \end{proof}

Twisted multiplicativity for $H(\boldsymbol{d}; \boldsymbol{t})$ with respect to $\boldsymbol{d}$ will follow from the Chinese remainder theorem, provided one can demonstrate that the moduli $D_i$ for integers $c_i$ in the exponential sum satisfy a precise relationship with the conductors of the additive characters in the sum. This relationship is the content of the following result.

\begin{lemma} \label{twistedmultlemma} For each
simple root $\alpha_j$ in the maximal parabolic $P$, and a pair $(k, k') \in \mathcal{S}_j$ (i.e., with $k > k'$ and $i \gamma_k - i' \gamma_{k'} = - \alpha_j$ for positive integers $i, i'$), set
$$ D(k, k'; \alpha_j) := d_k^i d_{k'}^{i'} \prod_{\ell \geqslant k} d_{\ell}^{\langle \alpha_j, \gamma_\ell^\vee \rangle} \prod_{k' < \ell < k} d_{\ell}^{i' \langle \gamma_{k'}, \gamma_{\ell}^\vee \rangle}, $$
the product of $d$'s appearing in the denominator of the summand of $v_j$ corresponding to the pair $(k,k')$ as in (\ref{unipelt}). Then
$$ D(k, k'; \alpha_j) = \frac{D_{k'}^{i'}}{D_k^i}, $$
where
$$ D_k := d_k \prod_{\ell = k+1}^N d_{\ell}^{\langle \gamma_{k}, \gamma_\ell^\vee \rangle}, $$
the modulus of $c_k$ ($k = 1, \ldots, N$) in the exponential sum $H(\boldsymbol{d}; \boldsymbol{t})$.
\end{lemma}

\begin{proof} This follows by straightforward calculation of $D_{k'}^{i'} / D_k^i$. The only term that requires special care is $d_k$, where $D_{k'}^{i'}$ contributes $d_k^{i' \langle \gamma_{k'}, \gamma_{k}^\vee \rangle}$ and $D_k^i$ contributes $d_k^i$. But
$$ i' \langle \gamma_{k'}, \gamma_{k}^\vee \rangle - i = \langle \alpha_j + i \gamma_k, \gamma_{k}^\vee \rangle - i = \langle \alpha_{j}, \gamma_{k}^\vee \rangle + i, $$
since $\langle \alpha, \alpha^\vee \rangle = 2$ for all roots $\alpha$. This matches the power of $d_k$ appearing in $D(k, k'; \alpha_j)$.
\end{proof}

The result is independent of any restrictions on the parabolic $P$ (e.g., cominuscule), so it could be used, for example, to prove twisted multiplicativity statements in the case discussed in Section~\ref{cosetrepremark}.

\begin{proposition} \label{isareone} Let $P$ be cominuscule with roots $\gamma_k \in U^P$ as above. If, for $k > k'$, there exist positive integers $i, i'$ such that
$$ i \gamma_k - i' \gamma_{k'} = -\alpha_j, \quad \text{for some simple root $\alpha_j \in P$}, $$
then $i = i' = 1$.
\end{proposition}

\begin{proof} Express $\gamma_k$ and $\gamma_{k'}$ as linear combinations of simple roots:
$$ \gamma_k = \sum_\ell c_\ell \alpha_\ell, \quad \gamma_{k'} = \sum_\ell c_\ell' \alpha_\ell. $$ 
As both $\gamma_k, \gamma_{k'}$ are in $U^P$ with $P$ cominuscule, both sums contain the omitted root $\alpha$ with coefficient 1. As $\alpha \ne \alpha_j$ this forces $i = i'$. But then $i = i' = 1$, else there is no pair of positive integers $c_j, c_{j}'$ such that $i (c_j - c_j') = -1$.
\end{proof}

With these results in hand, we are ready to prove the twisted multiplicativity of the exponential sum $H$ with respect to $\boldsymbol{d}$.

\begin{theorem} Let $P$ be a cominuscule parabolic. Given $\boldsymbol{d} = (d_1, \ldots, d_N)$ with each $d_i = e_i f_i$ such that $\gcd(e_1 \cdots e_N, f_1 \cdots f_N) = 1$, define
$$ E_k := e_k \prod_{\ell = k+1}^N e_{\ell}^{\langle \gamma_{k}, \gamma_\ell^\vee \rangle}, \quad F_k := f_k \prod_{\ell = k+1}^N f_{\ell}^{\langle \gamma_{k}, \gamma_\ell^\vee \rangle} $$
in analogy with $D_k$ in (\ref{Dkdef}). Then the exponential sum $H(\boldsymbol{d}; \boldsymbol{t})$ factors as:
\begin{equation} H(\boldsymbol{d}; \boldsymbol{t}) = \left( \frac{E_k}{f_k} \right)^{\!\! q_k} \left( \frac{F_k}{e_k} \right)^{\!\! q_k} H(e_1, \ldots, e_N; \boldsymbol{t}) H(f_1, \ldots f_N; \boldsymbol{t}) . \label{dtwistedmult} \end{equation}
\end{theorem}

\begin{proof} The result follows by use of the Chinese remainder theorem. Recall that $c_k$ runs mod $D_k$ as in (\ref{Dkdef}). Let us define $E_k$ and $F_k$ similarly, with $d_\ell$'s replaced by $e_\ell$'s and $f_\ell$'s, respectively, in (\ref{Dkdef}) so that $D_k = E_k F_k.$ Thus the $c_k$ may be reparametrized writing
\begin{equation} c_k = x_k E_k + y_k F_k \quad \text{with $x_k$ (mod $F_k$) and $y_k$ (mod $E_k$)}. \label{reparam} \end{equation}
Recall that $d_k || D_k$ and hence by definition $e_k || E_k$ and $f_k || F_k$, so that the residue symbols appearing in $H(\boldsymbol{d}; \boldsymbol{t})$ may be rewritten:
$$ \left( \frac{c_k}{d_k} \right) = \left( \frac{x_k E_k + y_k F_k}{e_k f_k} \right) = \left( \frac{x_k E_k}{f_k} \right) \left( \frac{y_k F_k}{e_k} \right) = \left[ \left( \frac{E_k}{f_k} \right) \left( \frac{F_k}{e_k} \right) \right] \left( \frac{x_k}{f_k} \right) \left( \frac{y_k}{e_k} \right).  $$
Thus the twisted multiplicativity in (\ref{dtwistedmult}) will follow if the additive character $ \psi \left( \sum_{j=1}^r t_j v_j \right) $ appearing in $H$ factors neatly into a product of characters for any choice of $c_k$'s in the sum $H$. There are two cases to consider. 

\medskip

\noindent {\bf Case 1:} If $j$ is the index of the omitted simple root, then $v_j$ is simply $c_N / d_N$. And
$$ \psi \left( t_j \frac{c_N}{d_N} \right) = \psi \left( t_j \frac{x_N E_N + y_N F_N}{e_N f_N} \right) = \psi \left( t_j \frac{x_N}{f_N} \right) \psi \left( t_j \frac{y_N}{e_N} \right), $$
where we have used that $E_N = e_N$ and $F_N = f_N$ in the last equality, as the products in the definition of $E_N$ and $F_N$ are empty.

\medskip

\noindent {\bf Case 2:} If $j$ is not the index of the omitted simple root, recall that $v_j$ is as in (\ref{unipelt}) and each summand in $v_j$ is (up to an inconsequential integral structure constant $\pm \eta$) of form $b_k^i c_{k'}^{i'} / D(k,k';\alpha_j)$ for a pair of indices $(k,k') \in \mathcal{S}_j$. Moreover since $P$ is cominuscule, by Proposition~\ref{isareone}, $i = i' =1$ for each pair $(k,k') \in \mathcal{S}_j$. In $H(\boldsymbol{d}; \boldsymbol{t})$, the $c_k$ are then summed mod $D_k$ and the $b_k$ are multiplicative inverses of the $c_k$ (mod $d_k$). Thus to prove the theorem, it suffices to show for each $j$ and each pair $(k, k') \in \mathcal{S}_j$:
\begin{equation} \psi \left( t_j \frac{b_k c_{k'}}{D(k,k';\alpha_j)} \right) = \psi \left( t_j \frac{\bar{x}_k x_{k'}}{F(k,k';\alpha_j)} \right) \psi \left( t_j \frac{\bar{y}_k y_{k'}}{E(k,k';\alpha_j)} \right), \label{psifactor} \end{equation}
where $c_k$ is parametrized as in (\ref{reparam}), $\bar{x}_k$ and $\bar{y}_k$ are inverses of the $x_k$ and $y_k$ mod $f_k$ and $e_k$, respectively, and $E(k,k';\alpha_j)$ and $F(k,k';\alpha_j)$ are defined analogously to $D(k,k';\alpha_j)$ with $d_j$'s replaced by $e_j$'s and $f_j$'s, respectively. 

To prove (\ref{psifactor}), expand:
$$ \frac{b_k c_{k'}}{D(k,k';\alpha_j)} = \frac{b_k (x_{k'} E_{k'} + y_{k'} F_{k'})}{E(k,k';\alpha_j) F(k,k';\alpha_j)} =  \frac{b_k x_{k'} E_k}{F(k,k';\alpha_j)} + \frac{b_k y_{k'} F_k}{E(k,k';\alpha_j)}, $$
where we have used Lemma~\ref{twistedmultlemma} with $(i, i') = (1,1)$ in the second equality. As $b_k$ satisfies $b_k x_k E_k \equiv 1$ (mod $f_k$), then we may rewrite
$$ \frac{b_k x_{k'} E_k}{F(k,k';\alpha_j)} = \frac{\bar{x}_k x_{k'}}{F(k,k';\alpha_j)}, $$
so the rational expression involving $x_{k'}$ gives one of the desired factors on the right-hand side of (\ref{psifactor}). Similarly, the expression involving $y_{k'}$ gives the other factor. 
\end{proof}

These two twisted multiplicativity properties of the exponential sum $H(\boldsymbol{d}; \boldsymbol{t})$ appearing in (\ref{mdsform}) imply 
that it suffices to determine $H$ when the components of $\boldsymbol{d}$ and $\boldsymbol{t}$ are powers of any fixed prime $p \in \mathfrak{o}_S.$ 
Given such a prime and non-negative integer tuples $\boldsymbol{\ell} = (\ell_1, \ldots, \ell_N)$ and $\boldsymbol{m} = (m_1, \ldots, m_r)$, define
\begin{equation} S_{\boldsymbol{\ell},\boldsymbol{m}} := H(p^{\ell_1}, \ldots, p^{\ell_N}; p^{m_1}, \ldots, p^{m_r}). \label{slmgeneral} \end{equation}

Before making general remarks about the sums $S_{\boldsymbol{\ell},\boldsymbol{m}}$ in Sections~\ref{canbases} and~\ref{genericevalsection}, we discuss an example to be used throughout the remainder of the paper.

\section{An example in $\widetilde{GL}_4$}

We now write the exponential sum explicitly in one of the first non-minimal cases: $\widetilde{G} = \widetilde{GL}_4$ and $P = MU$ with $M = GL_2 \times GL_2$. The metaplectic cover $\widetilde{G}$ is chosen to satisfy the conditions of Theorem~\ref{localmet} with $Q(\alpha^\vee) = 1$ for all roots $\alpha$. This choice does not uniquely determine $\widetilde{G}$ (see for example Chapter~0 of~\cite{kazhdan-patterson}) but the resulting exponential sum is independent of this choice. Lastly, the structure constants $\eta_{\alpha, \beta; i,j}$ in (\ref{eecommutation}) used for the realization of $GL_4$ may be taken to be $\pm 1$, as usual. The example presented here will be used repeatedly to illustrate more general phenomena discussed in the subsequent sections.

In order to apply the group decomposition theorems from Section~\ref{sectionMDT}, we choose a reduced expression for $w_0$ respecting the factorization $w_0 = w_M w^P$, where $w_M$ is the long element for the Levi subgroup $M$. Thus we may take $w_0 = s_1 s_3 s_2 s_1 s_3 s_2$ with $w^P = s_2 s_1 s_3 s_2$.
This gives rise to the ordering of positive roots in $U^P$ which are indexed according to~(\ref{phip}) by
$$ (\gamma_1, \gamma_2, \gamma_3, \gamma_4) = (\alpha_1 + \alpha_2 + \alpha_3, \alpha_2 + \alpha_3, \alpha_1 + \alpha_2, \alpha_2), $$
where, as usual, $\alpha_j$ denotes the simple positive root at position $(j, j+1)$ in $GL_4$.

Given $t_1,t_2,t_3\in \o_S$ nonzero, the Whittaker coefficient in (\ref{mdsform}) has the form
$$ \mathcal{W}_{f_1,f_2,s}(1) \sum_{\substack{d_1, d_2, d_3, d_4 \in \mathfrak{o}_S / \mathfrak{o}_S^\times \\ d_j \ne 0}} (d_2, d_1)_S (d_2 d_4, d_3)_S H(\boldsymbol{d}; \boldsymbol{t}) \Psi(\mathfrak{D}) c_{f_1, f_2}^{\psi_{\boldsymbol{t}}}(\mathfrak{D})  | d_1 d_2 d_3 d_4 |^{-(1+2s)}. $$
Here 
\begin{equation} H(\boldsymbol{d}; \boldsymbol{t}) = \sum_{c_1, c_2, c_3, c_4} \left[ \prod_{k=1}^4 \left( \frac{c_k}{d_k} \right) \right] \psi\left(-t_1 \left(\frac{b_2c_1d_4}{d_1 d_3}+\frac{b_4c_3}{d_3}\right)
+t_2 \frac{c_4}{d_4}+ t_3 \left(\frac{c_1b_3d_4}{d_1 d_2}+\frac{c_2b_4}{d_2}\right) \right) \label{specialized} 
 \end{equation}
with the sum over
$$ c_1 \; (\bmod{~d_1 d_2 d_3}), ~c_2 \; (\bmod{~d_2 d_4}), ~c_3 \; (\bmod{~d_3 d_4}), ~c_4 \; (\bmod{~d_4}), $$
and the $b_i$ are multiplicative inverses of $c_i$ for their respective moduli.
This sum arises only when the divisibility conditions of Lemma~\ref{whitdivis} are satisfied; in this case they are
\begin{equation}\label{divisibility}d_1d_3\mid t_1 d_2 d_4,\quad d_1d_2\mid t_3 d_3 d_4.
\end{equation}
We assume them henceforth.

The sum satisfies the twisted multiplicativity conditions of the previous section and so it suffices to compute the value of $H(\boldsymbol{d};\boldsymbol{t})$ when the parameters $d_i$ and $t_i$
are powers of a fixed prime $p$ in $\mathfrak{o}_S$. Let $q$ denote the cardinality of $\mathfrak{o}_S / p \mathfrak{o}_S$.

Let $\boldsymbol{\ell}=(\ell_1,\ell_2,\ell_3,\ell_4)$ and $\boldsymbol{m}=(m_1,m_2,m_3)$ be vectors of non-negative
integers. Then in the notation of (\ref{slmgeneral}), $S_{\boldsymbol{\ell},\boldsymbol{m}} = H((p^{\ell_1}, p^{\ell_2}, p^{\ell_3}, p^{\ell_4}); (p^{m_1}, p^{m_2}, p^{m_3}))$
where
\begin{multline}
S_{\boldsymbol{\ell},\boldsymbol{m}}:= q^{2\ell_4} \sum \left(\frac{c_1}{p^{\ell_1}}\right)\left(\frac{c_2}{p^{\ell_2}}\right)\left(\frac{c_3}{p^{\ell_3}}\right)
\left(\frac{c_4}{p^{\ell_4}}\right)\\
\psi\left(-p^{m_1}\left(\frac{b_2c_1p^{\ell_4}}{p^{\ell_1+\ell_3}}+\frac{b_4c_3}{p^{\ell_3}}\right)
+p^{m_2}\frac{c_4}{p^{\ell_4}}+p^{m_3}\left(\frac{c_1b_3p^{\ell_4}}{p^{\ell_1+\ell_2}}+\frac{c_2b_4}{p^{\ell_2}}\right)
\right). \label{slmdefined}
\end{multline}
Here we sum over $c_i$ modulo $p^{\ell_i}$ for $i=2,3,4$ and $c_1$ modulo $p^{\ell_1+\ell_2+\ell_3}$, 
with $c_i$ prime to $p$ if $\ell_i>0$ and no such condition if $\ell_i=0$; $b_i$  
satisfies the congruence $b_ic_i\equiv 1\pmod {p^{\ell_i}}$ for $i=2,3,4$.
The divisibility conditions become
\begin{align}
\ell_1+\ell_3&\leq m_1+\ell_2+\ell_4\label{div1}\\
\ell_1+\ell_2&\leq m_3+\ell_3+\ell_4.\label{div2}
\end{align}

Strictly speaking, (\ref{specialized}) has $c_j$ modulo $p^{\ell_j+\ell_4}$ for $j=2,3$ but using the above 
divisibility conditions we see that the sum depends only on $c_j$ modulo $p^{\ell_j}$
for $j=2,3$, so the support is unchanged by viewing $c_j$ modulo $p^{\ell_j}$.
Moreover, one could begin with the original sum with $c_j$ modulo $p^{\ell_j+\ell_4}$ for $j=2,3$
and observe that the sum is zero unless
the divisibility conditions hold.  For example, if $\ell_4>0$ then
changing $c_2$ to $c_2+ap^{\ell_2}$ where $a$ runs
modulo $p^{\ell_4}$ shows that the sum is zero unless
$\ell_1+\ell_3\leq m_1+\ell_2+\ell_4$.  Similarly, changing $c_3$ to $c_3+ap^{\ell_3}$ gives
the divisibility condition $\ell_1+\ell_2\leq m_3+\ell_3+\ell_4$.  If instead $\ell_4=0$ then we may take $c_4=b_4=0$
and the desired inequalities are easy to obtain.

\section{Canonical bases and the exponential sum \label{canbases}}

We return to the case of arbitrary (i.e.\ not necessarily cominuscule) parabolics $P$ and a reduced word $\boldsymbol{i}^P$ for $w^P$. Recall that for a fixed prime $p \in \mathfrak{o}_S$, in (\ref{slmgeneral}) we set
$$ S_{\boldsymbol{\ell},\boldsymbol{m}} := H(p^{\ell_1}, \ldots, p^{\ell_N}; p^{m_1}, \ldots, p^{m_r}), $$
with $H$ the exponential sum appearing in the Whittaker coefficient as in~(\ref{mdsform}).
Two basic questions immediately arise:
\begin{enumerate}
\item For which $\boldsymbol{\ell} \in \mathbb{Z}_{\geq 0}^N$ is $S_{\boldsymbol{\ell}, \boldsymbol{m}} \ne 0$?
\item Can we give an evaluation of $S_{\boldsymbol{\ell}, \boldsymbol{m}}$,
for any choice of $\boldsymbol{\ell}$ and $\boldsymbol{m}$, in terms of representation theoretic data on the Langlands dual group (independent of $p$)?
\end{enumerate}

In Section~\ref{connectthree}, the integers $\boldsymbol{\ell}$ are connected to the $\boldsymbol{i}^P$-Lusztig data for canonical basis elements on the dual group $G^\vee$ corresponding the choice of reduced word $\boldsymbol{i}^P$. (This terminology is recalled in Section~\ref{ilusztig}.) With this connection, answers to both of these questions  may be formulated in terms of representation theory and canonical bases for the dual group. Roughly, we expect the answer to the first question is that the set of $\boldsymbol{\ell}$'s for which $S_{\boldsymbol{\ell}, \boldsymbol{m}} \ne 0$ is {\it almost}
the set of Lusztig data for the highest weight representation of $G^\vee(\mathbb{C})$ of highest weight $\boldsymbol{m} + \rho$, identifying integer $r$-tuples $\boldsymbol{m}$ with elements of the weight lattice, with $\rho$ the Weyl vector of the dual group $G^\vee$. Here ``almost'' means that the statement is true up to a set of $\boldsymbol{\ell}$'s lying on a hyperplane of strictly smaller dimension than $N$. 
For the second question, we expect the general answer is expressed using the other most important parametrization of canonical bases -- Kashiwara's string data (whose definition is reviewed in Section~\ref{stringdata}). Indeed, for almost all $\boldsymbol{\ell}$, we expect that $S_{\boldsymbol{\ell}, \boldsymbol{m}}$ may be expressed as products of $N$ $n$-th order Gauss sums whose moduli are given by the Kashiwara string data of the corresponding canonical basis elements for the $G^\vee$ module with highest weight $\boldsymbol{m}+\rho$. Again ``almost all'' means that modifications must be made at certain degenerate points of the string polytope. Precise formulations of these results are given in the subsequent sections of the paper.

\subsection{The $\boldsymbol{i}$-Lusztig data\label{ilusztig}}

Let $\mathcal{U} := \mathcal{U}_q(\mathfrak{g})$ be the quantized universal enveloping algebra associated to a semisimple Lie algebra $\mathfrak{g}$ by Drinfeld and Jimbo. The algebra $\mathcal{U}$ has a presentation in terms of generators $E_j$, $K_j^{\pm}$, and $F_j$ for $j = 1, \ldots, r$ where $r$ is the rank of $\mathfrak{g}$. See for example~\cite{lusztig-book}, Section~3.1.1, (a)--(d), for their relations.
Let $\mathcal{U}^+$ denote the subalgebra generated by the $E_j$.

Lusztig describes bases for $\mathcal{U}^+$ in terms of certain algebra automorphisms $T_j$ on $\mathcal{U}_q(\mathfrak{g})$ for $j = 1, \ldots, r$. Their precise definition is given in \cite{lusztig-book}, Section~5.2 and Chapter~37, where they are called $T_{j, -1}'$. The $T_j$ satisfy the braid relations, so extend to an action of the braid group of $\mathcal{U}$. Moreover, if $s_{i_1} \cdots s_{i_{k}}$ is a reduced expression in $W$, the Weyl group of $\mathfrak{g}$, satisfying
$$ s_{i_1} \cdots s_{i_{k-1}}(\alpha_{i_k}) = \alpha_j, \qquad \alpha_j \; : \; \text{simple root}, $$
then
$$ T_{i_1} T_{i_2} \cdots T_{i_{k-1}} (E_{i_k}) = E_j. $$

Let $\boldsymbol{i} = (i_1, \ldots, i_N)$ be a word such that $s_{i_1} \cdots s_{i_N}$ is a reduced expression for the long element $w_0 \in W$. To each such $\boldsymbol{i}$, Lusztig associates a $\mathbb{C}(q)$-basis $\mathcal{B}_{\boldsymbol{i}}$ of $\mathcal{U}^+$ consisting of elements
\begin{equation} \left\{ p_{\boldsymbol{i}}^{\mathbf{c}} := E_{i_1}^{(c_1)} T_{i_1}(E_{i_2}^{(c_2)}) \cdots (T_{i_1} T_{i_2} \cdots {T}_{i_{N-1}})(E_{i_N}^{(c_N)}) \; \left| \; \mathbf{c} \in \mathbb{Z}_{\geq 0}^N \right. \right\}. \label{lbasiselt} \end{equation}
See Chapter~40 of \cite{lusztig-book} for a proof that these elements lie in $\mathcal{U}^+$ and form a basis.

One means of describing the canonical basis $\mathcal{B}$ of $\mathcal{U}^+$ is via a graph structure on the set of
pairs $(\boldsymbol{i}, \mathbf{c})$ as $\boldsymbol{i}$ ranges over all reduced decompositions of the long element $w_0$
and $\mathbf{c}$ runs over $\mathbb{Z}_{\geq 0}^N$ corresponding to elements of $\mathcal{B}_{\boldsymbol{i}}$ as in (\ref{lbasiselt}).
This is done in two steps. First, a preliminary graph structure may be placed on the set of all reduced decompositions. The words $\boldsymbol{i}$ and $\boldsymbol{i}'$ are joined by an edge if the two words differ by a single application of the braid relation $s_i s_j \ldots = s_j s_i \ldots$.

For any pair of words $\boldsymbol{i}$ and $\boldsymbol{i}'$ joined in this graph, define a map $R_{\boldsymbol{i}}^{\boldsymbol{i}'} : \mathbb{N}^N \rightarrow \mathbb{N}^N$ taking $\mathbf{c} \mapsto \mathbf{c'}$ according to the braid relation between $\boldsymbol{i}$ and $\boldsymbol{i}'$. If $s_i s_j = s_j s_i$ is the required braid relation and $(a,b)$ and $(a',b')$ are the respective consecutive entries of $\mathbf{c}$ and $\mathbf{c}'$ at which $\boldsymbol{i}$ and $\boldsymbol{i}'$ differ, then $R_{\boldsymbol{i}}^{\boldsymbol{i}'}$ restricted to these entries is
 \begin{equation} (a',b') = (b,a). \label{lusztigcommutes} \end{equation} 
If instead $s_i s_j s_i = s_j s_i s_j$ is the required relation, let $(a,b,c)$ and $(a',b',c')$ be the respective consecutive entries of $\mathbf{c}$ and $\mathbf{c}'$ at which $\boldsymbol{i}$ and $\boldsymbol{i}'$ differ. Then $R_{\boldsymbol{i}}^{\boldsymbol{i}'}$ restricted to these entries is
\begin{equation} (a',b',c') = (b+c - \min(a,c), \min(a,c), a+b-\min(a,c)). \label{lusztigbraid} \end{equation}
For the maps $R_{\boldsymbol{i}}^{\boldsymbol{i}'}$ corresponding to braid relations of type $B_2$ and $G_2$, see Section~3 of \cite{b-z-totpos} or Section~7 of \cite{mcnamara-duke}. For each, one can check that $R_{\boldsymbol{i}}^{\boldsymbol{i}'}$ is a bijection with inverse $R_{\boldsymbol{i}'}^{\boldsymbol{i}}$. The final graph structure is given by connecting $(\boldsymbol{i}, \mathbf{c})$ to $(\boldsymbol{i}', \mathbf{c}')$ by an edge if $\boldsymbol{i}$ and $\boldsymbol{i}'$ are adjacent in the preliminary graph, and $R_{\boldsymbol{i}}^{\boldsymbol{i}'}(\mathbf{c}) = \mathbf{c}'$. Let $\boldsymbol{X}$ denote the set of connected components of the resulting graph.

\begin{theorem}[Lusztig, \cite{lusztig-book}, Ch. 42] The connected components $\boldsymbol{X}$ are in canonical bijection with $\mathcal{B}$. Moreover, for any $\boldsymbol{i}$, the map from $\mathbb{Z}_{\geq 0}^N$ to $\boldsymbol{X}$ sending elements $\mathbf{c}$ to the connected component of $(\boldsymbol{i}, \mathbf{c})$ is a bijection.
\end{theorem}

The vectors $\mathbf{c}$ appearing in the above theorem are referred to as the ``$\boldsymbol{i}$-Lusztig data'' for $\mathcal{U}^+$. Let $\mathcal{B}^\vee$ be the canonical basis with respect to the upper triangular part of the quantized universal enveloping algebra of $G^\vee$, the dual group of $G$.

\subsection{Lusztig data, MV polytopes, and $S_{\boldsymbol{\ell}, \boldsymbol{m}}$\label{connectthree}}

In this section, we explain the connection between the $\boldsymbol{i}$-Lusztig data and the valuations $\ell_i$ of the $d_i$ appearing in $H(\boldsymbol{d}; \boldsymbol{t})$ of (\ref{mdsform}), using the theory of MV polytopes. Throughout this section, let $N$ continue to denote the length of the reduced word $\boldsymbol{i}$. Our first goal is to define terms in and explain consequences of the following result.

\begin{theorem}[Kamnitzer, \cite{kamnitzer} Thms. 7.1 and 7.2] There is a coweight preserving bijection between the set of stable MV polytopes and the canonical basis $\mathcal{B}^\vee$. Under this bijection, the $\boldsymbol{i}$-Lusztig data of $b$ in $\mathcal{B}^\vee$ is equal to the integer $N$-tuple of edge lengths $n_\bullet$ in the one-skeleton of the MV polytope at edges corresponding to $\boldsymbol{i}$. \label{kamtheorem}
\end{theorem}

To expand on this recall that in Section~4 of \cite{kamnitzer}, to a set of positive integers $n_\bullet$, the stable MV polytopes are ($Y$-orbits of Zariski closures of) certain subsets $A^{\boldsymbol{i}}(n_\bullet)$ of the affine Grassmannian $\mathcal{G} := G(\mathbb{C}[[t]]) \backslash G(\mathbb{C}((t))).$ Thus the aim of this section is twofold:
\begin{enumerate}
\item Define the sets $A^{\boldsymbol{i}}(n_\bullet)$ presented in \cite{kamnitzer}. For simplicity, we define them for reduced words $\boldsymbol{i}$ for the long element $w_0$ rather than $w^P$ for a maximal parabolic $P$. It is straightforward to adapt this to a relative version whenever the word $\boldsymbol{i}$ is compatible with the factorization $w_0 = w_M w^P$. 
\item Show that the computation of $n_\bullet$ corresponding to a unipotent element $u$ in $\mathcal{G}$ is formally identical to the computation of the valuations of the $d_i$ corresponding to a point $u$ in the big cell of the flag variety $P(\mathfrak{o}_S) \backslash G(\mathfrak{o}_S)$ in Corollary~\ref{corgencase}.
\end{enumerate}
In proving Statement (2), we must relate an Iwasawa decomposition for an element in $\mathcal{G}$ to a Bruhat decomposition for elements in the flag variety in Section~\ref{sectionMDT}; nevertheless, as we will show, the Steinberg commutation relations required in each are identical, and hence the result will follow. In view of  Theorem~\ref{kamtheorem}, Statement (2) then implies that the valuations of the $d_i$ corresponding to a reduced word $\boldsymbol{i}$ are equal to the $\boldsymbol{i}$-Lusztig data (see Theorem~\ref{matchldata} for the precise statement).

We now begin a systematic explanation of the terms in Theorem~\ref{kamtheorem}. The GGMS stratification of the affine Grassmannian $\mathcal{G}$ is given by the intersections of semi-infinite cells
$$ A(\lambda_\bullet) := \bigcap_{w \in W} S_w^{\lambda_w} \quad \text{where} \; S_w^\mu := t^\mu w U w^{-1}, $$
amd $\lambda_{\bullet} = (\lambda_w)_{w \in W}$ ranges over all $N$-tuples of coweights. The intersection is empty
unless $w^{-1} \lambda_v \geq w^{-1} \lambda_w$ for all $v, w \in W$. Here $\mu \geq \nu $ means that $\mu - \nu$ is
a non-negative linear combination of fundamental weights. See Section~5 of \cite{anderson-kogan} or 
Section~2.4 of \cite{kamnitzer} for details.

To any such collection of coweights $\lambda_\bullet$ satisfying the inequalities, define the corresponding pseudo-Weyl
polytope $P(\lambda_\bullet)$ to be the intersection of cones:
$$ P(\lambda_\bullet) := \bigcap_{w \in W} C_w^{\lambda_w}, \quad \text{where} \quad C_w^{\lambda_w} := \{ \alpha \in \mathfrak{t}_\mathbb{R} \; : \; w^{-1} (\alpha) \geq w^{-1} (\lambda_w) \}. $$
Here $\mathfrak{t}_\mathbb{R}$ denotes the real points of the Cartan subalgebra of $\mathfrak{g}$.
Further define the ``length'' of edges in the one-skeleton of $P(\lambda_\bullet)$ as follows. Any adjacent vertices in the polytope have Weyl group elements which differ by a simple reflection $s_i$. Then the length $n$ is given by
\begin{equation} \lambda_{w s_i} - \lambda_w = n \cdot w(\alpha_i^\vee). \label{lengthofedge} \end{equation}
Thus given a point $L$ in the affine Grassmannian $\mathcal{G}$, one may first determine the stratum $A(\lambda_\bullet)$ to which it belongs, compute the polytope $P(L) := P(\lambda_\bullet)$ corresponding to $L$ and extract the lengths $n_\bullet$ of the one-skeleton. Given a reduced word $\boldsymbol{i} = i_1 \cdots i_N$ for $w_0$, we focus
attention on the one-skeleton lengths for the edges between consecutive vertices in the set
\begin{equation} \lambda_\bullet^{\boldsymbol{i}} := \{ \lambda_e, \lambda_{s_{i_N}}, \lambda_{s_{i_{N}} s_{i_{N-1}}}, \ldots, \lambda_{w_0} \}. \label{lambdai} \end{equation}
These are the edges corresponding to $\boldsymbol{i}$ referred to in Theorem~\ref{kamtheorem}. With a view toward the theorem,
Kamnitzer calls these edge lengths the ``$\boldsymbol{i}$-Lusztig data for the pseudo-Weyl polytope.'' Kamnitzer's ordering is derived from a left-to-right reading
of the word. We have chosen a right-to-left reading in~(\ref{lambdai}) to match the notation of Section~\ref{sectionMDT}.

MV polytopes arise by executing this procedure for points of $\mathcal{G}$ in subsets of the form
$$ X(\lambda) := U \cap t^\lambda U^- = S_e^0 \cap S_{w_0}^\lambda \subset \mathcal{G} $$
for dominant coweights $\lambda$. Indeed, let $\boldsymbol{i}$ be a reduced word for $w_0$ and let $n_\bullet$ be
a set of $N$ natural numbers. Then let
$\mathcal{Q}^{\boldsymbol{i}}(n_\bullet; \lambda)$ denote the set of all pseudo-Weyl polytopes with one-skeleton lengths
$n_\bullet$ for edges between $\lambda_\bullet^{\boldsymbol{i}}$, with $\lambda_e = 0$ and $\lambda_{w_0} = \lambda$ for the fixed dominant 
coweight $\lambda$. Then
$$ A^{\boldsymbol{i}}(n_\bullet) := \{ u \in X(\lambda) \; : \; P(u) \in \mathcal{Q}^{\boldsymbol{i}}(n_\bullet; \lambda) \}. $$
In Section~4 of \cite{kamnitzer}, Kamnitzer proves that the closures of the $A^{\boldsymbol{i}}(n_\bullet)$ are precisely the MV cycles.
MV polytopes are the polytopes arising from distinguished elements of $A^{\boldsymbol{i}}(n_\bullet)$, but we are only concerned with the one-skeleton lengths 
for edges between $\lambda_\bullet^{\boldsymbol{i}}$ which are the same for any member of $A^{\boldsymbol{i}}(n_\bullet)$ by definition, and hence the lengths between
$\lambda_\bullet^{\boldsymbol{i}}$ for the MV polytope.

Finally, stable MV polytopes are the class of MV polytopes obtained from the $Y$ orbits of MV cycles induced from the action $\mu \in Y: L \mapsto L \cdot t^\mu$ on
$\mathcal{G}$, which then acts on the polytope by translation and so preserves the one-skeleton length data~$n_\bullet.$ With this, we've completed the explanation of
the objects in Theorem~\ref{kamtheorem}.

In the remainder of the section, we show that an algorithm for computing the integers $n_\bullet$ for an element $u \in X(\lambda) \subset \mathcal{G}$
is identical to the algorithm for computing the valuations of the $d_i$ in $\iota_{\gamma_i}(g_i)$ corresponding to $u \in P(\mathfrak{o}_S) \backslash G(\mathfrak{o}_S)$ 
in the bijection of Corollary~\ref{corgencase}. Thus the two sets of integers agree and, according to Theorem~\ref{kamtheorem}, the valuations of the $d_i$ for a fixed
prime in $\mathfrak{o}_S$ are the $\boldsymbol{i}$-Lusztig data.

Given $\boldsymbol{i} = i_1 \cdots i_N$, define the ordered set of 
positive roots as in Section~\ref{sectionMDT}:
\begin{equation} \gamma_1 = \alpha_{i_1}, \quad \gamma_2 = s_{i_1} (\alpha_{i_{2}}), \quad \ldots, \quad \gamma_N = s_{i_1} \cdots s_{i_{N-1}} (\alpha_{i_N}). \label{rootorderrev} \end{equation}
To compute the one-skeleton length data $n_\bullet$ for $u \in X(\lambda) \subset U \subset \mathcal{G}$, we must first determine the
values of $\lambda_w$ in~(\ref{lambdai}) for the GGMS stratum. Write
\begin{equation} u = e_{\gamma_{N}}(x_N) \cdots e_{\gamma_{1}}(x_1), \quad x_j \in \mathbb{C}((t)). \label{ufactagain} \end{equation}
By definition, $\lambda_{s_{i_N}}$ is the exponent $\lambda$ in the torus component of the equality
$$ u = t^{\lambda} s_{i_N} u' s_{i_N}^{-1}, $$
as points in $\mathcal{G}$. Thus $\lambda$ is easily extracted using the Iwasawa decomposition on $s_{i_N}^{-1} u s_{i_N}$. 
But using (\ref{ufactagain}),
\begin{equation} s_{i_N}^{-1} u s_{i_N} = e_{-\gamma_{N}}(x_N) s_{i_N}^{-1} e_{\gamma_{N-1}}(x_{N-1}) \cdots e_{\gamma_{1}}(x_1) s_{i_N} =  e_{-\alpha_{i_N}}(x_N) u'', \label{firststepiwasawa} \end{equation}
for some element $u'' \in U$. Here we have used the fact that $\gamma_N = \alpha_{i_N}$, which can be seen from applying $w_0$ to $\gamma_N$ in the form~(\ref{rootorderrev}). Thus we reduce to a rank one Iwasawa decomposition for $e_{-\gamma_{N}}(x_N)$. Write $x_N = y_N / z_N$ with $y_N, z_N \in \mathfrak{o} = \mathbb{C}[[t]]$ and coprime and find elements $a_N$ and $b_N$ in $\mathfrak{o}$ so that $a_N z_N + b_N y_N = 1$. Then the resulting Iwasawa decomposition (for the embedded version of this rank one decomposition) is:
\begin{equation} \begin{pmatrix} 1 & \\ x_N & 1 \end{pmatrix} = \begin{pmatrix} z_N & -b_N \\ y_N & a_N \end{pmatrix} \begin{pmatrix} z_N^{-1} & \\  & z_N \end{pmatrix} \begin{pmatrix} 1 & z_N b_N \\  & 1 \end{pmatrix}. \label{rankoneiwasawa} \end{equation}
The rightmost two matrices were called $h(z_N^{-1}) e(z_N b_N)$ in Section~\ref{sectionMDT}. Thus the length of the one-skeleton segment connecting
$\lambda_{s_{i_N}}$ and $\lambda_e = 0$ is just $\ord_t(z_N)$, according to (\ref{lengthofedge}) with $w=e$.

To find $\lambda_{s_{i_N} s_{i_{N-1}}}$, we further conjugate by $s_{i_{N-1}}$ in (\ref{firststepiwasawa}), ignoring the embedded matrix from the right-hand side of (\ref{rankoneiwasawa}) in $G(\mathbb{C}[[t]])$, which is invariant under conjugation by the Weyl group, and moving the one parameter subgroup $e_{\gamma_{N-1}}(x_{N-1})$ in the Borel leftward. Thus as a first simple step,
\begin{multline} s_{i_{N-1}} h_{\gamma_{N}}(z_N^{-1}) e_{\gamma_{N}}(z_N b_N) s_{i_N} e_{\gamma_{N-1}}(x_{N-1}) \cdots e_{\gamma_{1}}(x_1) s_{i_N} s_{i_{N-1}} = \\ s_{i_{N-1}} s_{i_N} h_{\gamma_{N}}(z_N) e_{-\gamma_{N}}(z_N b_N) e_{\gamma_{N-1}}(x_{N-1}) \cdots e_{\gamma_{1}}(x_1) s_{i_N} s_{i_{N-1}}. \label{afterfirstreflection} \end{multline}
Then moving the $h_{\gamma_{N}}$ and $e_{-\gamma_{N}}$ rightward changes the arguments of the $e_{\gamma_i}$. Indeed, the right-hand side of (\ref{afterfirstreflection}) can be rewritten using the Steinberg relations (\ref{hecommutation}) and (\ref{eecommutation}) and some judicious regrouping as:
\begin{multline} [s_{i_{N-1}} s_{i_N} e_{\gamma_{N-1}}(x_{N-1}') s_{i_N} s_{i_{N-1}}] [s_{i_{N-1}} s_{i_N} h_{\gamma_{N}}(z_N) e_{-\gamma_{N}}(z_N b_N) s_{i_N} s_{i_{N-1}}] \\ [(s_{i_N} s_{i_{N-1}})^{-1} e_{\gamma_{N-2}}(x_{N-2}') \cdots e_{\gamma_{1}}(x_1') s_{i_N} s_{i_{N-1}}] = \\ 
e_{-\alpha_{i_{N-1}}}(x_{N-1}') [s_{i_{N-1}} s_{i_N} h_{\gamma_{N}}(z_N) e_{-\gamma_{N}}(z_N b_N) s_{i_N} s_{i_{N-1}}] u''', \label{steptwoiwasawa} \end{multline}
upon simplifying some of the bracketed terms on the left-hand side. Here $x_j'$ with $j \leq N-1$ are elements of $\mathbb{C}((t))$ and $u'''$ is an element of $U$. In the above, Lemma~\ref{convexorder} has been invoked to guarantee that the relations produce the above form. Then $s_{i_N} s_{i_{N-1}} \lambda_{s_{i_{N}} s_{i_{N-1}}} (s_{i_N} s_{i_{N-1}})^{-1}$ may be obtained by performing another rank one Iwasawa decomposition on $e_{-\alpha_{i_{N-1}}}(x_{N-1}')$ and extracting the torus component from the element in brackets on the right-hand side of~(\ref{steptwoiwasawa}). Indeed, the term in brackets (upon conjugating by $s_{i_{N}} s_{i_{N-1}}$) gives the inductively determined $\lambda_{s_{i_1}}$ and so the length $n_{s_{i_N} s_{i_{N-1}}}$ of the edge in the one-skeleton given by
$$ \lambda_{s_{i_N} s_{i_{N-1}}} - \lambda_{s_{i_N}} = n_{s_{i_N} s_{i_{N-1}}} s_{i_N} (\alpha_{i_{N-1}}^\vee) $$
is obtained solely from the torus component in the Iwasawa decomposition of $e_{-\alpha_{i_{N-1}}}(x_{N-1}')$. Using the above notation, this is just $\ord_t(z_{N-1}')$.

Continuing in this manner, the terms $\lambda_w$ in $\lambda_\bullet^{\boldsymbol{i}}$ may be inductively determined. The properties of the convex ordering from $\boldsymbol{i}$ 
in Lemma~\ref{convexorder} again guarantee that the algorithm of pushing $e_{-\gamma_j}$'s past $e_{\gamma_k}$'s terminates in finitely many steps.

We now compare these one-skeleton lengths $n_\bullet$ to the parametrization of elements of the flag variety in Section~\ref{sectionMDT}. Recall that we found representatives for elements in the big cell of the flag variety $P(\o) \backslash G(\o) \cap P(F) U_{-}^P(F)$ as a product of elements in embedded $SL_2(\o)$'s according to the reduced word $\boldsymbol{i}$. More precisely, Corollary~\ref{corgencase} established a bijection between bottom rows $(c_j, d_j)$ with $j \in [1,N]$ of the embedded $SL_2(\o)$ matrices with points in the big cell of the flag variety. These bottom row elements $(c_i, d_i)$ are in bijection with elements of $U_{-}^P(F)$ by Corollary~\ref{leftcosets}. In Theorem~\ref{mainthm}, we explained how to determine the pairs $(c_j, d_j)$ corresponding to an element $u \in U_{-}^P(F)$. The process of determining $d_j$ in the proof of Theorem~\ref{mainthm} is formally identical to the process of determining the ring element $z_{j}$ of $\mathbb{C}[[t]]$ appearing in the Iwasawa decomposition of $e_{-\alpha_{i_{j}}}$ at each stage above. As we are only interested in the valuation of the $d_i$ at a fixed prime $p$ in $\o_S$, we may pass to the localization and formally identify the set of possible valuations of the resulting uniformizers $p$ and $t$. Thus we arrive at the following result.

\begin{theorem} \label{matchldata} Given a reduced word $\boldsymbol{i}$, let $(c_j, d_j)$ for $j = 1, \ldots, N$ be the corresponding coordinates for points in the big cell of the flag variety as in Corollary~\ref{corgencase}. Then for a fixed prime $p \in \o$, $n_\bullet = (\ord_p(d_1), \ldots, \ord_p(d_N))$ is the $\boldsymbol{i}$-Lusztig data. In particular, if $\boldsymbol{i}'$ is any other reduced word, with coordinates $(c_j', d_j')$ with valuations of the $d_j'$ written as $n_\bullet'$, then $n_\bullet$ and $n_\bullet'$ are related by the Lusztig transition maps $R_{\boldsymbol{i}}^{\boldsymbol{i}'}$ as in~(\ref{lusztigcommutes}),~(\ref{lusztigbraid}), and Section~3 of \cite{b-z-totpos}. \end{theorem}

\subsection{String data for canonical bases\label{stringdata}}

While the moduli of the exponential sum $S_{\boldsymbol{\ell}, \boldsymbol{m}}$ are given by Lusztig data, the evaluation of the sum is more easily described in terms of an alternate basis
for the module $\mathcal{B}^\vee$ of $(\mathcal{U}^\vee)^+$ -- the so-called ``string bases''
of Kashiwara. Their definition is recalled now. We illustrate each general observation with a running example in $GL_4$; the conclusion
of an example is marked by $\triangle$.

Let $b$ be any vector in a canonical basis $\mathcal{B}_\lambda^\vee$ for a finite dimensional representation of the dual group
$G^\vee$ with highest weight $\lambda$. Explicitly, $\mathcal{B}_\lambda^\vee$ is the set of all $b \in \mathcal{B}^\vee$ such
that $b \cdot v_\lambda \ne 0$, where $v_\lambda$ is the highest weight vector in the module. Given a word
$\boldsymbol{i} = (i_1, \ldots, i_\nu)$ corresponding to a reduced decomposition of the 
long element $w_0$, let $n_1$ be the maximal integer such that
$$ b_1 := F_{i_1}^{(n_1)} (b) \ne 0 $$
where $F_i$ denotes the Kashiwara operator for the simple root $\alpha_i$. Now
let $n_2$ be the maximal integer such that $F_{i_2}^{(n_2)}(b_1) \ne 0$. Repeating
for each index, we obtain a sequence of integers $(n_1, \ldots, n_\nu)$ and the
resulting vector $b_\nu$ is the lowest weight vector in the representation. The 
integers $(n_1, \ldots, n_\nu)$ for each
vertex $b$ in $\mathcal{B}_\lambda^\vee$ will be referred to as ``string data.''
Let $\mathcal{C}_{\boldsymbol{i}}^\lambda$ be the set of all such $\nu$-tuples of integers
over all elements of $\mathcal{B}_\lambda^\vee$. 

The following is a special case of a result of Berenstein-Zelevinsky and Littelmann may be found, for example, in~Corollary~1 of Theorem~1.7 in Littelmann~\cite{littelmann}.

\begin{theorem}[Berenstein-Zelevinsky, Littelmann] The integer lattice points $\mathcal{C}_{\boldsymbol{i}}^\lambda$ as defined above determine a convex polytope in $\mathbb{R}^N$.
\end{theorem}

Thus we refer to the resulting polytope as the ``BZL polytope'' corresponding to reduced word $\boldsymbol{i}$ and dominant weight $\lambda$. Moreover, Littelmann showed (see Theorem~4.2 of~\cite{littelmann}) that if $\boldsymbol{i}$ is chosen with respect to a sequence of braidless maximal parabolics (see Definition~\ref{braidlessdef}), then the corresponding inequalities defining the polytope have a particularly simple form. More precisely, let $B \subset P_1 \subset P_2 \subset \cdots \subset G^\vee$ be a sequence of braidless maximal parabolics. Here we mean that $P_i$ is maximal in $P_{i+1}$. Let $W_i$ be the Weyl group generated by simple reflections in $P_i$ and ${}^iW_i$ the set of minimal length coset representatives in $W_{i-1} \backslash W_i$. Then we choose a reduced expression for
\begin{equation} w_0 = \underbrace{s_{i_1}}_{\tau_1} \underbrace{s_{i_2} \cdots}_{\tau_2} \; \cdots \; \underbrace{s_{i_j} \cdots s_{i_\nu}}_{\tau_r} \label{braidlessinduction} \end{equation}
such that $\tau_i$ is the longest word in ${}^iW_i$.

\medskip

\noindent {\bf Example:} For $GL_4$, consider the strictly dominant highest weight $\lambda + \rho$ where
$$ \lambda = m_1 \epsilon_1 + m_2 \epsilon_2 + m_3 \epsilon_3, \quad \rho = \epsilon_1 + \epsilon_2 + \epsilon_3, \quad \text{$\epsilon_i$ : fundamental weights, $m_i \geq 0$}.  $$ 
The reduced word $\boldsymbol{i} = (1,3,2,1,3,2)$ corresponds to such a sequence of braidless maximal parabolics as in (\ref{braidlessinduction}). The polytope inequalities for the string data $(n_1, \ldots, n_6)$ also respect the sequence of parabolics. Focusing on the string data $(n_3, n_4, n_5, n_6)$ corresponding to the maximal parabolic $P$ in $G^\vee$ with $\boldsymbol{i}^P = \{ 2, 1, 3, 2 \}$, we describe the resulting 4-dimensional polytope.

By Theorem~7.1 of \cite{littelmann}, the cone inequalities satisfied by all string data in the canonical basis $\mathcal{B}^\vee$ are
\begin{equation} n_3 \geq \max(n_4, n_5), \quad \min(n_4, n_5) \geq n_6 \geq 0 \label{coneineqs} \end{equation}

Then imposing the highest weight $\lambda+\rho$ as above, according to Section 7, Corollary 1 of \cite{littelmann}, the string data for the finite crystal $\mathcal{B}^\vee_{\lambda+\rho}$ satisfies the
highest weight inequalities:
\begin{align}
n_6 &\leq m_2+1 \label{nsix} \\
n_5 &\leq m_3+1+n_6 \label{nfive} \\
n_4 &\leq m_1+1+n_6 \label{nfour} \\
n_3&\leq m_2+1+n_4+n_5-2n_6 \label{nthree}.
\end{align}
The indexing above differs from~\cite{littelmann} because our string data is recorded with lowering operators, while~\cite{littelmann}
uses raising operators. \example

\medskip

Returning to the general case, it remains to describe the algorithm for mapping from $\boldsymbol{i}$-Lusztig data to 
string data. To do so, we must describe
how to apply the Kashiwara lowering operators $F_{i_j}$ to the basis element 
$E_{\boldsymbol{i}}^{\mathbf{c}}$ in Lusztig's basis. According to the form of (\ref{lbasiselt}),  
it is easy to determine the effect of $F_{i_1}$ on $E_{\boldsymbol{i}}^{\mathbf{c}}$. The integer
$c_1$ is maximal such that
$$ F_{i_1}^{(c_1)} E_{i_1}^{(c_1)} {T}_{i_1} (E_{i_2}^{(c_2)}) \cdots ({T}_{i_1} {T}_{i_2} \cdots {T}_{i_{\nu-1}})(E_{i_\nu}^{(c_\nu)}) \ne 0. $$
After applying $F_{i_1}$ to $E_{\boldsymbol{i}}^{\mathbf{c}}$ $c_1$ times, the resulting $\boldsymbol{i}$-Lusztig data is
$$ (0, c_2, \ldots, c_\nu) $$
In order to determine the effect of $F_{i_2}$ on the resulting basis element, we use a
sequence of braid relations to move from Lusztig data for the long word $\boldsymbol{i} = (i_1, i_2, \ldots)$
to Lusztig data for $\boldsymbol{i}' = (i_2, \ldots)$. This can be done since any long word is
related to any other by a sequence of braid relations and there exists a long word beginning
with any given simple reflection. For example, the transition maps for the case of $G^\vee$ simply laced
were given in~(\ref{lusztigcommutes}) and~(\ref{lusztigbraid}). Repeating this for each of the successive 
$i_j$ appearing in $\boldsymbol{i}$, then the algorithm terminates at the lowest weight vector and returns
string data for the corresponding basis vector $E_{\boldsymbol{i}}^{\mathbf{c}}$.

\medskip

\noindent {\bf Example:} $G = GL_4$ with long word $\boldsymbol{i} = (1,3,2,1,3,2)$. 

Beginning with an element $b$ with Lusztig data $(c_1, \ldots, c_6)$, then applying the $F_1$ operator 
$c_1$ times, we arrive at a basis vector $b_1$ with $\boldsymbol{i}$-Lusztig data $(0, c_2, \ldots, c_6).$
Then since the simple reflections $s_1$ and $s_3$ commute, the transition map~(\ref{lusztigcommutes}) is of the form
$$ R^{\boldsymbol{i}'}_{\boldsymbol{i}} : (0, c_2, c_3, \ldots, c_6) \mapsto (c_2, 0, c_3, \ldots, c_6) $$
where $\boldsymbol{i}' = (3,1,2,1,3,2)$. Then $F_3$ may be applied $c_2$ times to the above
vector resulting in a basis vector $b_2$ with Lusztig data $(0,0,c_3, \ldots, c_6)$. Now one must
move from $\boldsymbol{i}'$ to a word beginning with $2$ to apply the $F_2$ operator. This may
be accomplished by the sequence of moves:
$$ \boldsymbol{i}' \mapsto (3,2,1,2,3,2) \mapsto (3,2,1,3,2,3) \mapsto (3,2,3,1,2,3) \mapsto (2,3,2,1,2,3). $$
Keeping track of the effect on the Lusztig data, the first application of the braid relation maps
$$ (0,0,c_3,c_4, c_5, c_6) \mapsto (0, c_3 + c_4, 0, c_3, c_5, c_6) $$
according to (\ref{lusztigbraid}). Continuing in this manner, the resulting string data takes the form
\begin{equation} (n_1, n_2, n_3, n_4, n_5, n_6) = (c_1, c_2, c_4 + c_5 + \max(c_3, c_6), c_3 + c_5, c_3 + c_4, \min(c_3, c_6)).
\label{datamap} \end{equation}
\example

\medskip

In general, this map from the Lusztig data to string data respects parabolic induction.
That is, let $w_0 = w_M w^P$ be a factorization of the long element corresponding to 
a Levi decomposition of the parabolic $P = M U^P$. Let $\boldsymbol{i}$ be a reduced word for $w_0$, 
chosen so that
$$ \boldsymbol{i} = \boldsymbol{i}_M \cdot \boldsymbol{i}^P, \qquad \text{where ``$\cdot$'' denotes concatenation,} $$
and $ \boldsymbol{i}_M$ is a  long word for $M$.
Then under the map described above, 
the string data for $\boldsymbol{i}^P$ will only involve Lusztig data for the long word $\boldsymbol{i}^P$. 
Indeed, $M$ is a reductive group and so the algorithm, which is performed first for $\boldsymbol{i}_M$, 
invokes only braid relations within $\boldsymbol{i}_M$. After completing the algorithm for each of the simple
reflections in $w_M$, then we arrive at a basis element with Lusztig data beginning with a list of 0's at each 
position corresponding to $i_j$ in $\boldsymbol{i}_M$ and leaving the original Lusztig data
with index in $\boldsymbol{i}^P$ unchanged.

\medskip

\noindent {\bf Example:} $G = GL_4$ and $M = GL_2 \times GL_2$ then the word 
$$ (1,3,2,1,3,2) = (1,3) \cdot (2, 1, 3, 2)$$ 
is indeed compatible with the choice of parabolic with Levi factor $M$. 
Then with respect to $\boldsymbol{i}^P = (2, 1,3, 2)$, the map given in (\ref{datamap})
shows that the string data $(n_3, \ldots, n_6)$ can indeed be given using only the
Lusztig data $(c_3, \ldots, c_6)$.

Returning to Whittaker coefficients associated to this example, note that the $\boldsymbol{\ell}$ in the exponential sum 
$S_{\boldsymbol{\ell}, \boldsymbol{m}}$ of (\ref{slmdefined}) may be considered as attached to Lusztig data $\boldsymbol{c}$ (according to our
numbering scheme) as follows:
$$ \boldsymbol{c} = (c_1, c_2, c_3, c_4, c_5, c_6) = (0, 0, \ell_1, \ell_2, \ell_3, \ell_4). $$
And so this may, in turn, be mapped to string data (relative to the parabolic $P$) via 
(\ref{datamap}) with
\begin{equation} (n_3, n_4, n_5, n_6) = (\ell_2 + \ell_3 + \max(\ell_1, \ell_4), \ell_1 + \ell_3, \ell_1 + \ell_2, \min(\ell_1, \ell_4)) 
\label{elldata} \end{equation}

Identifying $\boldsymbol{m}$ with the weight $\lambda = m_1 \epsilon_1 + m_2 \epsilon_2 + m_3 \epsilon_3$, we write $\mathcal{B}_{\boldsymbol{m} + \rho}^\vee$ for the canonical basis elements in the corresponding highest weight module. The inequalities (\ref{nsix})--(\ref{nthree}) for string data in $\mathcal{B}_{\boldsymbol{m} + \rho}^\vee$ may now be rewritten in terms of Lusztig data using (\ref{elldata}) and a simple case analysis according to the value of $\min(\ell_1, \ell_4)$. Then the highest-weight inequalities are equivalent to:
\begin{align}
\ell_4&\leq m_2+1\label{hw1}\\
\ell_1+\ell_2&\leq m_3+1+\min(\ell_1,\ell_4)\label{hw2}\\
\ell_1+\ell_3&\leq m_1+1+\min(\ell_1,\ell_4).\label{hw3}
\end{align}

Thus, the above three inequalities (together with the obvious requirement that $\ell_i \geq 0$) cut out the finite set of $\boldsymbol{i}^P$-Lusztig data for the elements of $\mathcal{B}_{\boldsymbol{m}+\rho}^\vee$.
In Section~\ref{vanishing} below, we use these inequalities to describe the support of the exponential sum $S_{\boldsymbol{\ell}, \boldsymbol{m}}$. \example

\section{Generic evaluation of the exponential sum at prime powers\label{genericevalsection}}

In this section, given a reductive group $G$, any cominuscule parabolic subgroup $P$, and a fixed prime $p \in \mathfrak{o}_S$, we evaluate the sum $S_{\boldsymbol{\ell}, \boldsymbol{m}}$ defined in~(\ref{slmgeneral}) for many choices of $\boldsymbol{\ell}$ and $\boldsymbol{m}$.   
Recall that $\boldsymbol{\ell} = (\ell_1, \cdots, \ell_N)$ with $\ell_i = \text{ord}_p(d_i)$ and $\boldsymbol{m} = (m_1, \ldots, m_r)$ with $m_j = \text{ord}_p(t_j)$. We establish two results. The first states that, for certain hyperplane inequalities on $\boldsymbol{\ell}$ and $\boldsymbol{m}$, the exponential sum $S_{\boldsymbol{\ell}, \boldsymbol{m}}$ has a very regular evaluation given in terms of the Euler phi-function. The second states that for the reverse inequality, now with respect to $\boldsymbol{\ell}$ and $\boldsymbol{m} + \rho = (m_1+1, \ldots, m_r + 1)$, the exponential sum ``generically'' vanishes. More precisely, it vanishes for these $\boldsymbol{\ell}$ and $\boldsymbol{m}$ outside of a set of integer lattice points lying on 
certain hyperplanes in $r$-dimensional space. These are almost immediate consequences of the formula for the exponential sum as given in~(\ref{hdtdefined}), but we have postponed them to the present section to discuss relations between these hyperplane inequalities and canonical bases (see Remark~\ref{canbasesconnection} at the end of this section).

We begin by rewriting~(\ref{hdtdefined}) when $\boldsymbol{d}$ and $\boldsymbol{t}$ are prime powers as above, and obtain the general expression
\begin{equation} S_{\boldsymbol{\ell}, \boldsymbol{m}} = \sum_{c_j (\text{mod } p^{\text{ord}_p(D_j)})} \psi \left( \sum_j p^{m_j} v_j \right) \prod_{k=1}^N \left( \frac{c_k}{p^{\ell_k}} \right)^{\!\! q_k}.
\label{slmexplicit} \end{equation}
Recall that
\begin{equation} \text{ord}_p(D_j) = \ell_j +  \sum_{i = j+1}^N \langle \gamma_{j}, \gamma_{i}^\vee \rangle \ell_i
\quad \text{and} \quad
v_{j} = \sum_{(k,k') \in \mathcal{S}_j} \left[ \eta_{i,i',k,-k'}  b_k c_{k'} p^{-\text{ord}_p(D(k, k'; \alpha_j))} \right] \label{ordpbigd} \end{equation}
for each root simple root $\alpha_j \in P$ according to~(\ref{unipelt}), noting that by Proposition~\ref{isareone} the pairs $(i,i')$ used in linear combinations of roots for $\mathcal{S}_j$ are both 1 if $P$ is cominuscule. The element $D(k, k'; \alpha_j)$ is given explicitly in~Lemma \ref{twistedmultlemma}. If $\alpha_i$ is the unique simple root omitted from $P$, recall that it appears as $\gamma_N$ in our ordering and thus $\text{ord}_p(D_N) = \ell_N$ and $v_i = c_N p^{-\ell_N}$. 

We now observe that within certain limits on the $\ell_j$, the coefficients $S_{\boldsymbol{\ell}, \boldsymbol{m}}$ are supported on an $N$-dimensional rectangular lattice with side lengths $n(\gamma_j)$ where we define $n(\alpha) := n / \gcd(n, Q(\alpha^\vee))$ for any root $\alpha$. The evaluation of $S_{\boldsymbol{\ell}, \boldsymbol{m}}$ at these lattice points is given in terms of the Euler phi function on $p^r \mathfrak{o}_S$ for non-negative integers $r$, denoted $\phi(p^r)$.  

\begin{proposition} Suppose that for all simple roots $\alpha_j \in P$ and pairs $(k,k') \in \mathcal{S}_j$, we have $m_j \geq \text{ord}_p(D(k,k';\alpha_j))$, and that for the omitted root $\alpha_i$, $m_i \geq \ell_N$. Then
$$ S_{\boldsymbol{\ell}, \boldsymbol{m}} = \begin{cases} \displaystyle \prod_{k=1}^{N} |p|^{\epsilon_k} \phi(p^{\ell_k}) & \text{if $n({\gamma_k}) \, | \, \ell_k$ for all $k = 1, \ldots, N$} \\ 0 & \text{otherwise,} \end{cases}  $$
where
$$ \epsilon_k = \sum_{i = k+1}^N {\langle \gamma_{k}, \gamma_{i}^\vee \rangle} \ell_i. $$
\label{phifunprop} \end{proposition}

\begin{proof}
With these assumptions on $m_j$, the character $\psi$ appearing in~(\ref{slmexplicit}) is identically 1 in $S_{\boldsymbol{\ell}, \boldsymbol{m}}$.  The result follows immediately, since the multiplicative character is trivial exactly when 
$n({\gamma_k}) \, | \, \ell_k$ for all $k = 1, \ldots, N$.
\end{proof}

\begin{proposition} Suppose that at least one of the following conditions hold:
\begin{enumerate}
\item There exists a simple root $\alpha_j \in P$ and a pair $(k, k') \in \mathcal{S}_j$ with $\ell_k, \ell_{k'}>0$ such that the difference
$$\ell_{k,k';j}:= \text{ord}_p(D(k,k';\alpha_j)) - m_j  $$
is greater than one.  Moreover, this difference is larger than $\ell_{k_0, k_0'; j_0}$ for all simple roots $\alpha_{j_0} \in P$ and all pairs $(k_0, k_0')$ in $\mathcal{S}_{j_0}$
with one of $k_0, k_0'$ equal to $k'$.  
\item One has $\ell_N - m_i > 1$, where $\alpha_i$ is the simple root omitted from $P$, and this difference is larger than $\ell_{k_0, k_0'; j_0}$ for all simple roots $\alpha_{j_0} \in P$ and all pairs $(k_0, k_0')$ in $\mathcal{S}_{j_0}$
with one of $k_0, k_0'$ equal to $N$.
\end{enumerate}  
Then
$ S_{\boldsymbol{\ell}, \boldsymbol{m}} = 0. $
\label{genericvanprop} \end{proposition}

\begin{proof} Suppose first that $\ell_{k,k';j} > 1$ for some $j$ and a pair $(k, k') \in \mathcal{S}_j$ and satisfies the above maximality condition.
Then we may perform the changes of variables
$$ c_{k'} \longmapsto c_{k'} (1 + r p^{\ell_{k,k';j}-1}) $$
for each $r$ mod $p$. Summing the result, we obtain $|p|$ times the original sum $S_{\boldsymbol{\ell}, \boldsymbol{m}}$.
On the other hand, according to~(\ref{slmexplicit}), a factor of 
$$ \psi \left( \frac{\eta_{i,i',k,-k'} b_k c_{k'} r}{p} \right) $$
pulls out of each summand for each $r$ mod $p$. This is true since, in all other occurrences of $c_{k'}$ or $b_{k'}$ in $\psi$, the character $\psi$ is invariant under this substitution according to our maximality assumption and moreover, the power residue symbols are unchanged. Since $\ell_k$ and $\ell_{k'}$ are positive, the integers $b_k$ and $c_k'$ are coprime to $p$ and hence non-zero. Thus summing first over $r$, the sum $S_{\boldsymbol{\ell}, \boldsymbol{m}}$ vanishes.

If instead $\ell_N - m_i > 1$ where $\alpha_i$ is the omitted simple root from $P$ and the difference is larger than all $\ell_{k_0, k_0'; j_0}$ as above, then a similar change of variables $c_N \longmapsto c_N (1+ r p^{\ell_N - m_i - 1})$ causes $S_{\boldsymbol{\ell}, \boldsymbol{m}}$ to vanish by the same argument as in the previous case.
\end{proof}

\begin{remark} In Theorem~\ref{matchldata}, we explained that $\boldsymbol{\ell}$ gives the Lusztig data for the reduced word $\boldsymbol{i}^P$. We conjecture that the inequalities $m_i \geq \ell_N$ with $\alpha_i \not\in P$ and $m_j \geq \text{ord}_p(D(k,k';\alpha_j))$ for $\alpha_j \in P$ and $(k,k') \in \mathcal{S}_j$ used in Propositions~\ref{phifunprop} and~\ref{genericvanprop} are precisely the inequalities for $\boldsymbol{i}^P$-Lusztig data in the representation of $G^\vee$ with highest weight $\boldsymbol{m}$. Surprisingly, these do not appear to have been written down explicitly in the literature. In principle, they may be extracted from \cite{berenstein-zelevinsky} which gives a polytope realization of string data (\cite{berenstein-zelevinsky}, Theorem 3.10), proves that the string cones factor when the long word factors according to parabolic induction (\cite{berenstein-zelevinsky}, Theorem 3.11), and gives explicit bijections between string and Lusztig data via use of natural ``twist'' automorphisms (\cite{berenstein-zelevinsky}, Theorem 5.7). The results are stated for the canonical bases $\mathcal{B}^\vee$ but are of course compatible with highest weight modules by taking those $b \in \mathcal{B}^\vee$ such that $b \cdot v_{\text{high}} \ne 0$, where $v_{\text{high}}$ is the highest weight vector of the module. However, making these maps explicit in terms of the root datum 
would take us far afield, and we shall not carry it out here. 

Morier-Genoud \cite{morier-genoud} has applied the ideas of \cite{berenstein-zelevinsky} to study 
generalizations of the Sch\"utzen\-ber\-ger involution for highest weight modules of semisimple groups. Though not addressing the above question, many of the quantities obtained are similar to those of Section~\ref{sectionMDT}. For example, our formula for $\text{ord}_p(D_j)$ given in~(\ref{ordpbigd}) agrees with (2.12) in \cite{morier-genoud}.

Finally, we note that the relation between highest weight inequalities for Lusztig data and the hyperplanes $m_j \geq \text{ord}_p(D(k,k';\alpha_j))$ can be checked directly in our running $GL_4$ example, using~(\ref{hw1}),~(\ref{hw2}), and~(\ref{hw3}). Moreover, the results reviewed in Section~\ref{stringdata} provide an algorithm for directly checking this identity for any group $G$ and parabolic $P$.
\label{canbasesconnection} \end{remark}

In summary, we have evaluated the exponential sum $S_{\boldsymbol{\ell}, \boldsymbol{m}}$ for fixed $\boldsymbol{m}$ except for those $\boldsymbol{\ell}$ such that either $\text{ord}_p(D(k,k';\alpha_j)) = m_j + 1$ for some $\alpha_j \in P$ and $(k,k') \in \mathcal{S}_j$, or $\ell_N = m_i + 1$ with $\alpha_i \not\in P$, or else $\boldsymbol{\ell}$ lies on one of the exceptional hyperplanes defined by the equations
\begin{align*}
\ell_{k,k';j}&=\ell_{k_0,k_0',j_0}&&\alpha_j,\alpha_{j_0} \in P; (k,k'), (k_0,k_0')\in \mathcal{S}_{j_0}; k_0=k'~\text{or}~k_0'=k'\\
\ell_N - m_i &=\ell_{k_0, k_0'; j_0}&&\alpha_i \not\in P, \alpha_{j_0} \in P; (k_0, k_0')\in\mathcal{S}_{j_0}; k_0=k'~\text{or}~k_0'=N,
 \end{align*}
which moreover satisfy the maximality conditions in Proposition~\ref{genericvanprop}. To give a flavor of the complexity in these remaining cases, we carry out a full analysis in the $GL_4$ example.

\section{Theorems on the support of $\widetilde{GL}_4$ exponential sums\label{vanishing}}

To go beyond the results of the previous section requires a much finer analysis of the exponential sum $S_{\boldsymbol{\ell}, \boldsymbol{m}}$. We have done this for
our recurring $\widetilde{GL}_4$ example, and shown the result is supported on the union of the $\boldsymbol{i}_P$-Lusztig data for the elements of the finite crystal $\mathcal{B}_{\boldsymbol{m}+\rho}^\vee$ and an infinite collection of integer lattice points lying on a particular 2-dimensional hyperplane.

\begin{theorem}\label{support} The sum $S_{\boldsymbol{\ell}, \boldsymbol{m}}$ defined in (\ref{slmdefined}) is zero unless either
\begin{enumerate}[a)]
\item the highest weight inequalities (\ref{hw1}), (\ref{hw2}), (\ref{hw3}) hold, or
\item $\ell_1=\ell_4 \leq m_2+1$ and $\ell_2-m_3=\ell_3-m_1>1$.  
\end{enumerate}
\end{theorem}

When the exponential sum $S_{\boldsymbol{\ell}, \boldsymbol{m}}$ is non-zero, we give an alternate expression in terms of $n$-th order Gauss sums, broken into cases depending on inequalities for $\boldsymbol{\ell}$ and $\boldsymbol{m}$. For non-negative integers $m, \ell$, an integer $j$ mod $n$, and additive character $\psi$ on $F_S$ of conductor $\mathfrak{o}_S$ as before, we use the notation
$$ g_j(m,\ell) := \sum_{\substack{a \bmod p^\ell\\(a,p)=1}} \left( \frac{a}{p} \right)^{j\ell} \psi \left( \frac{p^m a}{p^\ell} \right), $$
using the $n$-th power residue symbol as in Section~\ref{sints}. Here we assume both the degree of the cover $n$ and the 
prime $p \in \mathfrak{o}_S$ to be fixed throughout. As a further shorthand, let
\begin{equation} g(k) := g_1(k-1,{k}), \qquad h(k) := g_1(k,k),\qquad g(m,\ell):=g_1(m,\ell). \label{gausssumshort} \end{equation}
This choice of notation is reasonable since our answers will be given uniformly for all $n$ and all $p \in \mathfrak{o}_S$ in terms
of such Gauss sums, though of course their explicit evaluation as complex numbers depends on this data. For example, 
$h(k)$ is a degenerate Gauss sum, equal to the Euler phi function $\phi(p^k)$ of $p^k \mathfrak{o}_S$ if $n$ divides $k$ and to $0$ otherwise. 
By contrast, $g(k)$ does not admit a simpler explicit description for $n > 2$ unless $n\mid k$.

Moreover, we can give the following explicit expression for $S_{\boldsymbol{\ell}, \boldsymbol{m}}$ in case $(b)$ of Theorem~\ref{support}.
\begin{proposition}
Suppose that $\ell_1=\ell_4 \leq m_2+1$ and $\ell_2-m_3=\ell_3-m_1>1$. With notation as in~(\ref{gausssumshort}),
\begin{equation} S_{\boldsymbol{\ell},\boldsymbol{m}} =  
q^{\ell_2+\ell_3+2\ell_4-k-\ell_1}g_0(m_2,\ell_4)\,h(2\ell_1+\ell_2+\ell_3)
 \label{excep-contrib-again} \end{equation}
with $k = \ell_2 - m_3= \ell_3 - m_1$ (and $g_0(m_2,0)=1$ by definition).
\label{excep-prop} \end{proposition}

Both the proof of Theorem~\ref{support} and Proposition~\ref{excep-prop} will be omitted here, but may be found in~\cite{bf-supplement}.
The proof of both results follows from a case analysis according to whether the Lusztig data $\ell_i$ is positive (in which case summands $c_i$ in
$S_{\boldsymbol{\ell},\boldsymbol{m}}$ are relatively prime to $p$, resulting in rather uniform expressions for changes of variables) or not.

\subsection{\label{vanishingforwhit}A vanishing theorem for Lusztig data outside a highest weight module}

The evaluation of the exponential sums $S_{\boldsymbol{\ell}, \boldsymbol{m}}$ for $\boldsymbol{\ell} = (\ell_1, \ldots, \ell_4)$ with $\ell_i = \ord_p(d_i)$
is finer than what is required to compute the Whittaker coefficient. Recall from (\ref{mdsform}) that the Whittaker coefficient is a Dirichlet series in $\mathfrak{D}^{\rho_P}$, expressible as a product of the $d_i$ according to Proposition~\ref{generalunipres}. In the context of our $\widetilde{GL}_4$ example, this is given explicitly in~(\ref{specialized}) where $\mathfrak{D}^{\rho_P} = |d_1 d_2 d_3 d_4|^2$.
Thus to assess the vanishing of the Dirichlet series at powers of a fixed prime $p$, we may sum over contributions of $H(d_1, d_2, d_3, d_4)$ with $d_i = p^{\ell_i}$ and such that $\ell_1 + \cdots + \ell_4$ is a fixed constant. Indeed it suffices to analyze $H$ since the remaining terms in the Whittaker coefficient in~(\ref{mdsform}) depend only on $\mathfrak{D}$ and not the individual $d_j$ (since the $S$-Hilbert symbol $(p,p)_S=1$).
The following result demonstrates that while the prime power contributions to $H$, written as $S_{\boldsymbol{\ell}, \boldsymbol{m}}$, may individually be non-zero for $\boldsymbol{\ell}$ not belonging to the module of highest weight $\boldsymbol{m}+\rho$ for $GL_4(\mathbb{C})$, the total contribution of all such sums to the Whittaker coefficient is indeed zero.

\begin{proposition}\label{cancel} Fix a positive integer $k$ and non-negative 3-tuple $\boldsymbol{m} = (m_1, m_2, m_3)$. Suppose that $(\ell_1, \ell_2, \ell_3, \ell_4)$ fail to satisfy at least one of the inequalities~(\ref{hw1})--(\ref{hw3}). Then
$$ \sum_{\substack{(\ell_1, \ell_2, \ell_3, \ell_4) \\ \ell_1 + \ell_2 + \ell_3 + \ell_4 = k}} S_{\boldsymbol{\ell}, \boldsymbol{m}} = 0. $$
\end{proposition}

\begin{proof} By Theorem~\ref{support}, the only vectors $(\ell_1, \ell_2, \ell_3, \ell_4)$ which fail one of~(\ref{hw1})--(\ref{hw3}) and have non-zero $S_{\boldsymbol{\ell}, \boldsymbol{m}}$ are those with
$$ 0 \leq \ell_1 = \ell_4 \leq m_2+1, \quad \ell_2 - m_3 = \ell_3 - m_1 > 1, $$
so that (together with the condition $\ell_1 + \ell_2 + \ell_3 + \ell_4 = k$) $\ell_1, \ell_2,$ and $\ell_3$ are uniquely determined by a choice of $\ell_4 \in [0, m_2+1]$.

For each choice of $\ell_4$, the summands may be evaluated using
equation (\ref{excep-contrib-again}).  They are all zero unless $n|k$.  In the nonzero case, they are evaluated as follows.
The sum over the range  $0 \leq \ell_4 < m_2+1$
gives
 $$  \sum_{0 \leq \ell_4 < m_2 + 1} \phi(p^{\ell_4}) \phi(p^{3\ell_1+2\ell_3+\ell_2+ m_3})=\phi(p^{3\ell_1+2\ell_3+\ell_2 + m_2+m_3}). $$
If $\ell_4 = m_2+1$,  then, using (\ref{excep-contrib-again}), the summand becomes 
$-\phi(p^{3\ell_1 + 2\ell_3+\ell_2+m_2+m_3})$. The sum of these terms is zero, as claimed. 
\end{proof}

\subsection{Expressions for $S_{\boldsymbol{\ell}, \boldsymbol{m}}$ inside the highest weight polytope\label{evaluation}}

We continue in the context of the $\widetilde{GL}_4$ example, with fixed non-negative integer 3-tuple $\boldsymbol{m} = (m_1, m_2, m_3)$ and $S_{\boldsymbol{\ell}, \boldsymbol{m}}$ as in~(\ref{slmdefined}). The Lusztig data $\boldsymbol{\ell}$ for canonical basis elements in $\mathcal{B}^\vee_{\boldsymbol{m}+\rho}$ is thus defined as the set of non-negative integer 4-tuples satisfying~(\ref{hw1})--(\ref{hw3}). In this section, we give an expression for $S_{\boldsymbol{\ell}, \boldsymbol{m}}$ for all $\boldsymbol{\ell}$ in $\mathcal{B}^\vee_{\boldsymbol{m}+\rho}$ in terms of $n$-th order Gauss sums, using a case method.

\subsubsection*{Evaluation for positive Lusztig data\label{subsectnondegen}}

Suppose that all $\ell_i>0$ and the highest weight inequalities (\ref{hw1}), (\ref{hw2}), (\ref{hw3}) hold.  There are four cases.

\medskip

\noindent {\bf Case 1:}  $\ell_1+\ell_3\leq m_1+\ell_4$, $\ell_1+\ell_2\leq m_3+\ell_4$.

In this case, the $\psi$ term is independent of $c_1$ and so summing over $c_1$ produces a factor of $q^{\ell_2+\ell_3}h(\ell_1)$.
Furtheer changing $c_2\mapsto c_2c_4$ and $c_3\mapsto c_3c_4$ gives
$$q^{\ell_2+\ell_3+2\ell_4}h(\ell_1)
\sum \left(\frac{c_2}{p^{\ell_2}}\right)\left(\frac{c_3}{p^{\ell_3}}\right)
\left(\frac{c_4}{p^{\ell_2+\ell_3+\ell_4}}\right)
\psi\left(-p^{m_1}\frac{c_3}{p^{\ell_3}}
+p^{m_2}\frac{c_4}{p^{\ell_4}}+p^{m_3}\frac{c_2}{p^{\ell_2}}\right).
$$
These sums are easily evaluated.  The value in this case is:
\begin{equation} S_{\boldsymbol{\ell},\boldsymbol{m}}=q^{2\ell_4} h(\ell_1)\,g(m_3,\ell_2)\,g(m_1,\ell_3)\,g(\ell_2+\ell_3+m_2,\ell_2+\ell_3+\ell_4). \label{case1eval} \end{equation}

\noindent {\bf Case 2:}  $\ell_1+\ell_3\leq m_1+\ell_4$, $\ell_1+\ell_2> m_3+\ell_4$.

By (\ref{hw2}), this can only happen when $\ell_4\leq \ell_1$ and $\ell_1+\ell_2=m_3+\ell_4+1$.
Also, since $\ell_1+\ell_3\leq m_1+\ell_4$ and $\ell_4\leq \ell_1$, it follows that $\ell_3\leq m_1$.
For given $c_2,c_3,c_4$, we begin with the $c_1$ sum.  This is
$$\sum_{c_1\bmod\!^\times p^{\ell_1+\ell_2+\ell_3}} 
 \left(\frac{c_1}{p^{\ell_1}}\right)\psi\left(\frac{b_3c_1}{p}\right).$$
Changing $c_1\mapsto c_1c_3$, this is easily evaluated.  We find that
$$
S_{\boldsymbol{\ell},\boldsymbol{m}}=q^{\ell_2+\ell_3+2\ell_4}g(\ell_1)
\sum_{c_2,c_3,c_4} \left(\frac{c_2}{p^{\ell_2}}\right)\left(\frac{c_3}{p^{\ell_1+\ell_3}}\right)
\left(\frac{c_4}{p^{\ell_4}}\right)\\
\psi\left(p^{m_2}\frac{c_4}{p^{\ell_4}}+p^{m_3}\frac{c_2b_4}{p^{\ell_2}}
\right).
$$
Now change $c_2\mapsto c_2c_4$.  The sum factors and each summand is
easily evaluated.   The value in this case is:
\begin{equation} S_{\boldsymbol{\ell},\boldsymbol{m}}=q^{\ell_3+2\ell_4-\ell_1}g(\ell_1)\,g(m_3,\ell_2)\,h(\ell_1+\ell_3)\,
g(m_2+\ell_2,\ell_2+\ell_4). \label{case2eval} \end{equation}

\noindent {\bf Case 3:}  $\ell_1+\ell_3> m_1+\ell_4$, $\ell_1+\ell_2\leq m_3+\ell_4$.

This case is symmetric to Case 2.  In view of (\ref{hw3}), it occurs only when $\ell_4\leq \ell_1$
and $\ell_1+\ell_3=m_1+\ell_4+1$.  Arguing as in Case 2, one sees that
$$S_{\boldsymbol{\ell},\boldsymbol{m}}=q^{\ell_2+2\ell_4-\ell_1}g(\ell_1)\,h(\ell_1+\ell_2)\,g(m_1,\ell_3)\,
g(m_2+\ell_3,\ell_3+\ell_4).$$

\noindent {\bf Case 4:}  $\ell_1+\ell_3> m_1+\ell_4$, $\ell_1+\ell_2> m_3+\ell_4$.

As before, the highest weight inequalities imply that $\ell_4\leq\ell_1$ and that we have:
$$\ell_1+\ell_2=m_3+\ell_4+1,\qquad \ell_1+\ell_3=m_1+\ell_4+1.$$
Note that this implies that $m_1+\ell_2=m_3+\ell_3$ and also that $\ell_2\leq m_3+1$, $\ell_3\leq m_1+1$.
We separate out two cases:

\medskip

\noindent {\it Case 4, Subcase A:}  In addition to the assumptions of Case 4, assume $\ell_4<\ell_1$.  

This additional condition implies that $\ell_2\leq m_3$ and $\ell_3\leq m_1$.  Thus the sum reduces to

$$
S_{\boldsymbol{\ell},\boldsymbol{m}}=q^{2\ell_4} \sum_{c_1,c_2,c_3,c_4} \left(\frac{c_1}{p^{\ell_1}}\right)\left(\frac{c_2}{p^{\ell_2}}\right)\left(\frac{c_3}{p^{\ell_3}}\right)
\left(\frac{c_4}{p^{\ell_4}}\right)
\psi\left(-\frac{b_2c_1}{p}
+p^{m_2}\frac{c_4}{p^{\ell_4}}+\frac{c_1b_3}{p}
\right).
$$
Changing $c_2\mapsto c_2c_1$ and $c_3\mapsto c_3c_1$, the sum is easily evaluated, and gives
\begin{equation} q^{2\ell_4} g(m_2,\ell_4) \,h(\ell_1+\ell_2+\ell_3) \,\bar g(\ell_2) \,\bar g(\ell_3). \label{case4aeval} \end{equation}
The conjugate Gauss sums could be removed by exploiting additional identities since the contribution is 0 unless $n$ divides $\ell_1+\ell_2+\ell_3.$

\medskip

\noindent {\it Case 4, Subcase B:}  In addition to Case 4, we have $\ell_4 = \ell_1$.  

Now we have $S_{\boldsymbol{\ell},\boldsymbol{m}}$ equal to
\begin{multline}
q^{2\ell_4} \sum_{c_1,c_2,c_3,c_4} \left(\frac{c_1}{p^{\ell_1}}\right)\left(\frac{c_2}{p^{\ell_2}}\right)\left(\frac{c_3}{p^{\ell_3}}\right)
\left(\frac{c_4}{p^{\ell_4}}\right) 
\psi\left(-\frac{b_2c_1 + b_4c_3}{p}
+p^{m_2}\frac{c_4}{p^{\ell_4}}+\frac{c_1b_3+c_2b_4}{p}
\right).
\end{multline}

The subsum with $c_2\equiv c_3\bmod p$ is easily evaluated as the $\psi$ function reduces to
$\psi(c_4 p^{m_2-\ell_4})$ and $(c_3/p^{\ell_3})=(c_2/p^{\ell_3})$.  This subsum gives
\begin{equation}\label{subsum1}
q^{\ell_2+\ell_3+2\ell_4-1}g(m_4,\ell_4)\, h(\ell_2+\ell_3)\, h(\ell_1).
\end{equation}

The remaining term is the subsum with $c_3=c_2-a$ where $(a,p)=1$. 
(Also $a\not\equiv c_2\bmod p$ as $(p,c_3)=1$.)   Since in that case,
$$b_3-b_2\equiv c_3^{-1}-c_2^{-1}=(c_2-c_3)c_2^{-1}c_3^{-1}\equiv ab_2b_3\bmod p,$$ this subsum reduces to:
$$q^{2\ell_4} \sum \left(\frac{c_1}{p^{\ell_1}}\right)\left(\frac{c_2}{p^{\ell_2}}\right)\left(\frac{c_3}{p^{\ell_3}}\right)
\left(\frac{c_4}{p^{\ell_4}}\right)
\psi\left(\frac{c_1ab_2b_3}{p}
+p^{m_2}\frac{c_4}{p^{\ell_4}}+\frac{ab_4}{p}
\right).
$$

After changing $c_1\mapsto c_1 c_2 c_3 a^{-1}$ (where
$a^{-1}$ is an inverse of $a$ modulo $p$)
and evaluating the $c_1$ sum, we arrive at
$$q^{\ell_2+\ell_3+2\ell_4}g(\ell_1)\sum_{c_2,a,c_4} \left(\frac{c_2}{p^{\ell_1+\ell_2}}\right)\left(\frac{c_2-a}{p^{\ell_1+\ell_3}}\right)
\left(\frac{c_4}{p^{\ell_4}}\right)\left(\frac{a}{p^{\ell_1}}\right)^{-1}
\psi\left(p^{m_2}\frac{c_4}{p^{\ell_4}}+\frac{ab_4}{p}
\right).
$$
Here we regard the sum as over $c_2$, $c_4$, and $a$.  Changing $c_2\mapsto ac_2$,
we get a sum over $c_2,a,c_4$ where now $c_2\not\equiv 1\bmod p$:
$$q^{\ell_2+\ell_3+2\ell_4}g(\ell_1)\sum_{c_2,a,c_4} \left(\frac{c_2}{p^{\ell_1+\ell_2}}\right)\left(\frac{c_2-1}{p^{\ell_1+\ell_3}}\right)
\left(\frac{c_4}{p^{\ell_4}}\right)\left(\frac{a}{p^{\ell_3}}\right)
\psi\left(p^{m_2}\frac{c_4}{p^{\ell_4}}+\frac{ab_4}{p}
\right).
$$
Now we may change $a\mapsto a c_4$ and evaluate the $a$ and $c_4$ sums.  This gives
$$q^{\ell_2+2\ell_4}g(\ell_1)\, g(\ell_3+m_2,\ell_3+\ell_4) \,g(\ell_3)
\sum_{\substack{c_2\bmod p^{\ell_2}\\c_2\not\equiv
0,1\bmod p}} \left(\frac{c_2}{p^{\ell_1+\ell_2}}\right)\left(\frac{c_2-1}{p^{\ell_1+\ell_3}}\right).
$$

The remaining sum in the
expression above is a Jacobi sum. If $\ell_1+\ell_2$, $\ell_1+\ell_3$ are  both zero modulo $n$
then the sum takes value $(q-2)q^{\ell_2-1}$.  If exactly one of $\ell_1+\ell_2$, $\ell_1+\ell_3$ 
is zero modulo $n$ the sum is $-q^{\ell_2-1}.$
If $\ell_1+\ell_2$, $\ell_1+\ell_3$ are both nonzero modulo $n$
but $2\ell_1+\ell_2+\ell_3\equiv0\bmod n$ then the
sum is of the form $J(\chi,\chi^{-1})$ with $\chi$ not the trivial character, and gives $-q^{\ell_2-1}$.  If 
$\ell_1+\ell_2$, $\ell_1+\ell_3$, $2\ell_1+\ell_2+\ell_3$ 
are all nonzero modulo $n$ then the sum is a quotient of Gauss sums
(times a power of $q$):
$$q^{\ell_2}\,\frac{g(\ell_1+\ell_2)\,g(\ell_1+\ell_3)}{g(2\ell_1+\ell_2+\ell_3)}.$$

Combining with the other subsums, we see that the full sum for Case 4, Subcase B equals:
\begin{multline}
q^{\ell_2+\ell_3+2\ell_4-1}g(m_4,\ell_4)\, h(\ell_2+\ell_3)\, h(\ell_1) +
q^{\ell_2+2\ell_4}g(\ell_1)\, g(\ell_3+m_2,\ell_3+\ell_4) \,g(\ell_3) \; \times\\
\begin{cases} (q-2)q^{\ell_2-1}&\text{if $\ell_1+\ell_2$, $\ell_1+\ell_3\equiv0\bmod n$}\\
-q^{\ell_2-1}&\text{if exactly one of $\ell_1+\ell_2$, $\ell_1+\ell_3$ is $\equiv 0\bmod n$}\\
-q^{\ell_2-1}&\text{if $\ell_1+\ell_2$, $\ell_1+\ell_3\not\equiv0\bmod n$, but $2\ell_1+\ell_2+\ell_3\equiv0\bmod n$} \\
q^{\ell_2}\,\frac{g(\ell_1+\ell_2)\,g(\ell_1+\ell_3)}{g(2\ell_1+\ell_2+\ell_3)}&\text{if $\ell_1+\ell_2$, $\ell_1+\ell_3$, $2\ell_1+\ell_2+\ell_3\not\equiv0\bmod n$.}\\
\end{cases}
\end{multline}

This last case demonstrates the complexity of $S_{\boldsymbol{\ell}, \boldsymbol{m}}$ for Lusztig data $\boldsymbol{\ell}$ lying
simultaneously on the maximal hyperplanes $\ell_2=m_3+1$ and $\ell_1=m_1+1$. 

\subsubsection*{Evaluation with at least one $\ell_i$ equal to 0\label{subsectdegen}}

We handle the remaining cases when the highest weight inequalities
(\ref{hw1}), (\ref{hw2}), (\ref{hw3}) are satisfied. 

First, suppose that $\ell_1=0$, $\ell_2\leq m_3+1$, $\ell_3\leq m_1+1$.

If, in addition, $\ell_4>0$, then the sum is evaluated as follows.  Note
that since $\ell_1=0$, the $c_1$ sum is modulo $p^{\ell_2+\ell_3}$ without a relative
primality condition, and there is no residue symbol in $c_1$.  
The function $\psi$ is independent of $c_1$ so the $c_1$ sum gives $p^{\ell_2+\ell_3}$.
After changing $c_i\mapsto c_ic_4$, $i=2,3$, the evaluation proceeds similarly to Case 2.
It gives:
\begin{equation} S_{\boldsymbol{\ell},\boldsymbol{m}}=q^{2\ell_4} g(m_3,\ell_2)\,g(m_1,\ell_3)\,
g(m_2+\ell_2+\ell_3,\ell_2+\ell_3+\ell_4). \label{degenl1iszero} \end{equation}
(Here $g(m,0)=1$ in case $\ell_2$ or $\ell_3$ is zero.)

Suppose instead that $\ell_1=\ell_4=0$.  If $\ell_2 = \ell_3 = 0$, then the sum is of course 1. Otherwise, we break 
the sum into two pieces.  First, if $c_1\equiv 0\bmod p$ then the character $\psi(\cdot)$ is identically 1,
so this piece gives $q^{\ell_2+\ell_3-1}h(\ell_2)\,h(\ell_3).$
Second, there is the sum over $c_1$ such that $(c_1,p)=1$.  Changing
$c_2\mapsto c_1c_2$, $c_3\mapsto c_1c_3$ we obtain
\begin{multline}
\sum \left(\frac{c_1}{p^{\ell_2+\ell_3}}\right)\left(\frac{c_2}{p^{\ell_2}}\right)\left(\frac{c_3}{p^{\ell_3}}\right)
\psi\left(-p^{m_1}\frac{b_2}{p^{\ell_3}}
+p^{m_3}\frac{b_3}{p^{\ell_2}}\right)\\
= h(\ell_2+\ell_3)
\begin{cases}
\bar g(m_1+\ell_2-\ell_3,\ell_2)\,\bar g(m_3+\ell_3-\ell_2,\ell_3)&
\text{if $\ell_2,\ell_3>0$}\\
h(\ell_3)&\text{if $\ell_2=0$, $\ell_3>0,$}\\
h(\ell_2)&\text{if $\ell_2>0$, $\ell_3=0.$}
\end{cases}
\end{multline}
(Note that by the highest weight inequalities, $\ell_2\leq m_3+1$ and $\ell_3\leq m_1+1,$
so that the first entries of the Gauss sums above are non-negative.)

Next, suppose that $\ell_1>0$ but $\ell_2=0$.
The sum becomes
$$q^{2\ell_4} \sum \left(\frac{c_1}{p^{\ell_1}}\right)\left(\frac{c_3}{p^{\ell_3}}\right)
\left(\frac{c_4}{p^{\ell_4}}\right)\\
\psi\left(-p^{m_1}\frac{b_4c_3}{p^{\ell_3}}
+p^{m_2}\frac{c_4}{p^{\ell_4}}+p^{m_3}\frac{c_1b_3p^{\ell_4}}{p^{\ell_1}}\right).
$$
If $\ell_3,\ell_4>0$, then changing $c_1\mapsto c_1c_3$ and then $c_3\mapsto c_3c_4$ gives
\begin{multline}
q^{2\ell_4} \sum \left(\frac{c_1}{p^{\ell_1}}\right)\left(\frac{c_3}{p^{\ell_1+\ell_3}}\right)
\left(\frac{c_4}{p^{\ell_1+\ell_3+\ell_4}}\right)
\psi\left(-p^{m_1}\frac{c_3}{p^{\ell_3}}
+p^{m_2}\frac{c_4}{p^{\ell_4}}+p^{m_3}\frac{c_1p^{\ell_4}}{p^{\ell_1}}\right)\\
=
q^{2\ell_4-2\ell_1}
g(m_3+\ell_4,\ell_1)\, g(\ell_1+m_1,\ell_1+\ell_3)\,g(\ell_1+\ell_3+m_2,
\ell_1+\ell_3+\ell_4).
\label{degenl2iszero1} \end{multline}
In the remaining cases, a similar argument gives
\begin{equation} \begin{cases}
q^{-\ell_1+\ell_3}
g(m_3,\ell_1)\, h(\ell_1+\ell_3)&\text{if $\ell_3>0$, $\ell_4=0$,}\\
q^{2\ell_4} g(m_2,\ell_4)\, h(\ell_1)&\text{if $\ell_3=0$, $\ell_4>0$,}\\
 h(\ell_1)&\text{if $\ell_3=\ell_4=0$.}
 \end{cases} \label{degenl2iszero2} \end{equation}

For a third possibility, suppose that $\ell_1,\ell_2>0$ but $\ell_3=0$.  The sum is
$$ q^{2\ell_4} \sum \left(\frac{c_1}{p^{\ell_1}}\right)\left(\frac{c_2}{p^{\ell_2}}\right)
\left(\frac{c_4}{p^{\ell_4}}\right)\\
\psi\left(-p^{m_1}\frac{b_2c_1p^{\ell_4}}{p^{\ell_1}}
+p^{m_2}\frac{c_4}{p^{\ell_4}}+p^{m_3}\frac{c_2b_4}{p^{\ell_2}}
\right).$$
Proceeding as above, this gives
$$\begin{cases}
q^{2\ell_4-2\ell_1}
g(m_1+\ell_4,\ell_1)\, g(\ell_1+m_3,\ell_1+\ell_2)\,g(\ell_1+\ell_2+m_2,
\ell_1+\ell_2+\ell_4)&\text{if $\ell_4>0,$}\\
q^{-\ell_1+\ell_2} g(m_1,\ell_1)\,h(\ell_1+\ell_2)&\text{if $\ell_4=0.$}
\end{cases}
$$

Last suppose that $\ell_1,\ell_2,\ell_3>0$ but $\ell_4=0$.  Upon changing 
$c_2\mapsto c_2c_1$ and $c_3\mapsto c_3c_1$ in $S_{\boldsymbol{\ell},\boldsymbol{m}}$, we obtain
$$
\sum \left(\frac{c_1}{p^{\ell_1+\ell_2+\ell_3}}\right)\left(\frac{c_2}{p^{\ell_2}}\right)
\left(\frac{c_3}{p^{\ell_3}}\right)
\psi\left(-p^{m_1}\frac{b_2}{p^{\ell_1+\ell_3}}
+p^{m_3}\frac{b_3}{p^{\ell_1+\ell_2}}
\right).$$
This gives
\begin{equation} h(\ell_1+\ell_2+\ell_3)\,\bar g(m_1+\ell_2-\ell_1-\ell_3,\ell_2)\,
\bar g(m_3+\ell_3-\ell_1-\ell_2,\ell_3). \label{degenl4iszero} \end{equation}
(Note that by the highest weight inequalities, $\ell_1+\ell_2\leq m_3+1$ and $\ell_1+\ell_3\leq m_1+1,$
so that the first entries of the Gauss sums above are non-negative.)

\subsection{BZL patterns and the exponential sums}

In Section~\ref{stringdata}, we presented Berenstein-Zelevinsky and Littelmann's results on string data: for any highest weight
representation of $G^\vee$, the string data for the corresponding canonical basis vectors are integer lattice points in a polytope.
In this section, we show that for our $\widetilde{GL}_4$ example, the evaluation of the exponential sum $S_{\boldsymbol{\ell}, \boldsymbol{m}}$ 
with Lusztig data $\boldsymbol{\ell} = (\ell_1, \ldots, \ell_4)$ can be described using the corresponding {\it string} data according 
to~(\ref{elldata}). To facilitate this description, we place the string data corresponding to $\boldsymbol{\ell}$ in a simple one-row array:
\begin{equation}
\fbox{\strut $\max(\ell_1,\ell_4)+\ell_2+\ell_3$}\fbox{\strut $\ell_1+\ell_3$}\fbox{\strut $\ell_1+\ell_2$}\fbox{
\strut $\min(\ell_1,\ell_4)$}. \label{bzlforlusztig}
\end{equation}

We refer to the resulting array as a ``BZL pattern'' and use the notation $\mathcal{P}_{\boldsymbol{\ell}}$. Arrays with multiple rows appear in \cite{littelmann} corresponding to a sequence of relatively maximal parabolics from $G^\vee$ down to the Borel subgroup. These played a prominent role in our earlier investigations of metaplectic
Eisenstein series in type $A$ with Daniel Bump (\cite{eisenxtal, bbf-book}) and in proofs and conjectures of other types (e.g., \cite{beibrufre-crystal, chinta-gunnells-d, friedberg-zhang,
friedberg-zhang2}).
A product of Gauss sums is attached to any such BZL pattern as follows.

Each integer entry $c$ of a BZL pattern is constrained by inequalities depending on the other entries and the highest weight encoded by $\boldsymbol{m}$.
For example, the entry $c = \min(\ell_1, \ell_4)$ above is constrained by $0 \leq c \leq m_2+1$ according to~(\ref{coneineqs}) and~(\ref{nsix}). We refer to
the element $c$ as being ``maximal'' if the upper inequality is an equality and ``minimal'' if the lower inequality is an equality.
Then define the function
$$ \gamma(a) = \begin{cases} g(a), & \text{if $a$ is maximal,} \\ h(a), & \text{if $a$ is neither maximal nor minimal,} \\ q^a, & \text{if $a$ is minimal,} \\ 
0, & \text{if $a$ is both maximal and minimal.} \end{cases} $$

Given a BZL pattern $\mathcal{P}_{\boldsymbol{\ell}}$, we then define the {\sl standard contribution}, $G(\mathcal{P}_{\boldsymbol{\ell}})$, as follows:
$$ G(\mathcal{P}_{\boldsymbol{\ell}}) = \prod_{a \in \mathcal{P}_{\boldsymbol{\ell}}} \gamma(a). $$
In \cite{eisenxtal}, the standard contribution exactly matched the exponential sum for covers of $G= GL_{r+1}$ and maximal parabolic omitting the root $\alpha_1$. The following result shows that the standard contribution for $\mathcal{P}_{\boldsymbol{\ell}}$ ``generically'' agrees with the evaluation of the 
prime-powered exponential sum $S_{\boldsymbol{\ell}, \boldsymbol{m}}$, though comparison with the results of Section~\ref{evaluation} show it fails to agree in general.

\begin{theorem} Let $\mathcal{P}_{\boldsymbol{\ell}}$ be the BZL pattern corresponding to $\boldsymbol{\ell}$ as in (\ref{bzlforlusztig}). Then
$$ q^{2\ell_1-2\ell_4} S_{\boldsymbol{\ell}, \boldsymbol{m}} = G(\mathcal{P}_{\boldsymbol{\ell}}), $$
if $\boldsymbol{\ell} = (\ell_1, \ldots, \ell_4)$ avoids the following cases:
\begin{itemize}
\item $\ell_1 = \ell_4 = m_2 + 1,$
\item $\ell_1 = m_2+1$ and $\ell_4 = 0,$
\item $\ell_2 = m_3+1, \ell_3 = m_1+1$, and $\ell_1 = \ell_4$,
\item $\ell_1 = \ell_4 = 0,$ and
\item $\ell_2 = \ell_3 = 0.$
\end{itemize}
\label{stringmatching} \end{theorem}

In fact, the theorem holds for some Lusztig data $\boldsymbol{\ell}$ in the excluded cases, but because
it is difficult to state succinctly and in some cases depends on the divisibility of the $\ell_i$ by $n$, we 
omit the precise description of where the standard contribution fails. This information may be 
extracted from the evaluations of the previous section. The proof of Theorem~\ref{stringmatching} is a straightforward application of the
results of the previous section on a case-by-case basis and is left to the reader.

\bibliographystyle{acm}
\bibliography{maxparabolic}

\end{document}